\documentclass{amsart}
\usepackage{color}
\usepackage{graphicx}
\usepackage[colorlinks=true,linkcolor=blue,citecolor=blue]{hyperref}
\usepackage{pstricks}

\newtheorem{theorem}{Theorem}[section]
\newtheorem{lemma}[theorem]{Lemma}
\newtheorem{prop}[theorem]{Proposition}

\newtheorem{definition}[theorem]{Definition}

\newtheorem*{theorem*}{Theorem}

\theoremstyle{definition}
\newtheorem{remark}[theorem]{Remark}

\newcommand{\inn}{{\quad\hbox{in } }}

\newcommand{\nn}{ {\nabla}  }

\newcommand{\pp}{ {\partial} }

\newcommand{\cuad}{{\sqcap\kern-.68em\sqcup}}

\newcommand{\ve}{\varepsilon}

\newcommand{\be}{\begin{equation}}
\newcommand{\ee}{\end{equation}}

\newcommand{\equ}[1]{(\ref{#1})}

\numberwithin{equation}{section}

\newcommand{\R}{\mathbb{R}}

\newcommand{\Z}{\mathbb{Z}}
\newcommand{\N}{\mathbb{N}}

\newcommand{\lcal}{\mathcal{L}}

\newcommand{\pd}[2]{\frac{\partial #1}{\partial #2}}
\def\p#1,#2,#3 {\rlap{\kern#1pt\raise#2pt\hbox{#3}}}

\numberwithin{equation}{section}

\title[]{finite topology self-translating surfaces for the mean curvature flow in $\R^3$}

\author{Juan D\'avila}
\address{\noindent J. D\'avila -
Departamento de Ingenier\'{\i}a Matem\'atica and CMM, Universidad
de Chile, Casilla 170 Correo 3, Santiago, Chile.}
\email{jdavila@dim.uchile.cl}

\author{Manuel del Pino}
\address{\noindent M. del Pino - Departamento de
Ingenier\'{\i}a  Matem\'atica and CMM, Universidad de Chile, Casilla
170 Correo 3, Santiago, Chile.} \email{delpino@dim.uchile.cl}

\author{Xuan Hien Nguyen}
\address{\noindent X.H. Nguyen - 
Department of Mathematics, Iowa State University, Ames, IA 50011, U.S.A.}
\email{xhnguyen@iastate.edu}

\begin{document}

\begin{abstract}
	Finite topology self translating surfaces to mean curvature flow of surfaces constitute a key element for the analysis of Type II singularities from a compact surface, since they arise in a limit after suitable blow-up scalings around the singularity. We find in $\R^3$ a surface $M$ orientable, embedded and complete with finite topology (and large genus) with three ends
	asymptotically paraboloidal, such that the moving surface $\Sigma(t) = M + te_z$ evolves by mean curvature flow. This amounts to the equation
	$H_M = \nu\cdot e_z$ where $H_M$ denotes mean curvature, $\nu$ is a choice of unit normal to $M$, and $e_z$ is a unit vector along the $z$-axis.
	The surface $M$ is in correspondence with the classical 3-end Costa-Hoffmann-Meeks minimal surface with large genus, which has two asymptotically catenoidal ends and one planar end, and a long
	array of small tunnels in the intersection region resembling a periodic Scherk surface.
	  This example is the first non-trivial one of its kind, and it suggests a strong connection between this problem and the theory of embedded, complete minimal surfaces with finite total curvature.
	
\end{abstract}

\maketitle

\section{Introduction}

We say that a family of orientable, embedded hypersurfaces $\Sigma(t)$ in $\R^{n+1}$
evolves by mean curvature if each point of $\Sigma(t)$ moves in the normal direction
with a velocity proportional to its mean curvature at that point.
More precisely, there is a smooth family of  diffeomorphisms  $Y(\cdot, t) :\Sigma(0) \to \Sigma(t)$, $t>0$,
determined by the mean curvature flow equation

\be
\frac {\pp Y}{\pp t}   = H_{\Sigma(t)} (Y) \nu( Y)
\label{mcf}\ee
where
$H_{\Sigma(t)} (Y) $ 
  designates the mean curvature of the surface $\Sigma(t)$ at the point $Y(y, t)$, $y\in \Sigma(0)$, namely the trace of its second fundamental form, $\nu$ is a choice of unit normal vector.

\medskip
 The mean curvature flow  is one of the most important examples of parabolic geometric evolution of manifolds. Relatively simple in form, it generates a  wealth of interesting phenomena which is so far only partly understood.
 Extensive, deep  studies on the properties of this equation have been performed in the last 25 years or so. We refer for instance the reader to the surveys \cite{cm1} and \cite{wang0}.

\medskip
 A classical,  global-in-time definition of a weak solution to mean curvature flow is due to Brakke. These solutions typically develop finite time singularities. When they arise, the evolving manifold
 loses smoothness, and a change of topology of the surface may occur as the singular time is crossed.

\medskip
 The basic issue of the theory for the mean curvature flow, is to understand the way singularities appear and to achieve an accurate description of the topology of the surface obtained after blowing-up the manifold around the singularity.

 Singuarities are usually classified as types I and II. If $T$ is a time when a singularity
 appears, type I  roughly means that the curvatures grow no faster than $(T-t)^{-\frac 12}$. In such a case, a blowing-up procedure, involving a time dependent scaling and translation leads in the limit to a ``self shrinking'' {\rm ancient solution}, as established by Huisken in \cite{huisken}. The appearance of these singularities turns out to be generic under suitable assumptions, see Colding and Minicozzi \cite{cm2}.

\medskip 
Instead, if the singularity is not of type I, it is called type II. In that case, a suitable normalization leads in the limit to an {\em eternal solution} to MC flow.
 See Colding and Mincozzi \cite{cm}, Huisken and Sinestrari \cite{hs1,hs2}.
An eternal solution to \equ{mcf} is one that is defined at all times $t\in (-\infty, \infty)$.

\medskip
The simplest type of eternal solutions are the {\em self-translating solutions}, surfaces that solve \equ{mcf}, do not change shape and travel at constant speed in some specific direction.
A { self-translating} solution of the mean curvature flow \equ{mcf}, with speed $c>0$ and direction ${\tt e}\in S^{N-1}$ is a hypersurface  of the form
\be\label{st}
\Sigma (t)\, =\,  c t{\tt e} \, +\, \Sigma (0)\, ,\ee
that satisfies \equ{mcf}. Equivalently, such that
\be \label{st1}
H_{\Sigma(0)}  =  c  \tt e\cdot\nu.
\ee
This problem is nothing but the minimal surface equation in case $c=0$.
A result by Hamilton \cite{hamilton} states that in the case of a compact  convex surface,  the limiting scaled singularity does indeed take place in the form of a self-translating solution. This fact makes apparent  the importance of eternal self-translating solutions
in the understanding of singularity formation.
 On the other hand,  the result in \cite{hamilton}
   is not known without some convexity assumptions. An open, challenging issue is to understand whether or not a given ``self-translator'' (convex or non-convex) can arise as a limit of a type II singularity for \equ{mcf}.

 \medskip
 A situation in which strong insight has been obtained is the mean convex scenario (namely, surfaces with non-negative mean curvature, a property that is preserved under the flow). In fact under quite general assumptions, mean convexity in the singular limit becomes full convexity for the blown-up surface, as it has been established by  B. White \cite{white2,white3}, and by Huisken and Sinestrari \cite{hs1,hs2}.

 \bigskip
 In spite of their importance in the theory for the mean curvature flow,  relatively few examples of self-translating solutions are known, and a theory for their understanding, even in special classes is still far from achieved. Since for $c=0$, equation \equ{st1} reduces to the minimal surface equation, it is natural to look for analogies with minimal surface theory in order to obtain new nontrivial examples.
 On the other hand, it  has been proven by Ilmanen \cite{ilmanen1,ilmanen2}, that the genus
 of a surface is nonincreasing along the mean curvature flow. Therefore, self-translators originated from a singularity in the flow of a compact surface must have finite genus, or {\bf finite topology}.

 \bigskip
 The purpose of this paper is to construct new examples of self-translating surfaces to the mean curvature flow with the finite topology  in $\R^3$. More precisely, we are interested in tracing a parallel between the theory
 of embedded, complete minimal surfaces in with finite total Gauss curvature (which are precisely those with finite topology) and self-translators with positive speed. Before stating our main
 result, we recall next some classical examples of self-translators.

 \medskip
 If $\Sigma(t) = \Sigma(0)  +  cte_{n+1} $ is a {\bf travelling graph }, namely
 $$\Sigma (0) = \{ (x,x_{n+1})\ /\ x_{n+1} = F(x) , \ x\in\Omega\subset \R^n  \} $$ then equation \equ{st1}
 reduces to the elliptic PDE for $F$,
   \be
   \label{mc3}
   \ \nn \cdot \left ( \frac {\nn F } { \sqrt{ 1+
   		|\nn F|^2 }} \right ) = \frac{c}{\sqrt{1+|\nabla F|^2}}  \inn \Omega \subset \R^n.
   \ee
  For instance for $n=1$ and $c=1$, Grayson \cite{grayson} found the explicit solution, so-called the {\em grim reaper} curve $\mathcal G$ , given by the graph
  \be\label{G}  x_2 = F(x_1)= -\log (\cos x_1), \quad x_1\in (-\frac \pi 2, \frac \pi 2). \ee
In other words, $\Sigma(t) =  \mathcal G + te_2 $ solves \equ{mcf}.

\medskip
   For dimensions $n\ge 2$, there exist entire convex solutions to Equation \equ{mc3}.
   Altschuler and Wu \cite{altschuler-wu} found a radially symmetric convex solution $F(|x|)$  to \equ{mc3} by blowing-up a type II singularity of mean curvature flow. This solution can be explicitly found by solving the radial PDE \equ{mc3} which becomes simply
   \begin{align}
   	\label{eqFrad1}
   	\frac{F''}{1+(F')^2} +    (n-1)\frac  {F'}r = c .
   \end{align}
   See  \cite{angenent-velazquez} and \cite{clutterbuck-schnurer-schulze}.
     The resulting surface is asymptotically a paraboloid: at main order, when $c=1$, it has the
 behavior
 \be
 F(r)= \frac{r^2}{2(n-1)} - \log r + O(r^{-1}) \quad\hbox{as } r \to +\infty.
 \label{g} \ee
 We shall denote by  $\mathcal P$ the graph of this entire graphical self-translator (which is unique up to an additive constant) which we shall refer to as  the {\bf travelling paraboloid}.  Of course this means that $\Sigma(t) = \mathcal P + te_{n+1}$ solves \equ{mcf}.

 \medskip
 Xu-Jia Wang \cite{wang} proved that for $n=2$, solutions of \equ{mc3} are necessarily radially symmetric about some point, thus in particular they are convex.  Suprprisingly, for dimensions $n\ge 3$,  Wang  was able to construct nonradial convex solutions of \equ{mc3}.

\medskip
In dimension $n+1$, $n\ge 2$, Angenent and Vel\'azquez \cite{angenent-velazquez} constructed an axially symmetric solution to \equ{mcf} that develops a  type II singularity with a tip that blows-up precisely
into the paraboloid   $\mathcal P$.
On the other hand, B. White proved that the convex surface in $\R^{n+1}$ given by the $\mathcal G \times \R^{n-1}$ where $\mathcal G$ is the grim reaper curve \equ{G}, cannot arise as a blow-up of a type II singularity for \equ{mcf}.

 \medskip
 A non-graphical, two-end axially symmetric self translating solutions of \equ{mcf} for $n\ge 2$ has been found by direct integration of the radial PDE \equ{mc3} by Clutterbuck, Schnurer and Schulze   \cite{clutterbuck-schnurer-schulze}. It can be described as as follows:

\medskip
 Given any number $R>0$, there is a self-translating solution of \equ{mcf}
 $$ \Sigma(t) = \mathcal{W} + t e_{n+1} $$
  where $\mathcal{W}$ is a two-end  smooth surface of revolution of the form
 $$\mathcal{W} = \mathcal{W}^+ \cup \mathcal{W}^-, \quad \mathcal{W}^{\pm} = \{(x,x_{n+1}) \ /\  x_{n+1} = F^\pm (|x|)\,   \}   $$
    where the functions
    $F^\pm (r)$  solve \equ{eqFrad1} for
    $c=1$ and $r>R$, with  $F^-(r) < F^+(r)$ and  $F^+(R)=F^-(R)$.
    It is shown in \cite{clutterbuck-schnurer-schulze} that the functions $F^\pm$ have the asymptotic behavior \eqref{g} of $\mathcal P$ up to an additive constant. See Figure \ref{fig1}.
 We  call the two-end translating surface $\mathcal W$ the {\bf travelling catenoid}.
The reason is  natural: when $c=0$ equation \equ{eqFrad1} is nothing but the minimal surface
equation for an axially symmetric minimal surface around the $x_{n+1}$-axis.
When $n=3$ the equation leads (up to translations) to the plane
$
x_3 = 0
$, or the standard catenoid $ r= \cosh (x_3)$. The catenoid is exactly the parallel to $\mathcal W$.
The plane is actually in correspondence with the paraboloid $\mathcal P$.

\medskip
 These simple, however important examples, are the only ones available with finite topology.
 On the other hand,
 the third author  has constructed self translating surfaces with infinite topology, periodic in one direction in \cite{nguyen-2009,nguyen, nguyen-2014}.

\subsection*{Embedded minimal surfaces of finite total curvature in $\R^3$.} The
theory of embedded, minimal surfaces of finite total curvature in $\R^3$
has reached a spectacular development in the last 30 years or so. For about
two centuries, only two examples of such surfaces were known: the plane
and the catenoid. The first nontrivial example was found in 1981 by C.
Costa \cite{4,5}. The Costa surface is a genus one minimal surface, complete
and properly embedded, which outside a large ball has exactly
three components (its ends), two of which are asymptotically catenoids
with the same axis and opposite directions, the third one asymptotic
to a plane perpendicular to that axis.  Hoffman and Meeks \cite{20,hoffmann-meeks-1990-limits, 24}  built a class of three-end,
embedded minimal surfaces, with the same look as Costa'Âs far away, but
with an array of tunnels that provides arbitrary genus $k$. These are
known as the Costa-Hoffman-Meeks surfaces, see Figure \ref{chm}.
Many other examples of multiple-end embedded minimal surfaces
have been found since.


\medskip
All surfaces of this kind  are constituted, away from a compact region, by the disjoint union of  ends
ordered along one coordinate axis, which are asymptotic to planes or to catenoids with  parallel symmetry axes, as established by  Osserman \cite{35}, Schoen \cite{39} and Jorge and Meeks \cite{28}. The topology of such a  surface is thus characterized by the genus of a compact region and the number
of ends, having therefore finite topology.

\medskip
\subsection*{Main result: the travelling CHM surface of large genus}
In what follows we restrict ourselves to the the case $n+1=3$.

Our purpose is
to construct new complete and embedded surfaces in $\R^3$ which are self translating under mean curvature flow. After
a rotation and dilation we can assume that  $c=1$ and that the travelling direction is the
that of the positive $x_3$-axis. Thus we look for  orientable, embedded complete surfaces $M$ in $\R^3$ satisfying the equation
\begin{align}
\label{selftr-mcf}
H_M = e_z \cdot \nu.
\end{align}
where $ e_z=e_3$. In other words, the moving surface $\Sigma (t) = M + te_z$ satisfies equation \equ{mcf}.
 A major difficulty to extend the theory of finite total curvature minimal surfaces in Euclidean 3d space to equation
\equ{selftr-mcf} is that much of the theory developed relies in the powerful tool given by the Weierstrass representation formula, which is not available in our setting. Unlike the static case,
the travelling catenoid for instance is not asymptotically flat and does not have total finite total Gauss curvature.

\medskip
What we establish in our main result is the existence of a 3-end surface $M$ that solves
\equ{selftr-mcf}, homeomorphic to a Costa-Hoffmann-Meeks surface with large genus, whose ends
behave like those of a travelling catenoid and a paraboloid.

\medskip
More precisely, let us consider the union of a travelling paraboloid
$\mathcal P$ and and a travelling catenoid $\mathcal W$, which intersect transversally on a circle
$C_\rho$ for some $\rho>0$. See Figure \ref{fig1}.

  Our surface looks outside a compact set like $\mathcal P \cup \mathcal W $   in Figure \ref{fig1}, while near the circle $C_\rho$ the look is that of the static CHM surface in Figure \ref{chm}.

\medskip

\begin{figure}
\centerline{\includegraphics[scale=0.5]{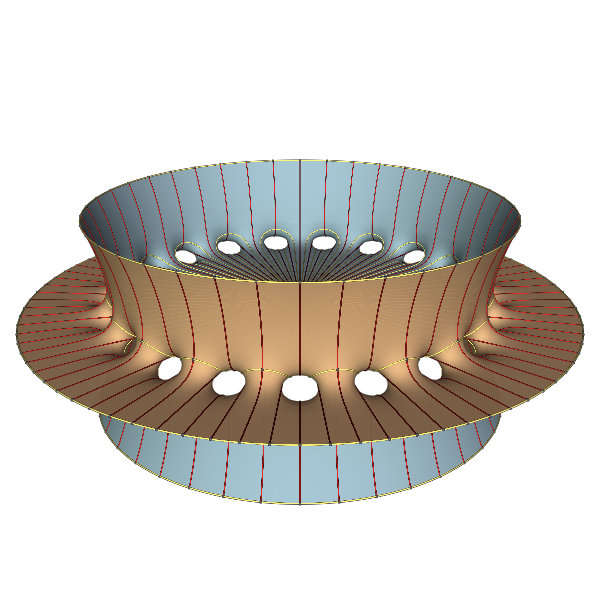}}
\caption{Costa-Hoffman-Meeks surface (from www.indiana.edu/\~{}minimal/)}
\label{chm}
\end{figure}

\begin{figure}

\centerline{\includegraphics[width=9cm]{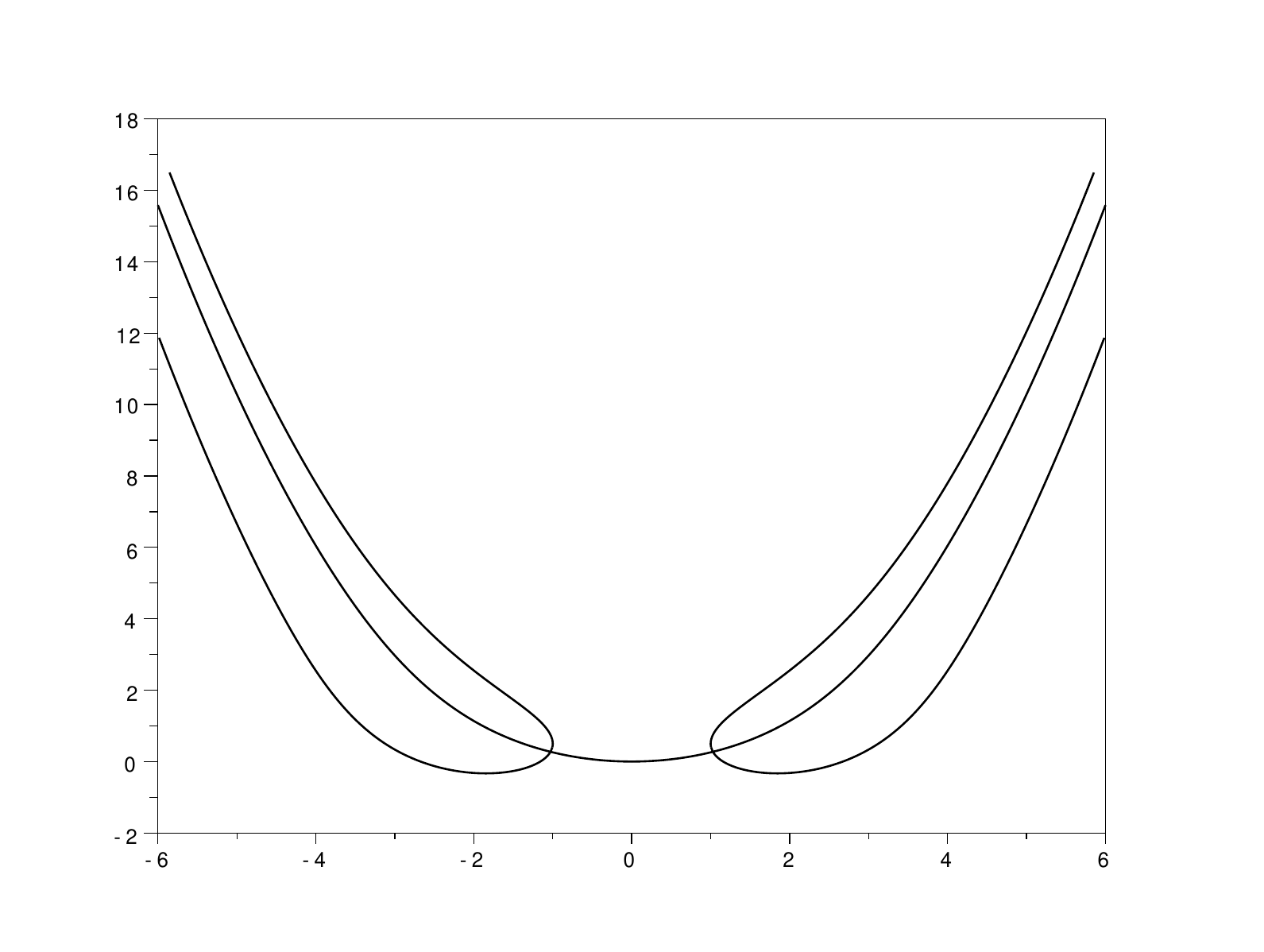}}
\caption{Travelling paraboloid and catenoid}
\label{fig1}
\end{figure}

%

\begin{theorem}
\label{main thm}
Let $\mathcal P$ and $\mathcal{W}$ be respectively a travelling paraboloid and catenoid, which intersect transversally.  Then for all $\ve>0$ small, there is a complete embedded 3-end surface $M_\ve$ satisfying equation \eqref{selftr-mcf}, which lies within an $\ve$-neighborhood of  $\mathcal P\cup \mathcal{W}  $.  Besides we have that
$$
{\rm genus}\, (M_\ve) \sim \frac 1\ve .
$$
\end{theorem}

\medskip
The construction provides much finer properties of the surface $M_\ve$. Let us point out that the CHM with large genus was found  in \cite{hoffmann-meeks-1990-limits}  to approach in the multiple-tunnel zone,  a Scherk singly periodic minimal surface.
See Figure \ref{chm}.

\medskip
 Kapouleas \cite{kapouleas-1990,kapouleas-1991,kapouleas-1997}, Traizet \cite{traizet,40} and Hauswirth and Pacard \cite{20} established a method for the reverse operation, namely, starting with a union of intersecting catenoids and planes, they desingularize them using Scherk surfaces to produce smooth minimal surfaces (complete and embedded). A key element in those constructions is a fine knowledge of the Jacobi operator of the Scherk surface and along the asymptotically flat ends.
This approach was used by the third author to construct translating surfaces in $\R^3$ built from a
2d picture of intersecting parallel grim reapers and vertical lines, trivially extended in an additional direction, desingularized in that direction by infinite Scherk surfaces, see
 \cite{nguyen-2009, nguyen, nguyen-2014}. We shall use a similar  scheme in our construction.  The
context here is considerably more delicate, since no periodicity is involved (the ultimate reason why the topology resulting is finite), and the fine interplay
between the slowly vanishing curvatures and the Jacobi operators of the different pieces requires
new ideas.  Our method extends to the construction of  more general surfaces built upon
desingularization of intersection of multiple travelling catenoids and travelling paraboloids,
but for simplicity in the exposition we shall restrict ourselves to the basic context of Theorem
\ref{main thm}. Before proceeding into the detailed proof, we sketch below the core ingredients of it.

\bigskip
\subsection{Sketch of the proof of Theorem \ref{main thm}}

After a change of scale of $1/\ve$, the problem is equivalent to finding a complete embedded surface $M \subset \R^3$ that satisfies
\begin{align}
\label{eq2}
H_M = \ve \nu \cdot e_z .
\end{align}

The first step is to construct a surface $\mathcal M$ that is close to being a solution to this equation. This is accomplished by desingularizing the union of  $\mathcal P/\ve$ and  $\mathcal{W}_R /\ve$ using singly periodic Scherk surfaces. At a large distance from $C_\rho/\ve$, the approximation $\mathcal M$ is $\mathcal P/\ve \cup \mathcal W_R/\ve$  and  in  some neighborhood of $C_\rho/\ve$, it is a slightly bent singly periodic Scherk surface.
We call the core of $\mathcal M$ the region where the desingularization is made.
The actual approximation $\mathcal M$ will depend on four real parameters: $\beta_1,\beta_4,\tau_1,\tau_4$, which are going to be small, of order $\ve$.

Let $\nu$ denote a choice of unit normal of $\mathcal M$. We search for a solution of \eqref{eq2} in the form of the normal graph over $\mathcal M$ of a function
$\phi:\mathcal M \to \R$, that is, of the form
$$
\mathcal M_\phi = \{ x+ \phi \nu(x): x \in \mathcal M \}.
$$
Let $H_\phi$ and $\nu_\phi$ denote the mean curvature and normal vector of $\mathcal M_\phi$, respectively, while $H$ and $\nu$ denote those of $\mathcal M$. Then
\begin{align}
\label{expansion mean curvature}
H_\phi = H + \Delta \phi + |A|^2 \phi + Q_1,
\end{align}
$$
\nu_\phi = \nu - \nabla \phi + Q_2,
$$
where $\Delta $ is the Laplace-Beltrami operator on $\mathcal M$, $\nabla $ is the tangential component of the gradient, and  $Q_1$, $Q_2$ are quadratic functions in $\phi, \nabla \phi, D^2 \phi$.
This allows us to write equation \eqref{eq2} as
\begin{align}
\label{b1}
\Delta \phi + |A|^2 \phi + \ve \nabla \phi\cdot e_z + H - \ve \nu\cdot e_z + Q(x,\phi,\nabla \phi, D^2 \phi)=0
\quad\text{in }\mathcal M .
\end{align}

To solve \eqref{b1}, we linearize around $\phi=0$, and the following linear operator becomes relevant:
$$
\lcal_{\ve}(\phi) = \Delta \phi + |A|^2 \phi + \ve \nabla \phi\cdot e_z.
$$
We work with the following norms for functions $\phi,h$ defined on $\mathcal M$, where $0<\gamma<1$, $0<\alpha<1$ are fixed:
\begin{align}
\label{norm u}
\|\phi\|_* = \sup_{s(x)\leq \delta_s/\ve} e^{\gamma s(x)} \| \phi \|_{C^{2,\alpha}(\overline{B}_1(x))}
+ \ve^2 \sup_{s(x) > \delta_s/\ve} e^{\gamma\delta_s/\ve +\ve \gamma s(x)} \| \phi \|_{C^{2,\alpha}(\overline{B}_1(x))}
\end{align}
and
\begin{align}
\label{norm rhs}
\|h\|_{**} =\sup_{s(x)\leq \delta_s/\ve} e^{\gamma s(x)} \| h \|_{C^{\alpha}(\overline{B}_1(x))} +\sup_{s(x) > \delta_s/\ve} e^{\gamma\delta_s/\ve +\ve \gamma s(x)} \| h \|_{C^{\alpha}(\overline{B}_1(x))} .
\end{align}
Here $\delta_s>0$ is a small fixed parameter. The function $s:\mathcal M\to \R$ measures geodesic distance to the core of $\mathcal M$ and will be defined precisely later on, and $B_1(x)$ is the geodesic ball centered at $x$ with radius~1.

The term in \eqref{b1}  that does not depend on $\phi$ is
$$
E = H - \ve \nu \cdot e_z .
$$
We have the following approximation for it.
\begin{prop}
\label{prop error}
$E$ can be decomposed as
$$
E = E_0 + E_d
$$
with
$$
\|E_0\|_{**} \leq C \ve
$$
and
$$
E_d = \tau_1 w_1 + \tau_4 w_2 + \beta_1 w_1' + \beta_4 w_2' + O(\sum_{i=1,4} \beta_i^2 + \tau_i^2).
$$
The functions $w_i$, $w_i'$ are defined later in \eqref{def w}, \eqref{def w prime}, but  they and the ones appearing in $O(\sum_{i=1,4} \beta_i^2 + \tau_i^2)$ are smooth with compact support, in particular have finite $\| \ \|_{**}$ norm.
\end{prop}

The following claim illustrates de invertibility of the linear operator $\lcal_{ve}$, although it will not be used directly. Let us fix $\mathcal M$ by fixing the parameters $\beta_1,\beta_4,\tau_1,\tau_4$ sufficiently small and consider the problem

\begin{align}
\label{intro linear M}
\Delta \phi + |A|^2 \phi + \ve \partial_z \phi = h + \sum_{i=1,4} \tilde \beta_i w_i' + \tilde \tau_i w_i
\quad\text{in }  \mathcal M.
\end{align}

Then, for $\ve>0$ small, there is a linear operator $h\mapsto \phi,\tilde\beta_i,\tilde\tau_i$ that produces for $\|h\|_{**}<\infty$ a solution of \eqref{intro linear M} with
$$
\|\phi\|_* + |\tilde\beta_1|+|\tilde\beta_4|+|\tilde\tau_1|+|\tilde\tau_4|\leq
C \|h\|_{**} ,
$$
where $C$ is independent of $\ve$.

Finally, the next result shows that the quadratic term $Q$ in \eqref{b1} is well adapted to the norms \eqref{norm u} and \eqref{norm rhs}.
\begin{prop}
\label{prop quadratic}
Assume $\phi_i \in C^{2,\alpha}(\mathcal M)$ ($i=1,2)$ and $\|\phi_i\|_*\leq 1$. Then, for $\ve>0$ small,
$$
\|Q(\cdot,\phi_1,\nabla\phi_1,D^2\phi_1)- Q(\cdot,\phi_2,\nabla\phi_2,D^2\phi_2) \|_{**}\leq C (\|\phi_1\|_*+\|\phi_2\|_*)\|\phi_1-\phi_2\|_*,
$$
with $C$ independent of $\phi_i$ and $\ve$.
\end{prop}

These results can be used to prove Theorem~\ref{main thm} by the contraction mapping principle, which is done in Section~\ref{sec proof them}. The preparatory steps are the construction of an initial approximate solution in Section~\ref{sec:construction of M} and some geometric computations in Section~\ref{sec Geometric computations}, which lead to the estimate of $E$ in Proposition~\ref{prop error} and the estimate of $Q$ in Proposition~\ref{prop quadratic}. In Section~\ref{sec:jacobi-on-scherk}, we analyze the Jacobi equation for the Scherk surface and in Section~\ref{sec Linear theory on the ends}, we study the Jacobi operator on the ends, which are the regions far from the desingularization.

\section{Construction of an initial approximation}
\label{sec:construction of M}

The purpose of this section is to construct a surface $\mathcal M$ that will serve as an initial approximation to \eqref{selftr-mcf}.

Let $F_0$ the unique radially symmetric solution of
\begin{align}
\label{eqFrad}
\frac{F''}{1+(F')^2} + \frac{F'}r = 1, \quad F(0)=0
\end{align}
 and let $\mathcal P$ be the corresponding surface $z = F_0(r)$.
Let $\mathcal{W}$ be a catenoidal self-translating solution of MCF, which can be written as $\mathcal{W} = \mathcal{W}^+\cup\mathcal{W}^-$ where $\mathcal{W}_R^\pm$ is given by $z=F^\pm(r)$ and $F_R^\pm$ satisfies \eqref{eqFrad} for $r>R$, with $F^+(R) = F^-(R)$,
$\lim_{r\to R^+}(F^+)'(r)=\infty$,
$\lim_{r\to R^+}(F^-)'(r)=-\infty$.

We assume that $\mathcal P$ and $\mathcal{W}$ intersect transversally at a unique circle $C_\rho$ of radius $\rho>0$. To quantify the transversality, we fix a small constant $\delta_{\alpha}>0$ so that all the intersection angles are greater than $4 \delta_{\alpha}$.
In this section,
%
we are going to replace  $\mathcal P \cup \mathcal{W}$  in a neighborhood of $C_{\rho}$ with an appropriately bent Scherk surface. The number of periods used, and thus the number of handles, is of order $\ve^{-1}$.
Two of the three ends of the resulting approximate solution will differ slightly from the original ends.


\subsection{Self-translating rotationally symmetric surfaces}
\label{ssec:rotationally_symmetric_surfaces}

We briefly recall some properties of self-translating  rotationally symmetric surfaces.
Let $\ve>0$ be a small constant, let $\gamma(s) = (\gamma_1(s), \gamma_3(s))$,  $s \in [0,\infty)$ be a smooth planar curve parametrized by arc length and let $S$ and $S_{\ve}$ be the surfaces of revolution parametrized by
\begin{align}
\label{def X}
\left\{
\begin{aligned}
(s,\theta) &\mapsto X (s,\theta) := (\gamma_1(ps) \cos (\theta), \gamma_1(ps) \sin(\theta),\gamma_3(ps))
\\
(s,\theta) &\mapsto X_{\ve}(s,\theta):=\ve^{-1}(\gamma_1(\ve p s) \cos (\ve \theta), \gamma_1(\ve p s) \sin(\ve \theta), \gamma_3(\ve p s)),
\end{aligned}
\right.
\end{align}
where $p=\gamma_1(0)$, and $s\in [0,\infty)$, $\theta\in [0,2\pi]$.
(The reason for introducing $p$ in \eqref{def X} is to make the parametrization conformal at $s=0$.)

The surface $S$ ($S_{\ve}$ respectively) is a self-translating surface under mean curvature flow with velocity $e_z$ ($\ve e_z$ respectively) if and only if  $\gamma$, parametrized by arc length, satisfies the differential equation
\begin{equation}
\label{eq:self-trans-ode}
- \gamma_1'' \gamma_3' +  \gamma_1' \gamma_3''+ \frac{\gamma_3'}{\gamma_1}-\gamma_1'=0.
\end{equation}

Another way to represent an axially symmetric self-translating solution is through the graph of a radial function, $z = F(r)$, where $F$ satisfies \eqref{eqFrad} on some interval $(R,\infty)$. Then $\varphi = F'$ satisfies
\begin{align}
\label{eq varphi}
\varphi' = (1+\varphi^2)\left(1-\frac\varphi r\right).
\end{align}
Given $R>0$ and an initial condition $\varphi(R)= \varphi_0 \in \R$, equation \eqref{eq varphi} has unique solution, which is defined for all $r \geq R$, see \cite{clutterbuck-schnurer-schulze}. All solutions have the common asymptotic behavior
\begin{align}
\label{asymp varphi}
\varphi(r) = r -\frac1r - \frac2{r^3} + O(\frac1{r^{5}}) ,
\quad
\varphi'(r) = 1 +\frac{1}{r^2} + O(\frac1{r^{4}}),
\end{align}
as $r\to\infty$, see  \cite{angenent-velazquez,clutterbuck-schnurer-schulze} (actually an expansion to arbitrary order is possible).

Using $\gamma_1(s) = r(s)$, $\gamma_3(s) = F(r)$, with $s (r)=\int_0^r \sqrt{1+\varphi(t)^2}\,dt$ and the asymptotic behavior \eqref{asymp varphi}, we can deduce the following estimates.

\begin{lemma}
\label{lemma:order-gamma}
For a smooth planar curve $\gamma(s) = (\gamma_1(s), \gamma_3(s))$, $s\in [0,\infty)$ parametrized by arc length with $\gamma_1$ and $\gamma_3$ satisfying \eqref{eq:self-trans-ode}, we have 	 
\begin{align*}
\gamma_1(s) &= \sqrt{2s} + \frac12 + o(1)
& \gamma_3(s) & = s+O(\sqrt{s})\\
\gamma_1'(s) &= \frac{1}{ \sqrt{2s}} + o(s^{-1/2}) & \gamma_3'(s) & = 1+O(s^{-1/2})\\
\gamma_1''(s) &= O(s^{-3/2}) & \gamma_3''(s) & = O(s^{-2})
\end{align*}
as $s$ tends to infinity.
\end{lemma}


\subsection{The Scherk surfaces}

Let $x,y,z$ be Euclidean coordinates in $\R^3$ and consider the one parameter family of minimal surfaces $\{\Sigma(\alpha)\}_{\alpha \in (0, \pi/2)}$ given by the equation
\begin{equation}
\label{scherk}
	\cos^2(\alpha) \cosh \left(\frac{x}{\cos \alpha}\right) - \sin^2 (\alpha) \cosh \left(\frac{y}{\sin \alpha}\right) - \cos (z)=0.
\end{equation}
Outside of a large cylinder around the $z$-axis, $\Sigma(\alpha)$ has four connected components. We call these components the wings of $\Sigma(\alpha)$ and number them according to the quadrant where they lie. Each wing of $\Sigma(\alpha)$ is asymptotic to a half-plane forming an angle $\alpha$ with the $xz$-plane (note that the asymptotic half-planes do not contain the $z$-axis unless $\alpha=\pi/4$). Here, we will restrict the parameter $\alpha$  to $ [\delta_{\alpha}, \pi/2 - \delta_\alpha]$ so that the geometry on all the $\Sigma(\alpha)$'s  can be uniformly bounded as stated in the following lemma.

Let $H^+$ be the half-plane $\{(s,z): s>0\}$. Note that the parameter $s$ here is on a different scale than the one used in the previous section. We construct approximate solutions satisfying \eqref{eq2} here, while the rotationally symmetric surfaces in Section \ref{ssec:rotationally_symmetric_surfaces} satisfy \eqref{selftr-mcf}.

\begin{lemma}
\label{lem:scherk-properties}
$\Sigma(\alpha)$ is a singly periodic embedded complete minimal surface which depends smoothly on $\alpha$.  There is a constant $a=a(\delta_{\alpha}) >0$ and smooth functions $f_{\alpha}:H^+ \to \R$ so that the wings of $\Sigma_{\alpha}$ can be expressed as the graph of $f_{\alpha}$ over half-planes. More precisely, the half-plane asymptotic to the first wing can be parametrized by $A^1_{\alpha}:H^+ \to \R^3$, with
	\[
	A^1_{\alpha}(s,z) :=  (a+s) ((\cos \alpha) e_x + (\sin \alpha) e_y) + z e_z +b_{\alpha} \nu_{\alpha},
	\]
where $b_{\alpha} = \sin(2 \alpha) \log |\cot \alpha|$ and $\nu_{\alpha}= -(\sin \alpha) e_x + (\cos \alpha) e_y$. The wing itself is parametrized by $F^1_{\alpha}:H^+ \to \R^3$, which is defined by
	\[
	F^1_{\alpha}(s, z) :=A^1_{\alpha}(s,z) +  f_{\alpha}(s,z)\nu_{\alpha}.
	\]
The functions $f_{\alpha}$ and $F^1_{\alpha}$ depend smoothly on $\alpha$. Moreover, we have
$$
\left\|e^s  \frac{d^i f_{\alpha}}{d\alpha^i}\right\|_{C^k(\overline {H^+})} \leq C_{k,i} e^{-a}
$$
for any $k,i\in \N$.


\end{lemma}

The function $f_\alpha(s,z)$ satisfies the minimal surface equation
\begin{align}
\label{min surf eq}
\partial_s \left( \frac{{\partial_s}f}{\sqrt{1+(\partial_s f)^2+(\partial_z f)^2}} \right)+
\partial_z \left( \frac{{\partial_z}f}{\sqrt{1+(\partial_s f)^2+(\partial_z f)^2}} \right)=0,
\end{align}
for $s>0$, $z\in\R$.

	\begin{definition}
\label{def:wings} Let us denote by $\mathcal R_{yz}$ the reflection across the $yz$-plane and $\mathcal R_{xz}$ the reflection across the $xz$-plane.  The parametrizations of the second, third, and fourth wings are given by
\begin{gather*}
F^2_{\alpha} = \mathcal R_{yz} \circ F^1_{\alpha},\quad
F^3_{\alpha} = \mathcal R_{xz} \circ \mathcal R_{yz} \circ F^1_{\alpha}, \quad
F^4_{\alpha} = \mathcal R_{xz} \circ F^1_{\alpha}.
\end{gather*}
The {\bf $i$th wing} of $\Sigma(\alpha)$ is  given by $F^i_{\alpha}(H^+)$ and is denoted by $W^i(\alpha)$. The parametrizations of the corresponding asymptotic half-planes are obtained by replacing $F^1_{\alpha}$ by $A^1_{\alpha}$ in the above formulas. We use $A^i_\alpha$ to denote the parametrization of the $i$th asymptotic half-plane as well as its image, $A^i_{\alpha}(H^+)$.
	The {\bf  inner core} of $\Sigma(\alpha)$ is the surface without its four wings.
	\end{definition}
	
Each half-plane $A^i_{\alpha}(H^+)$ starts close to the boundary of the corresponding wing $W^i$ and intersects neither the $xz$-plane nor the $yz$-plane. Each wing and each asymptotic half-plane inherit the coordinates $(s,z)$ from their descriptions in Lemma \ref{lem:scherk-properties} and Definition \ref{def:wings}. 

\begin{figure}

\hbox{
\p 0,0,{\includegraphics{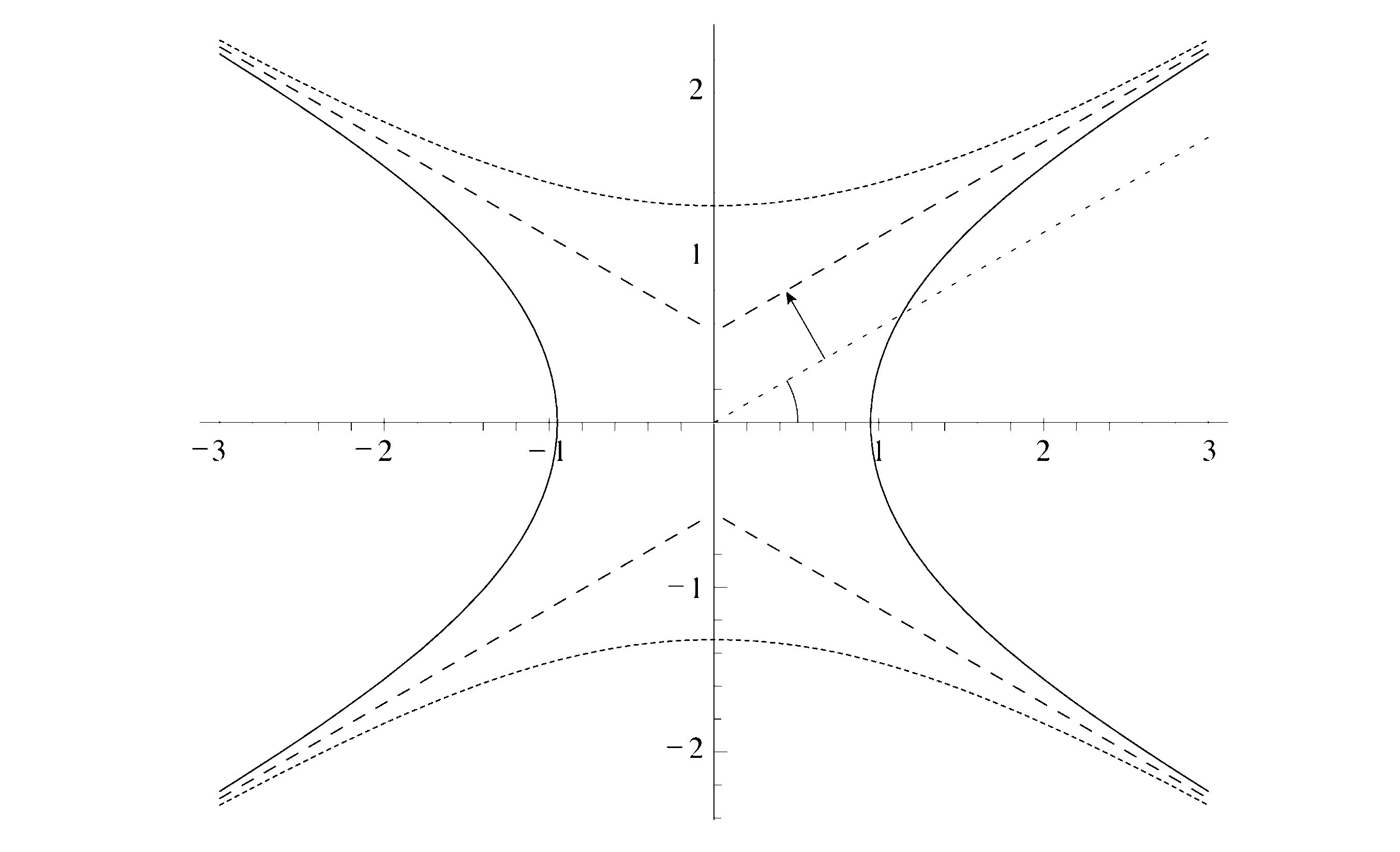}}
\p 295,92,{$x$}
\p 172,182,{$y$}
\p 290,180,{$W^1(\alpha)$}
\p 15,180,{$W^2(\alpha)$}
\p 15,10,{$W^3(\alpha)$}
\p 290,10,{$W^4(\alpha)$}
\p 190,105,{$\alpha$}
\p 195,125,{$b_\alpha$}
\p 140, 110,{$\frac{b_{\alpha}}{\cos \alpha}$}
\p 230,160,{$A_\alpha^1$}
\p 90,160,{$A_\alpha^2$}
\p 90,35,{$A_\alpha^3$}
\p 230,35,{$A_\alpha^4$}
\p 100,-20,{-----}
\p 120,-20,{intersection with $z=0$}
\p 100,-30,{$\cdot\!\!\cdot\!\!\cdot\!\!\cdot\!\!\cdot\!\!\cdot\!\!\cdot\!\!\cdot\!\!\cdot$}
\p 120,-30,{intersection with $z=\pi$}
\p 100,-40,{- - -}
\p 120,-40,{asymptotic planes $A_\alpha^i$}
}

\caption{Sections of the Scherk surface $\Sigma(\alpha)$.}
\label{fig:scherk}
\end{figure}



\subsection{Dislocation of the Scherk surfaces}

We now perform dislocations on the first and fourth wings of $\Sigma(\alpha)$. These perturbations will help us to deal with the kernel of the linear operator $\lcal_{\ve}$ associated to normal perturbations of solutions to \eqref{eq2}. Indeed, because translated solutions of \eqref{eq2} remain solutions, the functions $e_x \cdot \nu$, $e_y \cdot \nu$, and $e_z \cdot \nu$ are in the kernel of $\lcal_{\ve}$. Here we have taken the normal component of the translations because we are considering normal perturbations. The last function, $e_z \cdot \nu$ does not satisfy our imposed symmetries. Therefore, we can discard it from the kernel. The other two remain. In Section \ref{sec:jacobi-on-scherk},
we will show that the Dirichlet problem for the linear operator can be solved on a truncated piece of $\Sigma(\alpha)$, up to constants at the boundary. By adding a linear combination of the functions in the kernel, we can obtain a solution that vanishes on the boundary of two adjacent wings, say the second and third wings. To obtain a solution that vanishes on all the connected pieces of the boundary, we will artificially translate the first and fourth wing by constants $\tau_1$ and $\tau_4$.

The linear operator $\lcal_{\ve}$ is close to linear operator $L:= \Delta+|A|^2$ associated to the equation $H=0$, so we have small eigenvalues due to changes of Scherk angle and rotation. Indeed, because there is a one parameter family of Scherk surfaces, we expect a function in the kernel of the Jacobi operator $L$, namely, the normal component of the motion associated to changing the angle $\alpha$. One more dimension is generated by rotation of the Scherk surfaces around the $z$-axis. To summarize, besides the translations, we have two more dimensions in the kernel of $L$ generated by linear functions along the wings. This is reasonable since the $L$ is close to the Laplace operator along the wings. By adding a linear combination of these two linear eigenfunctions, we can force exponential decay along the second and third wings again. As before, we will generate linear functions on the first and fourth wings through rotations by angles $\beta_1$ and $\beta_4$ respectively.

\begin{definition} For $\beta \in \R$, we define the map $Z_\beta:\R^3 \to \R^3$ to be the rotation of angle $\beta$ (counterclockwise in the $xy$-plane) around the $z$-axis:
$$
Z_{\beta}(x,y,z) = ( \cos(\beta) x -\sin(\beta) y , \sin(\beta)x+\cos(\beta)y ,z).
$$
\end{definition}

In what follows, we will confine $\beta$ to $(-\delta_p, \delta_p)$, where $\delta_p>0$ is a small fixed number.
	
We consider two constants $R_{rot}>10$, $R_{tr}>R_{rot}+10$, and a family of smooth transition functions $\eta_{b}: \R \to \R$ such that $0\leq \eta_{b}(s)\leq 1$, $\eta_{b}(s) =0$ for  $s<b$, and $\eta_{b}(s) =1$ for $s>b+1$. The numbers $R_{rot}$, $R_{tr}$ will be fixed later to be large.


Given $\alpha \in  [\delta_{\alpha}, \pi/2 - \delta_\alpha]$, $\beta_1, \beta_4, \tau_1, \tau_4 \in (-\delta_p,\delta_p)$,
we modify the first and fourth wings in the following way: the $i$th wing is shifted by $\tau_i$ at around $s=R_{tr}$, then it is rotated by an angle $\beta_i$ at distance $s=R_{rot}$. The parametrization of the new $i$th wing, for $i=1,4$, is given by $F^i[\alpha, \beta_i, \tau_i]: H^+ \to \R^3$, where
	\begin{align*}
	F^1[\alpha, \beta_1, \tau_1] (s, z) &= (1-\eta_{R_{rot}}(s)) F^1_{\alpha}(s,z)+ \eta_{R_{rot}}(s)Z_{\beta_1}(F^1_{\alpha}(s,z) +  \tau_1 \eta_{R_{tr}} (s)\nu_{\alpha})\\
	F^4[\alpha, \beta_4, \tau_4] (s, z) &= \mathcal R_{xz} \circ F^1[\alpha, \beta_4, \tau_4] (s, z),
	\end{align*}
and $\mathcal R_{xz}$ is the reflection across the $xz$-plane.
Note that the $i$th wing is moved away from the $x$-axis for positive constants $\beta_i$ and $\tau_i$. We denote the new wings by $W^i[\alpha, \beta_i, \tau_i] :=F^i[\alpha, \beta_i, \tau_i](H^+)$, $i=1,4$ (see Figure \ref{fig dislocations}).

The wings have natural coordinates $(s,z)$ given by the parametrizations $F^1$ and $F^4$. The surface $\Sigma'[\alpha, \beta_1, \beta_4, \tau_1, \tau_4]$ (or $\Sigma'$ for short) is defined to be the union of the inner core of $\Sigma(\alpha)$ with the four wings $W^2(\alpha)$, $W^3(\alpha)$, $W^1[\alpha, \beta_1, \tau_1]$, and $W^4[\alpha, \beta_4, \tau_4]$. We will call the region of $\Sigma'$ for which $s \in [0, R_{tr}+10]$ the {\bf outer core}.

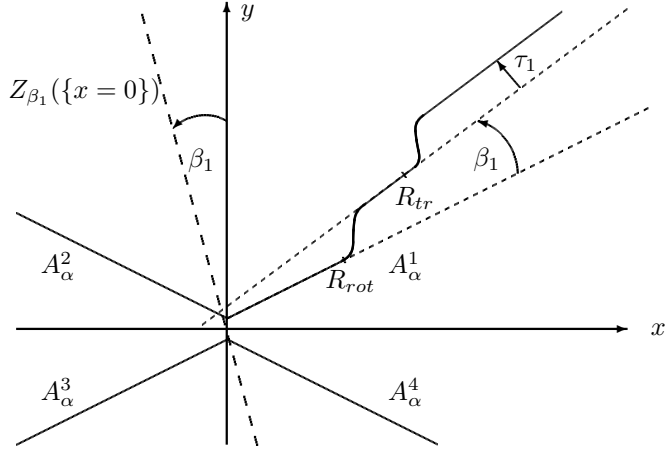
\begin{figure}

\def\JPicScale{0.7}

\ifx\JPicScale\undefined\def\JPicScale{1}\fi
\unitlength \JPicScale mm
\begin{picture}(132,80)(0,0)
\linethickness{0.3mm}
\multiput(50,20)(0.24,0.12){92}{\line(1,0){0.24}}
\linethickness{0.3mm}
\multiput(10,40)(0.24,-0.12){167}{\line(1,0){0.24}}
\linethickness{0.3mm}
\multiput(10,-4)(0.24,0.12){167}{\line(1,0){0.24}}
\linethickness{0.3mm}
\multiput(50,16)(0.24,-0.12){167}{\line(1,0){0.24}}
\put(106.79,68.39){\makebox(0,0)[cc]{$\tau_1$}}

\linethickness{0.3mm}
\multiput(101.47,68.88)(0.12,-0.14){36}{\line(0,-1){0.14}}
\put(101.47,68.88){\vector(-3,4){0.12}}
\put(84,30){\makebox(0,0)[cc]{$A_\alpha^1$}}

\put(78.75,30.62){\makebox(0,0)[cc]{}}

\put(84,6){\makebox(0,0)[cc]{$A_\alpha^4$}}

\put(18,6){\makebox(0,0)[cc]{$A_\alpha^3$}}

\put(18,30){\makebox(0,0)[cc]{$A_\alpha^2$}}

\linethickness{0.3mm}
\put(10,18){\line(1,0){116}}
\put(126,18){\vector(1,0){0.12}}
\linethickness{0.3mm}
\put(50,-4){\line(0,1){84}}
\put(50,80){\vector(0,1){0.12}}
\put(132,18){\makebox(0,0)[cc]{$x$}}

\put(104,17){\makebox(0,0)[cc]{}}

\put(54,78){\makebox(0,0)[cc]{$y$}}

\put(47,70){\makebox(0,0)[cc]{}}

\linethickness{0.3mm}
\multiput(105.26,48)(0.01,-0.5){1}{\line(0,-1){0.5}}
\multiput(105.23,48.5)(0.03,-0.5){1}{\line(0,-1){0.5}}
\multiput(105.17,49)(0.06,-0.5){1}{\line(0,-1){0.5}}
\multiput(105.1,49.49)(0.08,-0.49){1}{\line(0,-1){0.49}}
\multiput(104.99,49.98)(0.1,-0.49){1}{\line(0,-1){0.49}}
\multiput(104.86,50.47)(0.13,-0.48){1}{\line(0,-1){0.48}}
\multiput(104.71,50.94)(0.15,-0.48){1}{\line(0,-1){0.48}}
\multiput(104.54,51.41)(0.17,-0.47){1}{\line(0,-1){0.47}}
\multiput(104.34,51.87)(0.1,-0.23){2}{\line(0,-1){0.23}}
\multiput(104.12,52.32)(0.11,-0.22){2}{\line(0,-1){0.22}}
\multiput(103.88,52.76)(0.12,-0.22){2}{\line(0,-1){0.22}}
\multiput(103.62,53.19)(0.13,-0.21){2}{\line(0,-1){0.21}}
\multiput(103.34,53.6)(0.14,-0.21){2}{\line(0,-1){0.21}}
\multiput(103.03,54)(0.1,-0.13){3}{\line(0,-1){0.13}}
\multiput(102.71,54.38)(0.11,-0.13){3}{\line(0,-1){0.13}}
\multiput(102.37,54.75)(0.11,-0.12){3}{\line(0,-1){0.12}}
\multiput(102.01,55.1)(0.12,-0.12){3}{\line(1,0){0.12}}
\multiput(101.64,55.43)(0.12,-0.11){3}{\line(1,0){0.12}}
\multiput(101.25,55.75)(0.13,-0.1){3}{\line(1,0){0.13}}
\multiput(100.84,56.04)(0.2,-0.15){2}{\line(1,0){0.2}}
\multiput(100.43,56.31)(0.21,-0.14){2}{\line(1,0){0.21}}
\multiput(99.99,56.57)(0.22,-0.13){2}{\line(1,0){0.22}}
\multiput(99.55,56.8)(0.22,-0.12){2}{\line(1,0){0.22}}
\multiput(99.1,57.01)(0.23,-0.1){2}{\line(1,0){0.23}}
\multiput(98.63,57.19)(0.23,-0.09){2}{\line(1,0){0.23}}
\multiput(98.16,57.36)(0.47,-0.16){1}{\line(1,0){0.47}}
\multiput(97.68,57.5)(0.48,-0.14){1}{\line(1,0){0.48}}
\put(97.68,57.5){\vector(-4,1){0.12}}

\put(99.42,50.22){\makebox(0,0)[cc]{$\beta_1$}}

\put(81,37){\makebox(0,0)[cc]{}}

\linethickness{0.3mm}
\multiput(49.46,57.74)(0.22,-0.12){2}{\line(1,0){0.22}}
\multiput(49.01,57.95)(0.23,-0.11){2}{\line(1,0){0.23}}
\multiput(48.55,58.13)(0.23,-0.09){2}{\line(1,0){0.23}}
\multiput(48.08,58.29)(0.47,-0.16){1}{\line(1,0){0.47}}
\multiput(47.6,58.42)(0.48,-0.13){1}{\line(1,0){0.48}}
\multiput(47.11,58.51)(0.49,-0.1){1}{\line(1,0){0.49}}
\multiput(46.62,58.58)(0.49,-0.07){1}{\line(1,0){0.49}}
\multiput(46.12,58.62)(0.5,-0.04){1}{\line(1,0){0.5}}
\multiput(45.62,58.63)(0.5,-0.01){1}{\line(1,0){0.5}}
\multiput(45.12,58.62)(0.5,0.02){1}{\line(1,0){0.5}}
\multiput(44.63,58.57)(0.5,0.05){1}{\line(1,0){0.5}}
\multiput(44.14,58.49)(0.49,0.08){1}{\line(1,0){0.49}}
\multiput(43.65,58.38)(0.49,0.11){1}{\line(1,0){0.49}}
\multiput(43.17,58.24)(0.48,0.14){1}{\line(1,0){0.48}}
\multiput(42.7,58.08)(0.47,0.16){1}{\line(1,0){0.47}}
\multiput(42.24,57.89)(0.23,0.1){2}{\line(1,0){0.23}}
\multiput(41.8,57.67)(0.22,0.11){2}{\line(1,0){0.22}}
\multiput(41.36,57.42)(0.22,0.12){2}{\line(1,0){0.22}}
\multiput(40.95,57.15)(0.21,0.14){2}{\line(1,0){0.21}}
\multiput(40.55,56.86)(0.2,0.15){2}{\line(1,0){0.2}}
\multiput(40.16,56.54)(0.13,0.11){3}{\line(1,0){0.13}}
\multiput(39.8,56.2)(0.12,0.11){3}{\line(1,0){0.12}}
\put(39.8,56.2){\vector(-1,-1){0.12}}

\put(45.2,50.1){\makebox(0,0)[cc]{$\beta_1$}}

\linethickness{0.3mm}
\multiput(50,20)(1.8,0.9){45}{\multiput(0,0)(0.22,0.11){4}{\line(1,0){0.22}}}
\linethickness{0.3mm}
\qbezier(72,31)(73.04,31.51)(73.53,32.35)
\qbezier(73.53,32.35)(74.01,33.18)(74.01,34.46)
\qbezier(74.01,34.46)(74.01,35.74)(74.07,36.76)
\qbezier(74.07,36.76)(74.14,37.79)(74.26,38.73)
\qbezier(74.26,38.73)(74.39,39.67)(74.83,40.36)
\qbezier(74.83,40.36)(75.28,41.05)(76.12,41.61)
\linethickness{0.3mm}
\multiput(72.1,31.7)(0.12,-0.3){4}{\line(0,-1){0.3}}
\put(68.6,49.8){\makebox(0,0)[cc]{}}

\put(73.5,27.2){\makebox(0,0)[cc]{$R_{rot}$}}

\linethickness{0.3mm}
\multiput(75.95,41.5)(0.16,0.12){59}{\line(1,0){0.16}}
\linethickness{0.3mm}
\multiput(86.99,58.35)(0.16,0.12){167}{\line(1,0){0.16}}
\linethickness{0.3mm}
\qbezier(85.76,48.78)(86.33,49.29)(86.49,50.04)
\qbezier(86.49,50.04)(86.65,50.79)(86.41,51.88)
\qbezier(86.41,51.88)(86.17,52.97)(86.04,53.88)
\qbezier(86.04,53.88)(85.9,54.78)(85.84,55.65)
\qbezier(85.84,55.65)(85.77,56.52)(86.11,57.22)
\qbezier(86.11,57.22)(86.46,57.92)(87.28,58.57)
\linethickness{0.3mm}
\multiput(83.35,47.84)(0.11,-0.16){7}{\line(0,-1){0.16}}
\put(85.71,42.95){\makebox(0,0)[cc]{$R_{tr}$}}

\linethickness{0.3mm}
\multiput(45.29,18.79)(1.6,1.19){51}{\multiput(0,0)(0.16,0.12){5}{\line(1,0){0.16}}}
\linethickness{0.3mm}
\multiput(33.95,77.89)(1.02,-3.82){22}{\multiput(0,0)(0.13,-0.48){4}{\line(0,-1){0.48}}}
\put(23,63){\makebox(0,0)[cc]{$Z_{\beta_1}(\{x=0\})$}}

\put(30.2,64.3){\makebox(0,0)[cc]{}}

\end{picture}

\caption{Dislocations in wing 1}
\label{fig dislocations}
\end{figure}

	\begin{remark}
	\label{rem:pullback}
	The maps $F^i_{\alpha} \circ( F^i[\alpha, \beta_i, \tau_i])^{-1}$ and $F^i[\alpha, \beta_i, \tau_i] \circ (F^i_{\alpha})^{-1}$, $i=1,4$ can be used to pullback tensors defined on $W^i(\alpha)$ to $W^i[\alpha, \beta_i, \tau_i]$ and vice versa: in the case of a function $f$ defined on $W^i[\alpha, \beta_1, \tau_1]$, the composition $f  \circ F^i[\alpha, \beta_i, \tau_i] \circ (F^i_{\alpha})^{-1}$ is the corresponding pullback function on $W^i({\alpha})$. Taking each wing at a time, these maps transport functions and tensors between $\Sigma'$ and $\Sigma(\alpha)$. This is very useful as it lets us work on a fixed surface, usually $\Sigma(\alpha)$. We will use the same notation for functions and tensors on $\Sigma'$ or their pullback to $\Sigma(\alpha)$.
For example, $H_{\Sigma'}$ could denote the mean curvature of $\Sigma'$ as a function on $\Sigma'$ or its pullback to $\Sigma(\alpha)$.
The same notation convention applies to the unit normal vector $\nu$, the metric $g$, and the second fundamental form $A$.
	\end{remark}

Let us define the following functions on $\Sigma(\alpha)$, which capture the contribution of the dislocations to the mean curvature:
\begin{align}
\label{def w}
	w_1 := \frac{d}{d\tau_1}\Big|_{\beta_i=\tau_i=0} H_{\Sigma'}, \quad w_2 := \frac{d}{d\tau_4}\Big|_{\beta_i=\tau_i=0}  H_{\Sigma'}, \\
\label{def w prime}
	w'_1 := \frac{d}{d\beta_1} \Big|_{\beta_i=\tau_i=0} H_{\Sigma'}, \quad w'_2:=\frac{d}{d\beta_4} \Big|_{\beta_i=\tau_i=0} H_{\Sigma'}.
\end{align}
These functions are compactly supported because rotations and translations do not change the mean curvature. They will later help us solve the Dirichlet problem associated to the Jacobi operator $\Delta + |A|^2$ on the Scherk surfaces in Section \ref{sec:jacobi-on-scherk}.

Because the parameters $\beta_i$ are associated to rotations, the functions $w'_1$ and $w'_2$ can be written explicitly as the Jacobi operator on the normal component of rotation at a point $(x,y,z) \in \Sigma(\alpha)$:
\begin{align}
\label{formula w prime}
\left\{
\begin{aligned}
w'_1 (x,y,z) &= (\Delta_{\Sigma(\alpha)} + |A_{\Sigma(\alpha)}|^2) (\eta_{rot,1} \nu_{\Sigma(\alpha)} \cdot (-y,x,0)),  \\
w'_2 (x,y,z) &= (\Delta_{\Sigma(\alpha)} + |A_{\Sigma(\alpha)}|^2) (\eta_{rot,4} \nu_{\Sigma(\alpha)} \cdot (-y,x,0)),
\end{aligned}
\right.
\end{align}
where $\eta_{rot,1}(s)$ is defined as $\eta_{R_{rot}}(s)$ on wing 1 and zero elsewhere and similarly for $\eta_{rot,4}$.
We also have,
\begin{align}
\label{formula w}
\left\{
\begin{aligned}
w_1  &= (\Delta_{\Sigma(\alpha)} + |A_{\Sigma(\alpha)}|^2) (\eta_{tr,1} ),\\
w_2  &= (\Delta_{\Sigma(\alpha)} + |A_{\Sigma(\alpha)}|^2) (\eta_{tr,4}),
\end{aligned}
\right.
\end{align}
on $\Sigma(\alpha)$,
where $\eta_{tr,1}(s)=\eta_{R_{tr}}$ on wing 1 and zero elsewhere, and similarly for  $\eta_{tr,4}$.

\subsection{Wrapping the dislocated Scherk surfaces around a circle} We first rotate our new surface $\Sigma'$ so that its second and third wings match the directions of two chosen pieces of catenoid or paraboloid coming out of the intersection circle. The wrapping is performed by simply using a smooth map from a tubular neighborhood of the $z$-axis to a neighborhood of a large circle. The scaling factor is $\ve ^{-1}$ so our target circle will have a radius of order $\ve^{-1}$.

\begin{definition}
For $\ve>0$  and $\varrho>0$, we define
\[
B_{\ve,\varrho} (x,y,z) = (\ve^{-1} \varrho +x) (\cos (\ve \varrho^{-1} z), \sin(\ve \varrho^{-1} z),0)  + (0,0,  y).	
\]
\end{definition}
This map takes a segment of length $2\pi \ve ^{-1} \varrho$ on the $z$-axis to the circle of radius $\ve^{-1} \varrho$.

We can not wrap the whole surface $\Sigma'$, so we cut its four wings at $s=R_{tr}+10$ and denote the new surface by $\bar \Sigma'$, with a ``bar" on top to indicate that it has a boundary. Our desingularizing surface is a dislocated rotated wrapped Scherk surface
	\begin{equation}
	\label{eq:bar-sigma}
	\bar \Sigma := B_{\ve, \rho_{\ve}} \circ Z_\beta (\bar \Sigma'),
	\end{equation}
where the angle $\beta$ has yet to be fixed and $\rho_{\ve}$ is  the closest number in $\ve \Z$ to $\rho$ (the radius associated to the original intersection $C_{\rho}$)


We wish to prolong the wings of the desingularizing surface $\bar \Sigma$ with pieces of self-translating catenoids or paraboloids. At this point, it will be useful to record the boundary, not of the surface $\bar \Sigma$ itself, but of the asymptotic plane underneath at $s=R_{tr}+10$. We will extend the asymptotic pieces first, then construct the approximate surface by adding the graph of $f_{\alpha}$.

\subsection{Fitting the Scherk surface}
\label{section:Fitting the Scherk surface}
It is now time to examine the initial configuration in detail. We will work with cross-sections in the $xz$-plane. Let $C_{\rho}$ denote the intersection of the paraboloid and catenoid and let $\alpha_0 \in [\delta_{\alpha}, \pi/2-\delta_{\alpha}]$ be half of the angle of intersection between the top $\mathcal W_{R}$ and the inner part of $\mathcal P$ (see Figure \ref{fig3}).

\begin{figure}

\ifx\JPicScale\undefined\def\JPicScale{1}\fi
\unitlength \JPicScale mm
\begin{picture}(135,60)(40,30)
\linethickness{0.3mm}
\put(60,40){\line(1,0){73}}
\put(133,40){\vector(1,0){0.12}}
\linethickness{0.3mm}
\put(70,30){\line(0,1){60}}
\put(70,90){\vector(0,1){0.12}}
\put(116.88,68.75){\makebox(0,0)[cc]{$\mathcal P$}}

\put(135,60){\makebox(0,0)[cc]{}}

\put(106.88,82.88){\makebox(0,0)[cc]{$\mathcal W_R$}}

\put(85,51.88){\makebox(0,0)[cc]{$C_{\rho}$}}

\linethickness{0.3mm}
\qbezier(110,81)(101.14,73.7)(96.46,69.64)
\qbezier(96.46,69.64)(91.78,65.58)(90.56,64.12)
\qbezier(90.56,64.12)(89.29,62.64)(88.57,60.95)
\qbezier(88.57,60.95)(87.85,59.27)(87.56,57.12)
\qbezier(87.56,57.12)(87.26,54.98)(87.85,53.15)
\qbezier(87.85,53.15)(88.43,51.31)(90,49.5)
\qbezier(90,49.5)(91.54,47.68)(94.54,46.71)
\qbezier(94.54,46.71)(97.55,45.75)(102.5,45.5)
\qbezier(102.5,45.5)(107.44,45.23)(111.29,45.79)
\qbezier(111.29,45.79)(115.14,46.34)(118.5,47.81)
\qbezier(118.5,47.81)(121.87,49.27)(124.79,50.85)
\qbezier(124.79,50.85)(127.71,52.43)(130.62,54.38)
\linethickness{0.3mm}
\qbezier(70,48.12)(75.53,48.12)(79.36,48.42)
\qbezier(79.36,48.42)(83.2,48.72)(85.94,49.38)
\qbezier(85.94,49.38)(88.69,50.02)(91.47,51.14)
\qbezier(91.47,51.14)(94.25,52.27)(97.5,54.06)
\qbezier(97.5,54.06)(100.74,55.84)(104.12,57.95)
\qbezier(104.12,57.95)(107.5,60.05)(111.56,62.81)
\qbezier(111.56,62.81)(115.64,65.58)(117.52,66.86)
\qbezier(117.52,66.86)(119.4,68.14)(119.38,68.12)
\end{picture}
\caption{Paraboloid and catenoid with transversal intersection at $C_\rho$.}
\label{fig3}
\end{figure}
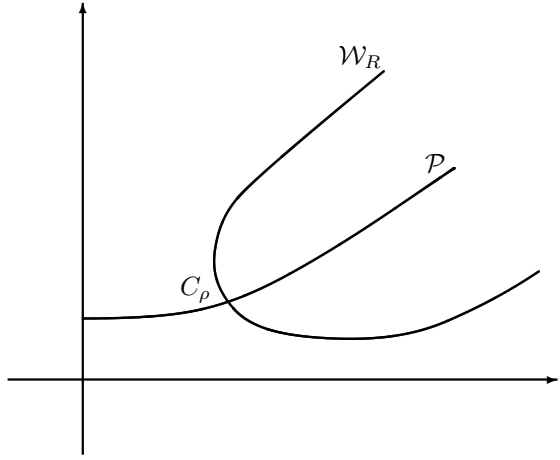

On the bounded part of the paraboloid, we take two points $P$ and $P'$ at distances $(a+b_{\alpha_0} \cot \alpha_0 + R_{tr}+20)\ve$ and $(a+b_{\alpha_0} \cot \alpha_0+R_{tr}+10)\ve$ from $C_{\rho}$ respectively and consider the half-line starting at $P'$ tangent to the paraboloid and pointing to $C_{\rho}$.
Recall that  $a$ was chosen in Lemma \ref{lem:scherk-properties} and that the term $a+ b_{\alpha_0} \cot \alpha_0$ is present because the distance from the intersection of the two (extended) asymptotic planes $A^2_{\alpha}(\R^2)$ and $A^3_{\alpha}(\R^2)$ to the line $A^2_{\alpha}(\{s=0\})$ on is $a+b_{\alpha} \cot \alpha$ by Lemma \ref{lem:scherk-properties}.
The new object $\tilde {\mathcal P}$ is formed by the paraboloid up to $P$, a smooth interpolating curve from $P$ to $P'$, and the tangent half-line after $P'$ (see Figure \ref{fig4}). We will also denote its corresponding surface of revolution by $\tilde {\mathcal P}$. We do a similar construction with the catenoid $\mathcal{W}_R$ and denote the new object $\tilde{\mathcal W}$  (see Figure \ref{fig4}).

\begin{figure}

\ifx\JPicScale\undefined\def\JPicScale{1}\fi
\unitlength \JPicScale mm
\begin{picture}(135,70)(40,30)
\linethickness{0.3mm}
\put(60,40){\line(1,0){73}}
\put(133,40){\vector(1,0){0.12}}
\linethickness{0.3mm}
\put(70,30){\line(0,1){60}}
\put(70,90){\vector(0,1){0.12}}
\put(116.88,68.75){\makebox(0,0)[cc]{$\mathcal P$}}

\put(135,60){\makebox(0,0)[cc]{}}

\put(106.8,82){\makebox(0,0)[cc]{$\mathcal{W}_R$}}

\put(88.6,47.2){\makebox(0,0)[cc]{$C_{\tilde \rho}$}}

\linethickness{0.3mm}
\qbezier(110,81)(101.14,73.7)(96.46,69.64)
\qbezier(96.46,69.64)(91.78,65.58)(90.56,64.12)
\qbezier(90.56,64.12)(89.29,62.64)(88.55,61.04)
\qbezier(88.55,61.04)(87.82,59.45)(87.5,57.5)
\qbezier(87.5,57.5)(87.17,55.56)(87.77,53.53)
\qbezier(87.77,53.53)(88.37,51.5)(90,49.06)
\qbezier(90,49.06)(91.63,46.63)(92.83,44.75)
\qbezier(92.83,44.75)(94.04,42.87)(95,41.25)
\qbezier(95,41.25)(95.98,39.62)(96.81,38.27)
\qbezier(96.81,38.27)(97.63,36.91)(98.44,35.62)
\qbezier(98.44,35.62)(99.25,34.33)(100.08,32.97)
\qbezier(100.08,32.97)(100.9,31.62)(101.88,30)
\linethickness{0.3mm}
\qbezier(70,48.12)(75.53,48.12)(79.36,48.42)
\qbezier(79.36,48.42)(83.19,48.72)(85.94,49.37)
\qbezier(85.94,49.37)(88.69,50.02)(91.69,50.77)
\qbezier(91.69,50.77)(94.7,51.52)(98.44,52.5)
\qbezier(98.44,52.5)(102.17,53.48)(105.1,54.23)
\qbezier(105.1,54.23)(108.03,54.98)(110.62,55.63)
\qbezier(110.62,55.63)(113.21,56.28)(116.21,57.03)
\qbezier(116.21,57.03)(119.22,57.78)(123.12,58.75)
\linethickness{0.1mm}
\qbezier(110,80.8)(101.14,73.5)(96.46,69.44)
\qbezier(96.46,69.44)(91.78,65.38)(90.56,63.92)
\qbezier(90.56,63.92)(89.29,62.43)(88.57,60.75)
\qbezier(88.57,60.75)(87.84,59.06)(87.56,56.92)
\qbezier(87.56,56.92)(87.26,54.78)(87.85,52.95)
\qbezier(87.85,52.95)(88.43,51.11)(90,49.3)
\qbezier(90,49.3)(91.54,47.48)(94.54,46.51)
\qbezier(94.54,46.51)(97.55,45.55)(102.5,45.3)
\qbezier(102.5,45.3)(107.44,45.03)(111.29,45.58)
\qbezier(111.29,45.58)(115.14,46.14)(118.5,47.61)
\qbezier(118.5,47.61)(121.87,49.06)(124.79,50.64)
\qbezier(124.79,50.64)(127.7,52.22)(130.61,54.18)
\linethickness{0.1mm}
\qbezier(70,48.12)(75.53,48.12)(79.36,48.42)
\qbezier(79.36,48.42)(83.2,48.72)(85.94,49.37)
\qbezier(85.94,49.37)(88.7,50.01)(91.48,51.14)
\qbezier(91.48,51.14)(94.26,52.27)(97.51,54.06)
\qbezier(97.51,54.06)(100.75,55.84)(104.13,57.95)
\qbezier(104.13,57.95)(107.51,60.05)(111.57,62.82)
\qbezier(111.57,62.82)(115.65,65.59)(117.53,66.87)
\qbezier(117.53,66.87)(119.4,68.15)(119.38,68.13)
\linethickness{0.3mm}
\put(89.3,50.2){\circle*{0.8}}

\put(80,50.4){\makebox(0,0)[cc]{$P$}}

\put(85.2,51.2){\makebox(0,0)[cc]{$P'$}}

\put(90.8,53.1){\makebox(0,0)[cc]{}}

\put(91.7,49){\makebox(0,0)[cc]{}}

\put(89.3,57.4){\makebox(0,0)[cc]{$Q$}}

\put(90.6,53.2){\makebox(0,0)[cc]{$Q'$}}

\linethickness{0.3mm}
\put(88.1,52.5){\circle*{0.8}}

\linethickness{0.3mm}
\put(85.5,49.3){\circle*{0.8}}

\linethickness{0.3mm}
\put(80.4,48.5){\circle*{0.8}}

\linethickness{0.3mm}
\put(87.4,56.1){\circle*{0.8}}

\end{picture}

\caption{Step in the construction.}
\label{fig4}
\end{figure}
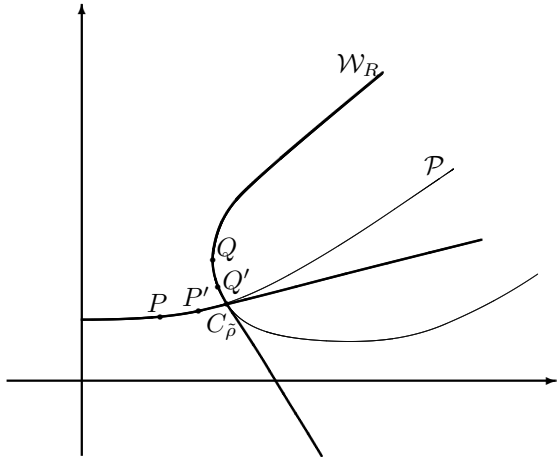

For $\ve>0$ small, the curves $\tilde{\mathcal P}$ an $\tilde{\mathcal W}$ intersect at a point $C_{\tilde\rho}$. 
We choose the angle $\alpha$ of the Scherk surface and the angle $\beta$ of the rotation such that lines $A_\alpha^2$, $A_\alpha^3$ are parallel to the segments $C_{\tilde \rho} Q'$ and $C_{\tilde \rho}P'$ respectively. Note that $\alpha = \alpha_0(1+O(\ve))$ and we did not dislocate the second and third wings of the Scherk surface, so $\alpha$ and $\beta$ do not depend on $\beta_1, \beta_2, \tau_1,$ or $ \tau_2$.

Because we have approximated our original curves $\mathcal W_R$ and $\mathcal P$ up to first order, the new intersection point $C_{\tilde \rho}$ is  at distance $O(\ve^2)$ from $C_{\rho}$. By the same reasoning, the distance from $C_{\tilde \rho}$ to either $P'$ or $Q'$ is $(a +b_{\alpha_0} \cot \alpha_0+ R_{tr} + 10) \ve + O(\ve^2)$. Combining with the estimate on $\alpha$, we have
	\begin{equation}
	\label{eq:CP'}
	|C_{\tilde \rho} Q'|, |C_{\tilde \rho} P'| = (a +b_{\alpha} \cot \alpha+ R_{tr} + 10) \ve + O(\ve^2).
	\end{equation}

We have to adjust the scale so that the dislocated bent Scherk surface $\bar \Sigma$ given in \eqref{eq:bar-sigma} fits around the circle of radius $\tilde \rho$ and so that its second and third asymptotic half-planes contain part of the line segments of $\tilde {\mathcal W}$ and $\tilde {\mathcal P}$ respectively. Considering the image of the $z$-axis under $\ve B_{\ve, \rho_{\ve}} \circ Z_\beta$ would be a mistake because in general, the second and third asymptotic planes do not meet there. Instead, we look at the image under $\ve B_{\ve, \rho_{\ve}} \circ Z_\beta$ of the line $(b_{\alpha}/\sin\alpha, 0,z)$ (see Figure \ref{fig:scherk}) and obtain a circle of radius $\rho'=\rho_{\ve} + \ve \frac{b_{\alpha} \cos \beta}{ \sin{\alpha}}=\rho_{\ve} (1+O(\ve))= \rho(1+O(\ve))$. This is the desired radius.

We had to wrap our Scherk surface around a circle of radius $\rho_{\ve}$ to get an embedded surface, so now we adjust the scale by defining $\lambda_{\ve} = \frac{\tilde \rho}{\rho'} = 1 + O(\ve^2)$.
This function is not continuous in $\ve$ and the jumps occur when the number of periods of the desingularizing surface increases. We take $\lambda_{\ve} \ve \bar \Sigma$ and shift it vertically so that the asymptotic cone associated with the second wing matches the cone generated by the straight part of $\tilde{\mathcal W}$ on an open set.  The cone associated to the third wing aligns automatically with the cone of $\tilde{\mathcal P}$ by our choice of $\alpha$ and $\beta$. We record the amount of vertical displacement with the constant $d_{\ve}$ and denote the shifted surface by $\lambda_{\ve}\ve \bar \Sigma^{\uparrow}$.

We now work on the large scale of equation \eqref{eq2}. The scaled surface $\lambda_{\ve} \bar \Sigma^{\uparrow}$ has a boundary at  $s=R_{tr}+10$. We wish extend the underlying asymptotic cones with pieces of catenoidal ends that will match the cones in a $C^1$ manner. We take the curve $\gamma_2$ to be just a parametrization of $\mathcal W_R$ and $\gamma_3$ to be the curve generating the inner part of $\mathcal P$.  Thanks to the estimate \eqref{eq:CP'}, the curves $\gamma_2$ and $\gamma_3$, the $C^1$ matching will occur for some $s \in (R_{tr} + 9, R_{tr} +11)$ if $\ve$ is small enough. For $i=1$ and $4$, we consider the circle on the $i$th asymptotic cone corresponding to $s=R_{tr}+10$ and the tangent unit vector to the cone perpendicular to this circle, pointing away from the core. This gives us an initial position and velocity for the unique curve  $\gamma_i =  (\gamma_{i,1}, \gamma_{i,3}):[0, \infty) \to \R^2$, $i=1,4$, generating a rotationally self-translating surface. The curves $\gamma_1$ and $\gamma_4$ are perturbations of sections of the original paraboloidal and lower catenoidal ends. We can assume without loss of generality that the $\gamma_i$'s are parametrized by arclength.

The surface $\mathcal M$, which is an approximate solution to \eqref{eq2} is defined in the following way. In the region $s\geq R_{tr}+20 $ the $i$-th wing  of $\mathcal M$ is taken as a graph over the rotationally symmetric surface generated by $\gamma_i$.
More precisely, let
$$
X_i(t,\theta):=\ve^{-1}(
\gamma_{i,1}(\ve p t) \cos (\ve \theta),
\gamma_{i,1}(\ve p t) \sin(\ve \theta),
\gamma_{i,3}(\ve p t)),
$$
for $t\geq 0$, $\theta \in [0,2\pi/\ve]$, where $p=\gamma_1(0)$. The factor $p$ is to make the parametrization conformal at $t=0$, which we will take to be $s = R_{tr}+20$. The unit normal vector is
$$
\nu (t,\theta) = (
-\gamma_{i,3}'(\ve p t) \cos (\ve \theta),
-\gamma_{i,3}'(\ve p t) \sin(\ve \theta),
\gamma_{i,1}'(\ve p t)).
$$
We parametrize  the $i$-th wing of $\mathcal M$ in the region $R_{tr}+20\leq s \leq 5 \delta_s/\ve$ by
$$
(s,\theta) \mapsto X_i(s-(R_{tr}+20),\theta) + u(s,\theta) \nu(s-(R_{tr}+20),\theta)
$$
where the function $u$ is given by
$$
u(s, \theta) = p  f_{\alpha} ( s, \theta) \eta(\ve  s),
$$
where $f_{\alpha}$ is as in Lemma \ref{lem:scherk-properties} and $\eta$ is a cut-off function satisfying $\eta(s) =1$ for $s \leq 4 \delta_s $ and $\eta(s) = 0$ for $s \geq 5 \delta_s$. For $s\geq 5 \delta_s / \ve$ the surface $\mathcal M$ is the union of the four pieces of rotationally symmetric self-translating surfaces generated by the graphs of $\gamma_i$'s.

In the region $R_{tr}+9\leq s \leq R_{tr}+20$ we smoothly interpolate the previous parametrization for $s\geq R_{tr}+20$ with the corresponding one for $s\leq R_{tr}+9$, where the surface can be written as the graph of a function over a cone.

\begin{lemma} There exists a constant $\delta_p>0$ depending only on $\delta_{\alpha}$ so that the surface $\mathcal M [\ve, \beta_1, \beta_4, \tau_1, \tau_4]$ is embedded for $\beta_1, \beta_4, \tau_1, \tau_4 \in (-\delta_p, \delta_p)$ and $\ve \in (0, \delta_p)$. Moreover, $\mathcal M$ depends smoothly on $\beta_1, \beta_4, \tau_1$, and $\tau_4$. It also depends smoothly on $\ve$, except on a countable set.
\end{lemma}

\begin{proof}
The only point that needs an argument is that the unbounded ends do not intersect for all $\beta_1, \beta_4, \tau_1, \tau_4 $ small. This is a consequence of the following observation: consider two solutions $\varphi_i=\varphi_i(r)$, $i=1,2$ of \eqref{eq varphi} defined for all $r\geq R$ with initial conditions
$$
\varphi_i(R) = \varphi_{i,0} , \quad
\varphi_{1,0} > \varphi_{2,0} .
$$
By uniqueness of solutions to ODE,
$$
\varphi_1(r) > \varphi_2(r) \quad \forall r\geq R.
$$
Let $F_i(r)$ be such that $F_i'=\varphi_i$ and $F_1(0)=F_2(0)$.
It follows that
$$
F_1(r)>F_2(r) + m\quad\forall r\geq R ,
$$
with $m>0$ and $m$ remains positive if  $\varphi_{1,0} - \varphi_{2,0} $ remains positive.
\end{proof}

 %
%

\subsection{Summary of notation and terminology}

We start by recalling the roles of the different parameters:
\begin{itemize}
\item $\ve$ controls the overall scale and the error  in the construction.
\item $\rho_{\ve}$ is  the closest number in $\ve \Z$ to $\rho$, which is the radius of the intersection of the paraboloid and catenoid in the original scale.
\item $\alpha$ is the angle associated to the original Scherk surface (see Figure \ref{fig:scherk}).
\item $\beta_1$ and $\beta_4$ are the angles of rotation of the first and fourth wings respectively.
\item $\tau_1$ and $\tau_4$ are the amount by which the first and fourth wings are translated (along the normal of the asymptotic plane to the respective wing).
\end{itemize}
Then we have the different scaled and bent Scherk surfaces:
\begin{itemize}
\item $\Sigma(\alpha)$ is the original minimal Scherk surface given by \eqref{scherk}. Its wings $W^i(\alpha)$ are asymptotic to the half-planes $A_\alpha^i(H^+)$.
\item $\Sigma'=\Sigma'[\alpha, \beta_1, \beta_4, \tau_1, \tau_4]$ is $\Sigma(\alpha)$ with dislocations (see Figure \ref{fig dislocations}).
\item $\bar \Sigma := B_{\ve, \rho_{\ve}} \circ Z_\beta (\bar \Sigma')$ is  $\Sigma'$ wrapped around a circle of radius $\rho_{\ve}/\ve$.
\item $\lambda_{\ve} \ve \bar \Sigma^{\uparrow}$ is the previous Scherk surface  scaled by a factor $\ve \lambda_\ve =  \ve (1+O(\ve^2))$ and shifted by $d_{\ve}$ in the $\vec e_z$ direction so that it fits the configuration in Figure \ref{fig4}.
\end{itemize}
Finally, let us recall the names of the different parts of the initial approximation  $\mathcal M=\mathcal M[\ve, \beta_1, \beta_4, \tau_1, \tau_4]$:
\begin{itemize}
\item The inner core is where the handles are.
\item The middle core is where we perform all the dislocations. It corresponds to the region $0 \leq s \leq R_{tr}+10$.

\item The region $ s \in (R_{tr}+9, R_{tr}+20)$ is a transition region called the outer core. Note that there is a change of scale in this region from $\lambda_{\ve} \ve$ to $\ve$ but this will not create too much error as the switch is contained in a bounded region with $s$ small compared to $\ve^{-1}$.


\item The core is the union of the inner, middle, and outer core.

\item
In the region $\{ s \in (R_{tr}+20, 4 \delta_s/\ve\}$,  $\mathcal M$ is the graph of $f_{\alpha}$ over the rotationally symmetric surface generated by $\gamma_i$, $i=1,\ldots 4$.


\item The region $\{ s \in (4 \delta_s/\ve, 5\delta_s/\ve)\}$ is a transition region where we cut-off the graph of $f_{\alpha}$. We have to do it far enough so that the function $f_{\alpha}$ is small and the length of the interval has to be long enough so we don't create too much error.
	
\item The wings are defined to be the region $\{ s \in (0, 5 \delta_s/\ve)\}$.

\item For $\{ s \geq 5 \delta_s /\ve\}$, $\mathcal M$ is just the rotationally symmetric surface generated by the curves $\gamma_i$'s.
\end{itemize}

\section{Geometric computations}
\label{sec Geometric computations}

In this section, we perform computations related to the ansatz $\mathcal M$ . In particular, we prove Proposition~\ref{prop error}, which gives the error. We also use these computations to prove Proposition~\ref{prop quadratic} for the quadratic terms appearing in the expansion of the mean curvature and normal vector.


\subsection{Perturbation by normal graphs}
\label{sec:prelim}
Consider a surface $M$ immersed in $\R^3$ with local parametrization of class $C^2$:
$$
X:U \subset \R^2 \to M, \qquad
X = X(x_1,x_2).
$$
We use the notation
$$
e_i = \partial_{i} X = \partial_{x_i} X
$$
for tangent vectors and we take the normal unit vector to be
$$
\nu = \frac{e_1\times e_2}{|e_1\times e_2|} ,
$$
where $\times$ is the cross product in $\R^3$.
The metric of $M$ is denoted by
$$
g_{ij} =\langle e_i ,e_j \rangle ,
$$
and its inverse by $g^{ij}$, where $\langle \cdot, \cdot \rangle$ is the standard inner product in $\R^3$.
We recall that
$$
\partial_i \nu = - A_i^{\ j} e_j
$$
where we use Einstein's convention of summation over repeated indices, and  $A_i^{\ j}$ is the second fundamental form, which can be computed as
$$
A_i^{\ j} = A_{ik} g^{kj} ,
\quad
A_{ij} = \langle  X_{ij}, \nu \rangle
= - \langle e_i, \partial_j \nu \rangle .
$$
The mean curvature $H$ of $M$ can be computed in the following way
\begin{align*}
H = \text{trace of $A$} = A_1^{\ 1}+A_2^{\ 2} =
g^{ij}
\langle X_{ij} , \nu \rangle .
\end{align*}

Consider a function $u \in C^2(M)$.
We write
$$
u_i = \partial_{x_i} u, \quad u_{ij} = \partial_{x_i x_j} u .
$$
Let
$
\tilde X = X + \nu u
$
be the graph of $u$ over $M$, we then have
\begin{align*}
\tilde e_i  = \partial_{x_i} \tilde X = e_i + u_i \nu - u A_i^{\ k} e_k
\end{align*}
and
\begin{align*}
\tilde g_{ij}
& =\langle \tilde e_i , \tilde e_j \rangle =
g_{ij} - u ( A_j^{\ l} g_{il} + A_i^{\ k} g_{kj})
+ u^2 A_i^{\ k} A_j^{\ l} g_{kl}
+ u_i u_j.
\end{align*}
We compute the cross product
\begin{multline*}
\tilde e_1 \times \tilde e_2
=
e_1 \times e_2 (1 - u H  - u^2 G)
+ u_1 \nu \times e_2
+ u_2 e_1 \times \nu
\\
- u u_1 A_2^{\ l} \nu \times e_l
- u u_2 A_1^{\ k} e_k \times \nu,
\end{multline*}
where $G = A_1^{\ 1} A_2^{\ 2} -A_1^{\ 2} A_2^{\ 1}  $ is the Gauss curvature.
We also compute the derivatives of $\tilde X$:
\begin{align*}
\tilde X_{ij}
&=
\partial_j( e_i + \nu_i u + \nu u_i)
\\
&=
e_{ij}
+ u_{ij} \nu
- u_{i} A_j^{\ l}  e_l
- u_j A_i^{\ k} e_k
- u (A_i^{\ k})_j e_k
- u A_i^{\ k} e_{kj}
\end{align*}
The mean curvature of $M_u$ is given by
\begin{align*}
\tilde H =
\tilde g^{ij}
\langle
\tilde X_{ij} ,
\frac{\tilde e_1 \times \tilde e_2}{|\tilde e_1 \times \tilde e_2|}.
\rangle
\end{align*}
Explicitly, the scalar product is
\begin{multline*}
\langle
\tilde e_1 \times \tilde e_2
,
\tilde X_{ij}
\rangle
\\
=
\langle
e_1 \times e_2 (1 - u H  - u^2 G)
+ u_1 \nu \times e_2
+ u_2 e_1 \times \nu
- u u_1 A_2^{\ l} \nu \times e_l
- u u_2 A_1^{\ k} e_k \times \nu ,
\\
e_{ij}
+ u_{ij} \nu
- u_{i} A_j^{\ l}  e_l
- u_j A_i^{\ k} e_k
- u (A_i^{\ k})_j e_k
- u A_i^{\ k} e_{kj}
\rangle
\end{multline*}
with $
\det \tilde g = | \tilde e_1 \times \tilde e_2|^2
$
and
\begin{align*}
\tilde g^{-1}
= \frac{1}{\det \tilde g}
\left[
\begin{matrix}
\tilde g_{22} &  - \tilde g_{12} \\
-\tilde g_{12} &  \tilde g_{11}
\end{matrix}
\right] .
\end{align*}



\subsection{Norms of tensors}
We work with the following norms, which are independent of coordinates.

\begin{definition} The pointwise norm of a tensor $T_{a_1 \ldots a_r}^{\ \ \ \ \ \ b_1\ldots b_s}$ is given by
\[
| T |^2 :=  T_{a_1 \ldots a_r}^{\ \ \ \ \ \ b_1\ldots b_s}  T_{c_1 \ldots c_r}^{\ \ \ \ \ \ d_1\ldots d_s} g^{a_1c_1} \cdots g^{a_r c_r} g_{b_1 d_1} \cdots g_{b_s d_s} ,
\]
with summation over repeated indices.
\end{definition}
We use the notation $\nabla_i $ to denote the covariant derivative with respect to $\frac{\partial}{\partial x_i}$.
In the case where the metric is diagonal, the norms of the gradient and Hessian of a function $u$ are
\begin{gather*}
| \nabla u |^2 = |u_i u_j g^{ij}| = \frac{(u_1)^2}{g_{11}} + \frac{(u_2)^2}{g_{22}},\\
|\nabla^2 u|^2 = (g^{11} \nabla_{11} u )^2 + (g^{22}\nabla_{22} u )^2 + 2 g^{11}g^{22} (\nabla_{12} u)^2.
\end{gather*}
For the second fundamental form, we have
\begin{gather*}
| A |^2 = (g^{11}A_{11} )^2  + (g^{22}A_{22} )^2 + 2 (g^{11} g^{22} A_{12} A_{12} ), \\
|\nabla A|^2 = \sum_{i,j,k =1}^2  (\nabla_{i} A_j^{\ k} )^2 g^{ii} g^{jj} g_{kk}.
\end{gather*}

\subsection{Geometry of rotationally symmetric self-translating surfaces}
We compute various geometric quantities attached to the parametrization $X_{\ve}$ given in \eqref{def X}. We use $\partial_i $ or $(\ )_i$ to denote regular differentiation with respect to the variables $s$ ($i=1$) or $\theta$ ($i=2$).
Let $\{e_1, e_2\}$ be the tangent vectors to $S_{\ve}$ given by
\begin{align}
\label{e1}
e_1 & = \partial_1 X_{\ve}= p (\gamma'_1(\ve ps) \cos (\ve \theta), \gamma'_1 (\ve ps) \sin(\ve \theta), \gamma'_3 (\ve ps)),\\
\label{e2}
e_2 & =  \partial_2 X_{\ve}=(- \gamma_1 (\ve ps) \sin(\ve \theta), \gamma_1 (\ve ps) \cos(\ve\theta), 0),
\end{align}
We recall that $(\gamma_1(s), \gamma_3(s))$ is parametrized by arc length and that $p:=\gamma_1(0)>0$.
The associated metric is then
\begin{equation}
\label{eq:metric}
g_{11} = p^2, \quad g_{12} = g_{21} = 0, \quad g_{22} = \gamma_1^2.
\end{equation}
The only nonzero Christoffel symbols are
\begin{gather}
\Gamma_{22,1} = - \ve p \gamma_1 \gamma_1', \quad \Gamma_{12,2}=\Gamma_{21,2} = \ve p  \gamma_1 \gamma_1', \notag \\
\Gamma_{22}^1  = -\ve\frac{\gamma_1}{p}\gamma_1', \quad \Gamma_{12}^2 = \Gamma_{21}^2 = \ve\frac{p}{\gamma_1}{\gamma_1'}. \label{eq:christoffel}
\end{gather}
Using $A_{ij} = \langle \partial_i e_j , \nu \rangle$, we obtain the coordinates of the second fundamental form
\begin{align}
\label{eq:A}
\left\{
\begin{aligned}
A_{11} &= - \ve p^2 (-\gamma_1 '' \gamma_3' + \gamma_1' \gamma_3''), &
A_{12} &=0, &
A_{22} &= \ve \gamma_1 \gamma_3' \\
A_1^{\ 1} &= - \ve (-\gamma_1'' \gamma_3'+ \gamma_1'\gamma_3 '')&
A_{1}^{\ 2} &= A_2^{\ 1} = 0 &
A_2^{\ 2} &= \ve \gamma_1^{-1}\gamma_3'
\end{aligned}
\right.
\end{align}
where all the functions are taken at $\ve ps$.

The following proposition is an immediate corollary of \eqref{eq:A} and the growth of $\gamma_1, \gamma_3$ given in Lemma \ref{lemma:order-gamma}.
\begin{prop}
\label{prop:boundA}
In the coordinates given by $X_{\ve}$, the second fundamental form $A$ and Christoffel symbols on $S_{\ve}$ satisfy
\begin{gather}
\label{eq:boundA}
\left| \frac{d ^k }{ds ^k} A_{i}^{\ j} (\ve p s) \right| \leq C \ve^{k+1},\\
\label{eq:boundGamma}
|\Gamma_{ij}^k| \leq C \ve.
\end{gather}
\end{prop}


\begin{prop}
\label{prop:normA}
We have
\begin{equation}
	\label{eq:normA}
	 | \nabla^k A | \leq C \ve^{k+1}, \quad k=0,1,2 .
	\end{equation}
 \end{prop}

\begin{proof}
The fact that $|A|^2 \leq C \ve^2$ is straightforward from \eqref{eq:A}, \eqref{eq:metric} and Proposition \ref{prop:boundA}. For the first covariant derivative of $A$, we recall
$
\nabla_k A_i^{\ j} = \partial_{k} A_1^{\ j} - \Gamma_{ki}^l A_l^{\ j} + \Gamma_{kl}^j A_i^{\ l}$ with implied summation over $l=1,2$. Upon inspection, $\nabla_1 A_1^{\ 2}$ vanishes. If  $(i, j, k) \ne (1,2,1)$, the quantity $\sqrt{g^{kk} g^{ii} g_{jj}} $ is bounded, therefore
\begin{align*}
\sqrt{g^{kk} g^{ii} g_{jj}} |\nabla_k A_i^{\ k} | & \leq C(|\partial_{k} A_1^{\ j}| + | \Gamma_{ki}^l A_l^{\ j}| + |\Gamma_{kl}^j A_i^{\ l}|)\\
& \leq C \ve^2
\end{align*}
and the estimate for $|\nabla A|$ is proved.

For the second covariant derivative, we argue similarly. Recall that
$
\nabla^2_{lk} A_i^{\ j} = \partial_l(\nabla_k A_i^{\ j}) - \Gamma_{lk}^m \nabla_m A_i^{\ j} - \Gamma_{li}^m \nabla_k A_m^{\ j} + \Gamma_{lm}^j \nabla_{k} A_i^{\ m}.
$
Note that $\nabla_{11}^2 A_1^{\ 2} =0$. As before, if $(i, j, k, m) \ne (1,2,1,1)$, the product $\sqrt{g^{kk} g^{ii} g_{jj} g^{ll}}$ is bounded and we prove $|\nabla^2 A|\leq C \ve^3$ by combining \eqref{eq:boundA}, \eqref{eq:boundGamma}, and \eqref{eq:normA} for $k=0,1$.
 \end{proof}

\subsection{Computation of the error}

\begin{proof}[Proof of Proposition~\ref{prop error}]

The initial approximation $\mathcal M$ consists of three parts. The core of $\mathcal M$ is a smooth perturbation of a compact piece of Scherk minimal surface, when we consider one period only. This introduces curvatures of order $\ve$, together with some dislocations, so the statement of Proposition~\ref{prop error} follows directly for the error restricted to this part.

The region $s\geq 5 \delta_s/\ve$ of $\mathcal M$ is a rotationally symmetric self-translating surface, so that $E$ is zero here.



Where $ R_{tr}+20 \leq s \leq 5\delta_s/\ve $, the surface $\mathcal M$ is a graph  over a self-translating rotationally symmetric surface $S_\ve$. We parametrize $S_\ve$ with $X_{\ve}$ defined in \eqref{def X}, which we will write for convenience as $X$, so that
$$
X(s,\theta):=\ve^{-1}(\gamma_1(\ve p s) \cos (\ve \theta), \gamma_1(\ve p s) \sin(\ve \theta), \gamma_3(\ve p s)),
$$
where $p=\gamma_1(0)$, and $s\in [0,\infty)$, $\theta\in [0,2\pi]$.

Then $\mathcal M$ in this region is parametrized by
\begin{equation}
\label{eq:tildeX}
(s,\theta) \mapsto \tilde X (s,\theta) = X(s,\theta) + u(s,\theta) \nu .
\end{equation}
where the function $u$ is given by
\begin{equation}
\label{eq:def-u}
u(s, \theta) = p  f_{\alpha} ( s, \theta) \eta(\ve  s),
\end{equation}
with $f_{\alpha}$ given in Lemma \ref{lem:scherk-properties} 
and $\eta$ is a cut-off function satisfying $\eta(s) =1$ for $s \leq 4 \delta_s $ and $\eta(s) = 0$ for $s \geq 5 \delta_s$. 
We observe that
\begin{equation}
\label{eq:Dku}
| e^{s} \partial^k u (s, \theta) | \leq C e^{-a}, \quad k=0, \ldots, 5,
\end{equation}
and that $f_\alpha(s,\theta)$ satisfies the minimal surface equation \eqref{min surf eq}.

In the rest of this proof,  $g, A, H, \nu$ denote the metric, second fundamental form, mean curvature and Gauss map of the rotationally symmetric surface $S_\ve$ and  $\tilde g, \tilde A, \tilde H, \tilde \nu$ the ones of $\mathcal M$ given by the parametrization above.

In the rest of this section, we shall work in the region  $s \leq 5\delta_s/\ve$. Following the calculations of Section~\ref{sec:prelim} with the parametrization \eqref{eq:tildeX} and using \eqref{eq:normA} we have
\begin{align*}
\tilde e_1 & = (1- u A_1^{\ 1}) e_1  - u A_1^{\ 2} e_2 +  u_1 \nu
= e_1 + \nu u_1 + O(\ve e^{-s})
, \\
\tilde e_2 & =  - u A_2^{\ 1} e_1 +(1- u A_2^{\ 2}) e_2 + u_2 \nu
= e_2 + \nu u_2 + O(\ve e^{-s})
,
\end{align*}
where $O(\ve e^{-s})$ is in the $C^2$ sense on the region $ s \leq 5\delta_s/\ve$ as $\ve \to 0$, and $e_1$, $e_2$ are given in \eqref{e1}, \eqref{e2}.
We compute the metric $\tilde g_{ij}$:
\begin{align}
\label{eq:tildeg}
\tilde g_{ij}
& = g_{ij} - 2 u A_{ij} + u^2 A_i^{\ l} A_j^{\ m} g_{lm} + u_i  u_j = g_{ij} 	+ u_i u_j + O(\ve e^{-s}).
\end{align}
Using \eqref{eq:metric}, we have
$$
\tilde g^{-1} =
\frac{1}{\det(\tilde g) }
\left(
\begin{matrix}
 \gamma_1^2  + u_2^2 & -u_1 u_2 \\
 -u_1 u_2 & p^2 + u_1^2
\end{matrix}
\right) +O(\ve e^{-s}),
$$
where $\gamma_1$ is evaluated at $p \ve s$, and
\begin{align}
\label{det tilde g}
\det(\tilde g) = p^2 \gamma_1^2 + \gamma_1^2 u_1^2+  p^2 u_2^2 + O(\ve  e^{-s}) .
\end{align}

For the normal direction, we recall that $e_1$ and $e_2$ are orthogonal and obtain  	
\begin{align}
\label{cross1}
\tilde e_1 \times \tilde e_2
&=
\det (g)^{1/2}  \nu - u_1 e_1
- u_{2} e_2 + O(\ve e^{-s}). 
\end{align}
Next, we compute
\begin{align}
\tilde X_{ij} = e_{ij} +  u_{ij} \nu + O(\ve e^{-s})
\end{align}
and
\begin{align*}
\tilde A_{ij} = \langle \tilde X_{ij},\tilde \nu\rangle
= \frac{\det(g)^{1/2}}{\det(\tilde g)^{1/2}} ( A_{ij}  + u_{ij})
+ O(\ve e^{-s}) .
\end{align*}
Since $
\frac{\det (g)}{\det(\tilde g)} = 1 + O(e^{-s})
$
and $A_{ij}$ are $O(\ve)$,
we get
\begin{align}
\label{expansion Aij}
\tilde A_{ij}
= A_{ij}  + u_{ij} + O(\ve e^{-s}) .
\end{align}
With this, the second fundamental form can be expressed as
\begin{align*}
\tilde A_1^{\ 1}
&=
\frac{1}{\det(\tilde g)}
\left[
( A_{11} + u_{11})(\gamma_1^1+u_2^2)
-( A_{12}  + u_{12})u_1u_2
\right]+ O(\ve e^{-s})
\\
\tilde A_2^{\ 2}
&=
\frac{1}{\det(\tilde g)}
\left[
- ( A_{21} + u_{21}) u_1 u_2
+ ( A_{22} + u_{22}) (p^2 + u_1^2)
\right]+ O(\ve e^{-s}) .
\end{align*}
Therefore
\begin{multline*}
\tilde H
= \frac{1}{\det(\tilde g)}
\left[
A_{11} \gamma_1^2 + A_{22} p^2 + u_{11}(\gamma_1^2 + u_2^2) - 2 u_{12} u_1 u_2 + u_{22}(p^2+u_1^2)
\right] \\
+O(\ve e^{-s}) .
\end{multline*}
Let
$
\bar u = \frac{1}{p} u = f_{\alpha} (s,\theta) \eta(\ve s).
$
 We recall that
$
H = \frac{1}{\det(g)}\left[ A_{11} \gamma_1^2  + A_{22} p^2 \right]
$
and expand $\gamma_1(p \ve s) = p + O(\ve s)$ in the region $s \leq 5\delta_s/\ve$ to get
\begin{align*}
\tilde H
&=
 H
+
\frac{ p^3}{\det(\tilde g)}
\left[
\bar u_{11}(1+ \bar u_2^2)
-2 \bar  u_{12}\bar u_1 \bar u_2
+
\bar  u_{22}(1+\bar u_1^2)
\right] + O(\ve s e^{-s})
\end{align*}
Because $\eta(\ve s)=1$ for $s \leq 4 \delta_s/\ve$ and $f_{\alpha}$ satisfies \eqref{min surf eq},  we actually  have
$$\bar u_{11}(1+ \bar u_2^2)
-2 \bar  u_{12}\bar u_1 \bar u_2
+
\bar  u_{22}(1+\bar u_1^2)
=0$$
 in this region. Thus we obtain
$$
\tilde H = H + O(\ve s e^{-s}), \quad s \leq 4 \delta_s/\ve.
$$
Also in this region, from \eqref{det tilde g}, \eqref{cross1} and $|\tilde e_1 \times \tilde e_2| = \det(\tilde g)^{1/2}$ we also have $\ve \tilde \nu \cdot e_z = \ve \nu \cdot e_z + O(\ve e^{-s})$ and since $S_\ve$ is a self translating surface, we get
\begin{align*}
\tilde H - \ve \tilde \nu \cdot e_z = O(\ve e^{-\gamma s}).
\end{align*}
The same estimate holds for derivatives of  $\tilde H - \ve \tilde \nu \cdot e_z$ (all the parametrizations here are smooth), which implies the corresponding estimate in the weighted $C^\alpha$ norm.

When $ 4 \delta_s/\ve  \leq s \leq  5 \delta_s/\ve  $, we have $|\partial^k u(s) | \leq C e^{-s} \leq \ve e^{-\gamma s}$ for $\ve$ small enough, so
$$
\tilde H = H + O(\ve e^{-\gamma s}) , \quad
\ve \tilde \nu \cdot e_z = \ve  \nu \cdot e_z + O(\ve e^{-\gamma s}) ,
$$
and  the desired estimate holds.
%
%
%
%
\end{proof}

\subsection{Estimate of the Jacobi operator}


Here we use the following notation:
The metric, Christoffel symbols, and second fundamental forms on a rotationally symmetric piece of self-translating surface are denoted by $g$, $\Gamma_{ij}^k$, and $A$, the corresponding quantities for the ansatz are $g_{\mathcal M}$ , $\Gamma_{ij,\mathcal M}^k$, and $A_{\mathcal M}$ while the ones on the corresponding original Scherk surface are $g_{\Sigma}$, $\Gamma_{ij, \Sigma}^k$,  and $A_{\Sigma}$. For short, we write $\Delta_{g_{\mathcal M}}=\Delta_{\mathcal M}$, $\Delta_{g_{\Sigma}}=\Delta_{\Sigma}$.
In the following proposition, the operators on $\mathcal M$ (the left-hand side) are pulled back to $\Sigma$ using the transformations of Section~\ref{sec:construction of M}.
\begin{prop}
\label{prop:linearopclose}
For $s \leq 5 \delta_s/\ve$, we have
\[
\Delta_{\mathcal M}+ | A_{\mathcal M}|^2 + \ve \langle e_z ,\nabla_{g_{\mathcal M}}  \rangle = \Delta_{\Sigma} + |A_{\Sigma}|^2 + L'
\]
where $L' $ is a second order differential operator with coefficients with $C^{1}$ norm bounded by $O(\delta_s + \delta_p + \ve)$.
\end{prop}
\begin{proof} We again divide into several regions.
For $s\leq R_{tr}+20$, $\mathcal M$ is obtained from the Scherk surface by a bending, which introduces terms of order $\ve$ and dislocations of order $\delta_p$, so the estimate for $L'$ here follows.

When $ R_{tr} + 20 \leq s \leq 5\delta_s/\ve$, $\mathcal M$ can be described by the parametrization \eqref{eq:tildeX}. In this region, we express all geometric quantities of $\mathcal M$ and $\Sigma$ as functions of the coordinates $s$ and $\theta$.

We note  that $\Delta_{\mathcal M} - \Delta_{\Sigma}$ is a second order operator with coefficients whose $C^1$ norm can be estimated from the $C^2$ norm of  $g_{\mathcal M} - g_{\Sigma}$.
The ansatz is the graph of $u$ from \eqref{eq:def-u}, so to be consistent we take the Scherk surface parametrized by $(s,\theta) \mapsto (ps,p\theta,p f_\alpha(s,\theta)$.
Then \eqref{eq:tildeg} gives
\[
g_{ij, \Sigma} - g_{ij,\mathcal M}   =  p^2 \delta_{ij} + p^2 f_{\alpha,i} f_{\alpha,j} - (g_{ij}  + u_i  u_j)  + O(\ve e^{-s}) .
\]
We use \eqref{eq:metric}, $\| u \|_{C^3} \leq C$ , and the expansion $\gamma_1(p \ve s) = p + O(\ve s)$ to deduce
\begin{equation}
\label{eq:uniformg}
\| g_{ij, \Sigma}  - \bar g_{ij,\mathcal M} \|_{C^2} \leq C (\delta_s+ \ve) ,
\end{equation}
where the norm is computed over $R_{tr}+20\leq s\leq 5\delta_s/\ve$.
Because the metrics are uniformly equivalent, in order to bound $|A_{\mathcal M}|^2 - |A_{\Sigma}|^2$, it suffices to control $|A_{ij,\mathcal M} (s,\theta) - A_{ij, \Sigma}(s,\theta)|$. By \eqref{expansion Aij},
\[
A_{ij, \Sigma} - A_{ij,\mathcal M} = p f_{ij}  - ( u_{ij} + A_{ij} )
+O(\ve e^{-s}),
\]
where $O(\ve e^{-s})$ is in $C^1$ norm. The functions $u$ and $f$ are equal if $s\leq 4 \delta_s/\ve$ and we have $e^{-s} \leq C \ve$ otherwise. Moreover, $A_{ij} = O (\ve)$ when $s \leq 5\delta_s/\ve$ by \eqref{eq:A}.
Therefore,
\begin{align}
\label{estimate A}
\|A_{ij, \Sigma} - A_{ij,\mathcal M}\|_{C^1} \leq C \ve,
\end{align}
where the norm is over $R_{tr}+20\leq s\leq 5\delta_s/\ve$.

Finally,  the term $\ve \langle e_z ,\nabla_{g_{\mathcal M}} \rangle $ has coefficients of order $\ve$ in $C^1$ norm over $R_{tr}+20\leq s\leq 5\delta_s/\ve$.
\end{proof}

\subsection{Estimates of the quadratic terms}
Here, we prove Proposition~\ref{prop quadratic} for functions defined on the surface $\mathcal M$. We recall from  \eqref{estimate A} and Proposition~\ref{prop:normA}
 that each $|\nabla^i A|$  remains uniformly bounded on $\mathcal M$, for $i=0,1,2$.

Let $Q_1$ be defined by \eqref{expansion mean curvature} and assume $|uA|<1$.
For this computation it is convenient to work with coordinates that are normal at a certain point $x_0 \in \mathcal M$, which means
$$
g_{ij}(x_0) = \delta_{ij}  \quad \text{and}\quad
\partial_k g_{ij}(x_0) = 0.
$$
This implies
$$
\langle X_{ij} , e_k \rangle = 0,
$$
at $x_0$.
Moreover, by a further rotation $\langle X_{12} , \nu \rangle= 0$ at $x_0$ so that $A_1^{\ 2}(x_0)=0$. With these properties, following the computation in Section~\ref{sec:prelim}, we obtain at the point $x_0$:
\begin{align*}
\tilde g^{-1} =
\frac{1}{\det(\tilde g)}
\left[
\begin{matrix}
1 - 2 u A_2^{\ 2} + u^2 (A_2^{\rm 2})^2 + u_2^2 & -u_1 u_2
\\
-u_1 u_2 & 1-2 u A_1^{\ 1} + u^2 (A_1^{\rm 1})^2 + u_1^2
\end{matrix}
\right]
\end{align*}
with
\begin{align*}
\det(\tilde g) = |\tilde e_1 \times \tilde e_2|^2 =
(1-u H - u^2 G)^2 + u_1^2(1-u A_2^{\rm 2})^2 + u_2^2(1-u A_1^{\rm 1})^2 ,
\end{align*}
where $G$ is the Gaussian curvature.

We will use $Q$ to denote different functions of $u,u_i,u_{ij}, x$ with the properties:
\begin{align}
\label{propQ}
\left\{
\begin{aligned}
& \text{$Q$ is $C^\infty$ in  $u,u_i,u_{ij}$,}
\\
& Q(0,0,0,x)=0,
\quad
D_{u} Q(0,0,0,x)=0,
\\
& D_{u_i} Q(0,0,0,x)=0,
\quad
D_{u_{ij}} Q(0,0,0,x)=0,
\\
&\text{\vtop{\noindent \hsize 8cmsecond derivatives with respect to  $u,u_i,u_{ij}$ are bounded by universal functions of  $|A|$ and $|\nabla A|$ for $|u A|<1/2$,}}
\\
&\text{$Q$ is linear in $u_{ij}$.}
\end{aligned}
\right.
\end{align}
Then we can write
\begin{align*}
\tilde g^{-1} =
\frac{1}{\det(\tilde g)}
\left[
\begin{matrix}
1 - 2 u A_2^{\ 2}  & 0
\\
0 & 1-2 u A_1^{\ 1}
\end{matrix}
\right]+Q
\end{align*}
and
\begin{align*}
\det(\tilde g) = 1 - 2 u H + Q.
\end{align*}
Similarly,
\begin{align*}
\langle
\tilde X_{ij} ,
\tilde e_1 \times \tilde e_2
\rangle
= (1-uH) A_i^{\ j}+u_{ij} - u (A_i^{\ j})^2 + Q,
\end{align*}
therefore
\begin{align*}
\tilde H = H + u_{11}+u_{22} + ( (A_1^{\ 1})^2+(A_2^{\ 2})^2) u + Q_1,
\end{align*}
where $Q_1$ satisfies the properties \eqref{propQ}. Let $u$, $v$ be $C^{2,\alpha}$ functions on $\mathcal M$ with $|u A|<1/2$, $|v A|<1/2$.
To simplify notation, let $U(x) = (u(x),\nabla u(x), \nabla^2 u(x))$.
From the properties of $Q_1$ and Taylor's formula
\begin{align*}
& |Q_1( U(x), x)-Q_1( V(x), x)|
 \leq C ( |U(x)|+ |V(x)|) (  |U(x)-V(x)| ).
\end{align*}

To estimate the H\"older norm of $Q_1$, we note that the expression for $\tilde H - \ve \tilde \nu \cdot e_z$ is linear in the second derivative of $u$ and we have $C^1$ bounds on all the other terms.  We have $C^{\alpha}$ bounds on $\nabla_{g} u$ and $C^1$ bounds on everything else ($u$, $\nabla_{g} u$, $A$, and $\nabla_{g} A$).

%
%

\section{The Jacobi equation on Scherk surfaces}
\label{sec:jacobi-on-scherk}

Let $\Sigma = \Sigma(\alpha)$ be the singly periodic Scherk surface defined by \eqref{scherk}.
In this section, we want to solve the problem involving the Jacobi operator on $\Sigma$,
$$
\Delta \phi + |A|^2 \phi = h \quad\text{in } \Sigma ,
$$
where $\Delta$ is the Laplace-Beltrami operator and $A$ is the second fundamental form of $\Sigma$.

We let $s$ and $z$ denote the coordinates on the wings $W^i(\alpha)$, $i=1,\ldots,4$, described  in Lemma~\ref{lem:scherk-properties}.
In the rest of the section, we will work with right-hand sides  $h $  defined on $\Sigma $ and such that  $\| e^{\gamma s} h \|_{L^\infty(\Sigma)}<\infty $ for a fixed $ \gamma \in (0,1)$.

We will work with functions that are $2\pi$ periodic in $z$ and even with respect to $z$, that is, $\phi$ and $h$ satisfy
\begin{align}
\label{sym}
\phi(x,y,z) = \phi(x,y,z+2\pi ), \quad
\phi(x,y,z) = \phi(x,y,-z) \quad\forall (x,y,z)\in\Sigma ,
\end{align}
which is equivalent to symmetry with respect to the planes $z = k \pi$, $k\in \Z$.

We choose the unit normal vector to $\Sigma$ such that
\begin{align}
\label{orient scherk}
\text{$\nu\cdot e_y>0$ on wings 1 and 2, and $\nu\cdot e_y<0$ on wings 3 and 4.}
\end{align}
Because translating the surface $\Sigma$ leaves its mean curvature unchanged, the functions $\nu\cdot e$ are in the kernel of $\Delta + |A|^2$ for any fixed $e\in\R^3$.
Hence $\nu\cdot e_x $, $\nu\cdot e_y$,  and  $\nu\cdot e_z $  are in the kernel of the Jacobi operator. Of these functions,
 $\nu\cdot e_x $ and $\nu\cdot e_y $ satisfy the symmetries \eqref{sym}
and $\nu\cdot e_z $ does not because it is antisymmetric with respect to $z=0$.
We will write
\begin{align}
\label{def z1 z2}
z_1 = \nu\cdot e_x, \quad z_2 = \nu\cdot e_y.
\end{align}

The main results in this section are the following. First, we consider the problem of finding a bounded solutions $\phi$ of
\begin{align}
\label{eq23}
\left\{
\begin{aligned}
&\Delta \phi + |A|^2 \phi = \sum_{i=1}^2 c_i \eta_0 z_i + h \quad \text{in } \Sigma,
\\
&
\phi \text{ satisfies the symmetries  \eqref{sym}}
\end{aligned}
\right.
\end{align}
for which
\begin{align}
\label{orth2}
\int_{\Sigma} \phi \eta_0 z_i = 0,\quad i=1,2,
\end{align}
where $\eta_0\in C^\infty(\Sigma)$ is a smooth function depending only on $x^2+y^2$ such that
\begin{align}
\label{cut0}
\text{$0\leq\eta_0\leq 1$, $\eta_0 = 1 $ on $x^2 + y^2 \leq R_0^2$, and $\eta_0=0$ on $x^2 + y^2 \geq ( R_0+1)^2$,}
\end{align}
where $R_0>1$ is fixed.

\begin{prop}
\label{prop linear1}
Let $0<\gamma<1$.
Let $h $ be defined in $\Sigma$, satisfy the symmetries \eqref{sym}, and
$\| e^{\gamma s} h \|_{L^\infty(\Sigma)} < \infty $.
Then there are unique $c_1,c_2\in\R$ and $\phi\in L^\infty(\Sigma)$ satisfying \eqref{eq23}
 and  \eqref{orth2}. Moreover
\begin{align}
\label{eq27}
|c_1|+|c_2|+\|\phi\|_{L^\infty(\Sigma)}\leq C \|e^{\gamma s} h\|_{L^\infty(\Sigma)}
\end{align}
and
\begin{align}
\label{decay gradient}
| \nabla \phi|\leq C ( \|e^{\gamma s} h\|_{L^\infty(\Sigma)} + \|\phi\|_{L^\infty(\Sigma)}) e^{-\gamma s} .
\end{align}
\end{prop}

If $\phi$ is a bounded solution of \eqref{eq23}, by \eqref{decay gradient} $\phi$ has a limit on each wing, that is, $L_i = \lim_{s\to\infty} \phi(s,z)$ on all wings $i=1,\ldots,4$.
Moreover, these limits define linear functionals of $h$ and we have the estimate $|L_i|\leq C \|e^{\gamma s}h\|_{L^\infty}$.
%
%
%
%
%
%
%
For later consideration, it is desirable to  find a solution to \eqref{eq23} with limit equal to zero on all wings. To achieve this, the right hand side has to satisfy four restrictions, or equivalently has to be projected onto a space of codimension 4. We do this by considering the main terms introduced by the dislocations.
So we consider now the problem
\begin{align}
\label{50}
\left\{
\begin{aligned}
&\Delta \phi + |A|^2 \phi =  h + \sum_{i=1}^2 ( \beta_i w_i' + \tau_i w_i )
\quad \text{in } \Sigma,
\\
&
\phi \text{ satisfies  \eqref{sym}} ,
\end{aligned}
\right.
\end{align}
where the functions $w_i$, $w_i'$ are defined in \eqref{def w}, \eqref{def w prime}.

\begin{prop}
\label{prop linear scherk 4}
Let $0<\gamma<1$ and  $h$ be a function satisfying  $\| e^{\gamma s} h \|_{L^\infty(\Sigma)} < \infty $ and the symmetries \eqref{sym}.
Then if $R_{rot}$ and $R_{tr}$, which are the parameters in the construction associated to the dislocations, are fixed large enough,
there  exist $\beta_i,\tau_i$, $i=1,2$ and $\phi$ a bounded solution of \eqref{50} such that $\phi$ has limit equal to zero on all wings. Moreover $\phi$,  $\beta_i,\tau_i$ depend linearly on $h$ and
\begin{align}
\label{est linear scherk2}
\|e^{\gamma s} \phi\|_{L^\infty} + \sum_{i=1}^2 ( |\beta_i|+|\tau_i|)\leq C \| e^{\gamma s} h \|_{L^\infty}.
\end{align}
\end{prop}

Using this proposition, we fix the parameters $R_{tr}$, $R_{rot}$ of the construction of the initial approximation $\mathcal M$. The following non-degeneracy property of the Jacobi operator is crucial in the proof of the above results and was proved by Montiel and Ros \cite{montiel-ros}.
\begin{prop}
\label{prop nondeg}
Any bounded solution  $\phi$  of
$$
\Delta \phi + |A|^2\phi =0 \quad\text{in } \Sigma
$$
is a linear combination of
$\nu\cdot e_x $, $\nu\cdot e_y $  and  $\nu\cdot e_z $.
\end{prop}

The rest of the section is devoted to prove Propositions~\ref{prop linear1} and \ref{prop linear scherk 4}.
We start by considering the problem
\begin{align}
\label{prob sigmaR}
\left\{
\begin{aligned}
&\Delta \phi + |A|^2 \phi = h \quad \text{in } \Sigma_R
\\
&
\phi = 0 \quad\text{on } \partial \Sigma_R ,
\quad \phi \text{ satisfies  \eqref{sym}}
\end{aligned}
\right.
\end{align}
where $R>0$ is large and
\begin{align}
\label{def SigmaR}
\text{$\Sigma_R $ is the union of the core of $\Sigma$ and $\cup_{i=1}^4 W^i(\alpha) \cap \{ s\leq R \}$.}
\end{align}
In the sequel, we work with $R>>R_0+1$.

\begin{lemma}
\label{lemma-apriori-sigmar}
Let $0<\gamma<1$.
Let $h $ be defined in $\Sigma_R$, satisfy the symmetries \eqref{sym}, and
$\| e^{\gamma s} h \|_{L^\infty(\Sigma)} < \infty$.
Then there are $R_1$, $C$ such that for all $R\geq R_1$ and any solution $\phi$ of \eqref{prob sigmaR} such that
\begin{align}
\label{orth}
\int_{\Sigma_R} \phi \eta_0 z_i = 0\quad i=1,2,
\end{align}
we have
\begin{align}
\label{apriori}
\|\phi\|_{L^\infty(\Sigma_R)} \leq C \|e^{\gamma s} h\|_{L^\infty(\Sigma_R)} .
\end{align}
\end{lemma}
\begin{proof}
We proceed by contradiction and assume that for any positive integer $n$, there are $R_n$, $\phi_n$, $h_n$ such that $R_n\to\infty$ as $n\to\infty$, \eqref{prob sigmaR}, \eqref{orth} hold and
\begin{align}
\label{hyp cont}
\|\phi_n\|_{L^\infty(\Sigma_{R_n})} =1,\quad
\|e^{\gamma s} h_n\|_{L^\infty(\Sigma_{R_n})}  \to 0 \quad\text{as }n\to\infty.
\end{align}
Let us show first that $\phi_n\to0$ uniformly on compact sets of $\Sigma_{R_n}$. Otherwise, up to a subsequence and using standard local estimates for elliptic equations, $\phi_n\to\phi$ uniformly on compact sets of $\Sigma$, where $\phi $ is bounded, not identically zero, and satisfies
$$
\Delta\phi + |A|^2 \phi = 0 \quad\text{in } \Sigma.
$$
By Proposition~\ref{prop nondeg}, $\phi$ is a linear combination of
$\nu\cdot e_x $, $\nu\cdot e_y $  and  $\nu\cdot e_z $. But $\nu\cdot e_z $ is not present in this linear combination by the imposed symmetry \eqref{sym}, so
$$
\phi = c_1 \nu\cdot e_x +  c_2 \nu\cdot e_y
$$
for some $c_1$, $c_2$.
But passing to the limit in \eqref{orth}, we deduce that $\phi$ satisfies \eqref{orth}. This implies that $c_1=c_2=0$.

We note that  $\psi= 1-e^{-\gamma s}$ satisfies
$$
\Delta \psi + |A|^2 \psi \leq - C e^{-\gamma s}
$$
for some $C>0$ and $s\geq s_0$ where $s_0$ is large enough. Using these barriers on each wing, we obtain that $\|\phi_n\|_{L^\infty(\Sigma_{R_n})}\to0$ as $n\to\infty$, which contradicts \eqref{hyp cont}.
\end{proof}

With almost the same argument we can prove the next result.
\begin{lemma}
\label{lemma apriori sigma}
Let $0<\gamma<1$.
Let $h $ be defined in $\Sigma$, satisfy the symmetries \eqref{sym}, and
$\| e^{\gamma s} h \|_{L^\infty(\Sigma)} < \infty $.
Then there is a constant $C$ such that for any bounded solution $\phi$ of
\begin{align}
\label{prob sigma theta}
\left\{
\begin{aligned}
&\Delta \phi + |A|^2 \phi = h \quad \text{in } \Sigma
\\
&
\phi \text{ satisfies  \eqref{sym}}
\end{aligned}
\right.
\end{align}
which also satisfies \eqref{orth2}, we have
\begin{align*}
\|\phi\|_{L^\infty(\Sigma)} \leq C \|e^{\gamma s} h\|_{L^\infty(\Sigma)} .
\end{align*}
\end{lemma}
\begin{proof}
The proof changes only in the last step, when we use maximum principle to prove that $\phi \leq C \psi$. We change slightly the barriers by
$$
\psi+ \delta Z,
$$
where $Z$ is an element in the kernel of the Jacobi operator that grows linearly,
and then take $\delta\to 0$.
\end{proof}

\begin{lemma}
\label{lemma grad decays}
Let $0<\gamma<1$.
Let $h $ be defined in $\Sigma$, satisfy the symmetries \eqref{sym}, and
$\| e^{\gamma s} h \|_{L^\infty(\Sigma)} < \infty $.
Suppose $\phi$ is a bounded solution of \eqref{prob sigma theta}.
Then
$$
| \nabla \phi|\leq C ( \|e^{\gamma s} h\|_{L^\infty(\Sigma)} + \|\phi\|_{L^\infty(\Sigma)}) e^{-\gamma s} .
$$

\end{lemma}
\begin{proof}
Changing variables,
we rewrite the equation on a fixed wing as
\begin{align}
\label{eq02}
\Delta \phi = a_0(s,z) \phi + a_1(s,z) \nabla \phi + a_2(s,z) D^2 \phi + h
\quad\text{in } S
\end{align}
with boundary conditions
$$
\frac{\partial\phi}{\partial z}=0\quad \text{on }
(s_0,\infty) \times \{0,\pi\},
$$
where $S$ is the strip $(s_0,\infty) \times (0,\pi)$.
We write the variables in $S$ as $(s,z)$, $s>s_0$, $z\in (0,\pi)$
and in \eqref{eq02}, $\Delta = \partial_{ss}+\partial_{zz}$.
The functions $a_0$, $a_1$, $a_2$ are smooth and have the decay
$$
|a_i(s,z)|\leq C e^{-s}.
$$
For $T>s_0$, we have
$$
\| \phi \|_{L^2((T,T+5)\times (0,\pi))} \leq C \|\phi\|_{L^\infty(S)} .
$$
Since the coefficients of $a_i$ are small as $s\to\infty$, we have from standard estimates
$$
\| \phi \|_{H^2((T,T+1)\times (0,\pi))} \leq C
( \|\phi\|_{L^\infty(S)} + \|e^{\gamma s} h\|_{L^\infty(S)} ) .
$$
Hence,
\begin{align}
\label{est g}
\|a_0 \phi + a_1 \nabla \phi + a_2 D^2 \phi + h \|_{L^2((T,\infty)\times (0,\pi))} \leq
C (\|\phi\|_{L^\infty(S)} + \|e^{\gamma s} h\|_{L^\infty(S)}) e^{-\gamma T} .
\end{align}
Let  $g=a_0 \phi + a_1 \nabla \phi + a_2 D^2 \phi + h $ and write
$$
\phi(s,z) = \sum_{n=0}^\infty \phi_n(s) \cos(nz)
,\quad
g(s,z) = \sum_{n=0}^\infty g_n(s) \cos(nz) ,
$$
where, for $n\geq 0$,
$
\phi_n(s) = \frac 2\pi\int_0^\pi \phi(s,z) \cos(nz)\, d z$ and $
g_n(s) = \frac 2\pi\int_0^\pi g(s,z) \cos(nz)\, d z.
$
We can write
\begin{align}
\label{phi0}
\phi_0(s) = b_0 + \int_s^\infty(t-s) g_0(t)\, d t,
\end{align}
where $b_0 = \lim_{s\to\infty} \phi_0(s)$. The claim is that
\begin{align}
\label{decay}
\|\phi - b_0\|_{L^2((T,\infty)\times (0,\pi))}\leq
C(\|\phi\|_{L^\infty(S)} + \|e^{\gamma s} h\|_{L^\infty(S)}) e^{-\gamma T}.
\end{align}

To prove this, let
$
\tilde \phi(s,z) =\sum_{n=1}^\infty \phi_n(s) \cos(nz).
$ We claim that, for $T$ large
\begin{align}
\label{exp decay phi tilde}
\|\tilde \phi\|_{L^2((T,\infty)\times (0,\pi) )}\leq C ( \|\phi\|_{L^\infty(S)} + \|e^{\gamma s} h\|_{L^\infty}) e^{-\gamma T} .
\end{align}
Indeed, we have
$$
\phi_n''- n^2 \phi_n = g_n \quad\text{for } s>s_0.
$$
Note that $\phi_n(s)$ is bounded as $s\to\infty$, so that $\phi_n(s)$ must  have the form  (for $n\geq 1$)
$$
\phi_n(s) = d_n e^{-n (s-s_0)} + \phi_{0,n}(s) ,
$$
where
$$
d_n = \phi_n(s_0), \quad
\phi_{0,n}(s) = -
 e^{- n d}\int_{s_0}^s e^{2nt} \int_t^\infty g_n(\tau) e^{-n \tau} \, d \tau \, d t .
$$
Let $\gamma<a<1$ be fixed. By the Cauchy-Schwarz inequality, for $n\geq  1$, we have
\begin{align}
\label{est phi 0n}
|\phi_{0,n}(s)|
\leq
\frac{1}{(4 n(n-a))^{1/2}}
e^{-as}
(
\int_{s_0}^s
e^{2at} \int_t^\infty |g_n(\tau)|^2\,d\tau  \, d t
)^{1/2} .
\end{align}
%
%
Therefore
$$
\|\tilde \phi\|_{L^2((T,\infty)\times (0,\pi) )}^2
\leq C e^{-2T} \sum_{n=1}^\infty d_n^2 +  C\sum_{n=1}^\infty
\int_T^\infty
|\phi_{0,n}(s)|^2 \, d s ,
$$
and using \eqref{est g}, \eqref{est phi 0n}, we deduce \eqref{exp decay phi tilde}.
The estimate above and a similar one for the integral in \eqref{phi0}
give \eqref{decay}.

Note that $\bar\phi = \phi-b_0$ satisfies
$$
\Delta \bar\phi = a_0(s,z) \bar \phi + a_1(s,z) \nabla \bar \phi + a_2(s,z) D^2 \bar \phi + h + a_0(s,z) b_0
\quad\text{in } S .
$$
From standard elliptic estimates, we get
\begin{align*}
\|\bar \phi + |\nabla \bar \phi| \|_{L^\infty((T+1,T+2)\times (0,\pi)) }
& \leq C (
\|\bar \phi \|_{L^2((T,T+3)\times (0,\pi) )}
\\
& \qquad
+ \| h + a_0 b_0 \|_{L^\infty( (T,T+3)\times (0,\pi) )}
)
\\
&\leq
C(\|\phi\|_{L^\infty(S)} + \|h\|_{L^\infty(S)}) e^{-\gamma T}. \qedhere
\end{align*}

\end{proof}

To prove existence, let $\Sigma_R$ be defined by \eqref{def SigmaR} and consider the truncated problem of finding $\phi$ and $c_1,c_2$ such that
\begin{align}
\label{eq32}
\left\{
\begin{aligned}
&\Delta \phi + |A|^2 \phi = \sum_{i=1}^2 c_i \eta_0 z_i + h \quad \text{in } \Sigma_R ,
\\
& \phi =0 \quad \text{on }\partial\Sigma_R  , \\
& \phi \text{ satisfies  \eqref{sym}} .
\end{aligned}
\right.
\end{align}

\begin{lemma}
\label{lemma linear existence truncated}
Let $0<\gamma<1$.
Let $h $ be defined in $\Sigma_R$, satisfy the symmetries \eqref{sym}, and $\| e^{\gamma s} h \|_{L^\infty(\Sigma_R)} < \infty $.
Then there are unique $c_1,c_2\in\R$ and $\phi\in L^\infty(\Sigma_R)$ satisfying \eqref{eq32}
 and  \eqref{orth2}. Moreover,
\begin{align}
\label{eq31}
|c_1|+|c_2|+\|\phi\|_{L^\infty(\Sigma_R)}\leq C \|e^{\gamma s} h\|_{L^\infty(\Sigma_R)} ,
\end{align}
with $C$ independent of $R$.
\end{lemma}
\begin{proof}
We prove first \eqref{eq31}.
Indeed, by Lemma~\ref{lemma-apriori-sigmar},
$$
\|\phi\|_{L^\infty} \leq C ( |c_1|+|c_2|+ \|e^{\gamma s}h\|_{L^\infty}).
$$
Multiplying by $z_j$ and integrating in $\Sigma_R$, we find
\begin{align}
\label{eq29}
\int_{\partial\Sigma_R} \pd{\phi}{\nu}z_j = c_j \int_{\Sigma_R} \eta_0 z_j^2 .
\end{align}
We claim that
\begin{align}
\label{eq28}
\int_{\partial\Sigma_R} \pd{\phi}{\nu}z_j = o(1) ( \|\phi\|_{L^\infty} + \|e^{\gamma s}h\|_{L^\infty})
\end{align}
where $o(1)\to0$ as $R\to\infty$. Combining \eqref{eq28} and \eqref{eq29}, we get \eqref{eq31}.
To prove \eqref{eq28}, we rewrite the equation on the $i$th wing as
$$
L\phi = \sum_{i=1}^2 c_i \eta_0 z_i   + h
\quad\text{in } S_R,
$$
where $ L\phi= \Delta \phi + a_0 \phi + a_1 \nabla \phi + a_2 D^2 \phi$ and $S_R=\{(s,z): 0< s<R, 0<z<\pi\}$. Here $\Delta = \partial_{ss}+\partial_{zz}$ and $a_i(s,z)$ are smooth with $|a_i(s,z)|\leq C e^{-s}$. Let $R_1>>R_0$ and $R>>R_1$. The function $\bar\phi=\left(\frac{R-s}{R-R_1}\right)^\mu$, where $0<\mu<1$ is fixed, satisfies
$$
L\bar\phi \leq - c (\frac{R-s}{R-R_1})^{\mu-2}
$$
in $(R_1,R)\times(0,\pi)$ for some $c>0$, if we take $R_1$ large. By the maximum principle, $|\phi|\leq C ( \|\phi\|_{L^\infty}+ \|e^{\gamma s}h\|_{L^\infty}) \bar \phi$  in $(R_1,R)\times(0,\pi)$. It follows that
$$
|\pd{\phi}{s}(R,z)|\leq \frac{C}{R-R_1} (\|\phi\|_{L^\infty}+ \|e^{\gamma s}h\|_{L^\infty}) ,
$$
and this proves \eqref{eq28}.

For the existence of a solution of \eqref{eq31}, let us define the Hilbert space
$$
H = \left\{ \phi \in H^1(\Sigma_R \cap \{z\in(0,\pi)\}) :
\phi =0\ \text{on } s=R, \
\int_{\Sigma_R\cap \{z\in(0,\pi)\}} \phi \eta_0 z_i = 0, \ i=1,2 \right\}
$$
with the inner product
$$\langle \varphi_1, \varphi_2 \rangle
=\int_{\Sigma_R\cap \{z\in(0,\pi)\}} \nabla \varphi_1 \nabla \varphi_2
$$
where $\eta_0$ is the cut-off function with properties \eqref{cut0}. We consider the following weak form of the equation to find $\phi\in H$:
\begin{align}
\label{var1}
\int_{\Sigma_R\cap \{z\in(0,\pi)\}} \nabla\phi\cdot\nabla \varphi - |A|^2 \phi\varphi = - \int_{\Sigma_R\cap \{z\in(0,\pi)\}} h\varphi
\quad\forall \varphi \in H.
\end{align}
Let $T:H\to H$, $T(\phi)=\psi$ be the linear operator defined by the Riesz theorem from the relation
$$
\langle \psi,\varphi\rangle = \int_{\Sigma_R\cap \{z\in(0,\pi)\}}|A|^2 \phi \varphi \quad\forall \varphi\in H.
$$
Then $T$ is compact and we can formulate \eqref{var1} as  finding $\phi\in H$ such that
\begin{align}
\label{eq30}
\phi = T(\phi) + L_h,
\end{align}
where $L_h(\varphi) = \int_{\Sigma_R\cap \{z\in(0,\pi)\}} h \varphi$. By the Fredholm theorem, this problem is uniquely solvable for any $h$ provided the only solution of $\phi = T(\phi)$ in $H$ is $\phi=0$. This holds by \eqref{eq31}.
\end{proof}

We now turn our attention to the problem on the whole Scherk surface $\Sigma$.

\begin{proof}[Proof of Proposition~\ref{prop linear1}]
First, let us show that, for any bounded solution $\phi,c_1,c_2$ of \eqref{eq23}
satisfying \eqref{orth2}, the estimate \eqref{eq27} is valid.
Indeed, by Lemma~\ref{lemma apriori sigma},
\begin{align}
\label{est1}
\|\phi\|_{L^\infty} \leq C (|c_1|+|c_2|+\|e^{\gamma s} h\|_{L^\infty}).
\end{align}
For $R>>1$, let $\Sigma_R =\Sigma \cap \{(x,y,z): x^2 + y^ 2 \leq R^2 \} $.
Multiplying \eqref{eq23} by $z_j$ and integrating in $\Sigma_R\cap \{z\in(0,\pi)\}$ we have
\begin{align}
\label{eq24}
\int_{\partial\Sigma_R\cap \{z\in(0,\pi)\}} \pd{\phi}{\nu}z_j = c_j \int_{\Sigma_R\cap \{z\in(0,\pi)\}} \eta_0 z_j^2  + \int_{\Sigma_R\cap \{z\in(0,\pi)\}} h z_j
\end{align}
because $\int_{\Sigma_R\cap \{z\in(0,\pi)\}} \eta_0 z_1 z_2 = 0$ by symmetry. By Lemma~\ref{lemma grad decays},
\begin{align}
\label{eq25}
\int_{\partial \Sigma_R\cap \{z\in(0,\pi)\}} | \pd{\phi}{\nu}z_i  | \leq
C e^{-\gamma R} ( |c_1| + |c_2| +    \|e^{\gamma s }h\|_{L^\infty} + \|\phi\|_{L^\infty} )  .
\end{align}
Combining \eqref{est1}, \eqref{eq24}, and \eqref{eq25}, we deduce
\eqref{eq27} and prove the uniqueness.

%

By Lemma~\ref{lemma linear existence truncated}, for $R$ large, the problem \eqref{eq32} is uniquely solvable, yielding a solution to $\phi_R$ (and constants $c_{1,R}$ and $c_{2,R}$), which remains bounded as $R\to\infty$ by \eqref{eq31}. By standard elliptic estimates, $\phi_R$ converges locally uniformly in $\Sigma$ to a solution of \eqref{eq23}.
Estimate \eqref{decay gradient} follows from Lemma~\ref{lemma grad decays}.
\end{proof}

\begin{proof}[Proof of Proposition~\ref{prop linear scherk 4}]
Given $h$ as stated, let $\phi$, $c_1$, $c_2$ be the solution to \eqref{eq23} provided by Proposition~\ref{prop linear1}.
Let $L_i = \lim_{s\to\infty} \phi(s,z)$ be the limit of $\phi$ on wing $i$.
By adding appropriate multiples of $z_1$ and $z_2$, we can make two of these limits equal to zero. More precisely, by the choice of orientation \eqref{orient scherk} on $\Sigma$ and the definition of $z_1$ and $z_2$ in \eqref{def z1 z2}, we have
$$
\lim_{s\to\infty} z_1 = \begin{cases}
-\sin\alpha&\text{on wing 1},\\
\sin\alpha&\text{on wing 2},\\
\sin\alpha&\text{on wing 3},\\
-\sin\alpha&\text{on wing 4},
\end{cases}
\qquad
\lim_{s\to\infty} z_2  = \begin{cases}
\cos\alpha&\text{on wing 1},\\
\cos\alpha&\text{on wing 2},\\
-\cos\alpha&\text{on wing 3},\\
-\cos\alpha&\text{on wing 4}.
\end{cases}
$$
Let $d_1$, $d_2$ satisfy
$$
d_1 \sin\alpha+d_2 \cos\alpha=L_2, \quad
d_1 \sin\alpha-d_2 \cos\alpha=L_3 .
$$
Then $\tilde \phi = \phi-d_1 z_1 - d_2 z_2 $ satisfies \eqref{eq23} and has limit equal to zero on wings 2 and 3. 

We remark that  we could also achieve limit equal to zero on any two adjacent wings, but not on opposite ones in general.
Also note that if the original $\phi$ satisfies the orthogonality conditions \eqref{orth2}, the new $\tilde \phi$ does not in general.

Let $\tilde \eta_i$, $i=1$ or $i=4$, be smooth cut-off functions on $\Sigma$ such that:
\begin{align*}
& \text{$\tilde \eta_i=1$ on wing $i$ and for  $s\geq R_c+1$},
\\
& \text{$\tilde \eta_i=0$ on wing $i$ for $0\leq s \leq R_c$},
\\
&
\text{$\tilde \eta_i=0$ on the core and the rest of the wings.}
\end{align*}
Here $R_c>0$ is a large constant to be fixed later. Define
$$
\tilde z_1 = \nu\cdot\nu_\alpha ,\quad \tilde z_4 = \nu\cdot (-\sin\alpha,-\cos\alpha,0),
$$
where $\nu_\alpha = (-\sin\alpha,\cos\alpha,0)$ is  the normal vector to the asymptotic plane of wing 1. Note that $\tilde z_1,\tilde z_4$ are in the kernel of $\Delta+|A|^2$, $\tilde z_1\to1$ as $s_1\to\infty$, and  $\tilde z_4\to1$ as $s_4\to\infty$ (they are convenient linear combinations of $z_1,z_2$). Define
$$
\hat \phi = \tilde \phi - L_1  \tilde \eta_1 \tilde z_1 - L_4 \tilde \eta_4 \tilde z_1 .
$$
Then $\hat \phi$ satisfies
\begin{align*}
\Delta \hat \phi + |A|^2 \hat \phi = h + \sum_{i=1}^2 c_i \eta_0 z_i -\sum_{i=1,4} L_i (\Delta + |A|^2)(\tilde \eta_i \tilde z_i)
\quad\text{in }\Sigma
\end{align*}
and $\hat \phi$ has limit equal to zero on all the wings. Moreover $\hat \phi$ satisfies the symmetries \eqref{sym}.

Suppose that $\beta_i,\tau_i$, $i=1,2$, are given and let us consider the function $\hat \phi$ constructed previously with $h$ replaced by $ h + \sum_{i=1}^2 \beta_i w_i' + \tau_i w_i $. This $\hat \phi$ satisfies
\begin{align}
\label{51}
\Delta \hat \phi + |A|^2 \hat \phi = h + \sum_{i=1}^2 c_i \eta_0 z_i - \sum_{i=1,4} L_i (\Delta + |A|^2)(\tilde \eta_i \tilde z_i) + \sum_{i=1}^2 (\beta_i w_i' + \tau_i w_i)
\end{align}
in $\Sigma$,  has the symmetries \eqref{sym} and the limits on all wings are equal to zero. In this construction, $h$,  $\beta_i,\tau_i$, $i=1,2$ are data and $\hat \phi$, $c_i$ ($i=1,2$), $L_i$ ($i=1,4$) are bounded linear functions of these data.

We claim that there is a unique choice of $\beta_i,\tau_i$, $i=1,2$, such that $c_1=c_2=L_1=L_4=0$. To prove this, we test \eqref{51} with functions, which are going to be linear combinations of $z_1,z_2$ and 
$$
z_3 = \nu \cdot   (-y,x,0), \quad  z_4=  \frac{\partial_\alpha S}{|\nabla_{x,y,z} S|} ,
$$
where
$$
S(x,y,z,\alpha)
=\cos^2(\alpha) \cosh \left(\frac{x}{\cos \alpha}\right) - \sin^2 (\alpha) \cosh \left(\frac{y}{\sin \alpha}\right) - \cos (z)
$$
is the function \eqref{scherk} defining the Scherk surface, and $\nabla_{x,y,z} S = (\partial_x S,\partial_y S, \partial_z S)$.
Note that $z_3$ arises from a rotation about the $z$-axis and $z_4$ is generated by the motion in  $\alpha$ of the Scherk surfaces $\Sigma(\alpha)$, so these functions are in the kernel of the Jacobi operator $\Delta +|A|^2 $ and have the symmetries \eqref{sym}.
Also, $z_3$ and $z_4$ have linear growth.

Since $\hat \phi$ has exponential decay,
$$
\int_{\Sigma \cap \{z\in(0,\pi)\} } (\Delta \hat\phi + |A|^2 \hat\phi) z_i =
\int_{\Sigma \cap \{z\in(0,\pi)\} } (\Delta z_i + |A|^2 z_i) \hat\phi  =0 ,
$$
for all $1\leq i \leq 4$, while
$
\int_{\Sigma \cap \{z\in(0,\pi)\} }(RHS) z_i
$
becomes an affine function of the numbers $c_1,c_2,L_1,L_4,\beta_1,\beta_2,\tau_1,\tau_2$.
More precisely, we obtain
$$
0=
\left[
\begin{matrix}
\int h z_1\\\int h z_2\\\int h z_3\\\int h z_4\\
\end{matrix}
\right] + M_1
\left[
\begin{matrix} c_1\\c_2\\L_1\\L_4
\end{matrix}
\right]
+
M_2
\left[
\begin{matrix} \beta_1\\ \beta_2 \\ \tau_1 \\ \tau_2
\end{matrix}
\right].
$$
The claim is that we can choose $\beta_1,\beta_2,\tau_1,\tau_2$ to achieve $c_1=c_2=L_1=L_4=0$. For this, we will verify that $M_1$, $M_2$ are invertible (choosing adequately the parameters).

Note that, by symmetry,
$$
\int_{\Sigma\cap \{z\in(0,\pi)\} } \eta_0 z_i z_j= c \delta_{ij} , \quad \text{for }i,j=1,2,
$$
where $c>0$ is some constant. Also thanks to symmetry, we have for $i=1,2$,
$$
\int_{\Sigma\cap \{z\in(0,\pi)\} } \eta_0 z_i z_3=0 ,\qquad
\int_{\Sigma\cap \{z\in(0,\pi)\} } \eta_0 z_i z_4=0.
$$

Let us compute for $i=1$ or $i=4$, and $j=1,2$:
\begin{align*}
\int_{\Sigma\cap \{z\in(0,\pi)\} } (\Delta+|A|^2)(\tilde \eta_i \tilde z_i) z_j
&= \int_{\Sigma\cap \{z\in(0,\pi)\} } (2\nabla \tilde \eta_i \cdot \nabla \tilde z_i + \tilde z_i \Delta \tilde \eta_i) z_j \\
&= \int_{\Sigma\cap \{z\in(0,\pi)\} } \nabla \cdot (\tilde z_i^2 \nabla \tilde \eta_i) \frac{z_j}{\tilde z_i}\\
&= -\int_{\Sigma\cap \{z\in(0,\pi)\} } \tilde z_i^2 \nabla \tilde \eta_i \nabla \left(\frac{z_j}{\tilde z_i}\right) = O ( e^{-R_c}),
\end{align*}
where the last equality holds because $\tilde z_i,z_j$ approach constants at an exponential rate.

For $i=1,4$, by the same computation, and using that $z_3$ has linear growth (for example $z_3 = \sin(\alpha) y + \cos(\alpha) x + O( s e^{-s})$ on wing 1, where $s = \sqrt{x^2+y^2}+O(1)$), we obtain
\begin{align*}
\int_{\Sigma\cap \{z\in(0,\pi)\} } (\Delta+|A|^2)(\tilde \eta_i \tilde z_i) z_3
& =-\int_{\Sigma\cap \{z\in(0,\pi)\} } \tilde z_i^2 \nabla \tilde \eta_i \nabla \left(\frac{z_3}{\tilde z_i}\right)
\\
&= - \bar B (-1)^{i-1}+ o(1)
\end{align*}
where $\bar B>0$ and $o(1)\to 0$ as $R_c\to\infty$.
Similarly, for $i=1,4$,
\begin{align*}
\int_{\Sigma\cap \{z\in(0,\pi)\} } (\Delta+|A|^2)(\tilde \eta_i \tilde z_i) z_4
& =-\int_{\Sigma\cap \{z\in(0,\pi)\} } \tilde z_i^2 \nabla \tilde \eta_i \nabla (\frac{z_4}{\tilde z_i})
\\
&= B + o(1)
\end{align*}
where $B>0$.
This implies that the matrix $M_1$
is invertible if $R_c$ is taken large (and fixed).

Let us now estimate the matrix $M_2$. This means, we need  to estimate the following integrals for $i=1,2$, $j=1,\ldots,4$:
\begin{align*}
\int_{\Sigma\cap \{z\in(0,\pi)\} }
w_i z_j, \quad \int_{\Sigma\cap \{z\in(0,\pi)\} }
w_i' z_j
.
\end{align*}
Using \eqref{formula w}, for $j=1,2$, we have
\begin{align*}
\int_{\Sigma\cap \{z\in(0,\pi)\} } w_1 z_j
&=
\int_{\Sigma\cap \{z\in(0,\pi)\} }
(\Delta+|A|^2)(\eta_{tr,1})  z_j = O(e^{-R_{tr}})
\end{align*}
and similarly  $\int w_2 z_j= O(e^{-R_{tr}})$.
Next, we compute
\begin{align*}
\int_{\Sigma\cap \{z\in(0,\pi)\} } w_1 z_3
&=
\int_{\Sigma\cap \{z\in(0,\pi)\} }
(\Delta+|A|^2)(\eta_{tr,1})  z_3
\\
&=\int_{\Sigma\cap \{z\in(0,\pi)\} }
z_3 \Delta\eta_{tr,1}
+\int_{\Sigma\cap \{z\in(0,\pi)\} }
|A|^2 \eta_{tr,1}  z_3.
\end{align*}
The first integral is supported on $R_{tr}\leq s \leq R_{tr}+1$, so integrating by parts gives
\begin{align*}
\int_{\Sigma\cap \{z\in(0,\pi)\} }
z_3 \Delta\eta_{tr,1}
&=
-\int_{\{R_{tr}\leq s \leq R_{tr}+1\}} \nabla z_3\nabla \eta_{tr,1}
\\
&=-\nabla z_3 \hat n|_{s=R_{tr}+1} +
\int_{ \{R_{tr}\leq s \leq R_{tr}+1\} }\eta_{tr,1}  \Delta z_3,
\end{align*}
where $\hat n$ is tangent to $\Sigma$ and perpendicular to the curve $s=R_{tr}+1$.
Therefore
\begin{align*}
\int_{\Sigma\cap \{z\in(0,\pi)\} } w_1 z_3
&=
-\nabla z_3 \hat n|_{s=R_{tr}+1} +
\int_{\{s\geq R_{tr}+1\}} |A|^2 z_3 \\
&= - \pi +o(1) ,
\end{align*}
where  $o(1)\to0$ as $R_{tr} \to \infty$, thanks to the behavior $z_3 = s + O(1)$ as $s\to\infty$ and a corresponding estimate for the its derivative. Similarly,
\begin{align*}
\int_{\Sigma\cap \{z\in(0,\pi)\} } w_2 z_3
&= -\pi +o(1) ,
\\
\int_{\Sigma\cap \{z\in(0,\pi)\} } w_1 z_4
&= -\pi +o(1) ,
&
\int_{\Sigma\cap \{z\in(0,\pi)\} } w_2 z_4
&= \pi +o(1) .
\end{align*}

Let us now compute
\begin{align*}
\int_{\Sigma\cap \{z\in(0,\pi)\} } w_1' z_1
&=
\int_{\Sigma\cap \{z\in(0,\pi)\} } (\Delta+|A|^2)(\eta_{rot,1} z_3  ) z_1,
\end{align*}
where we have used \eqref{formula w prime}.
We have
\begin{align*}
\int_{\Sigma\cap \{z\in(0,\pi)\} } (\Delta+|A|^2)(\eta_{rot,1} z_3) z_1
&=
\int_{\Sigma\cap \{z\in(0,\pi)\} } (z_3 \Delta\eta_{rot,1} + 2 \nabla \eta_{rot,1} \nabla z_3 ) z_1
\\
&=\int_{\Sigma\cap \{z\in(0,\pi)\} } z_1 \nabla \eta_{rot,1} \nabla z_3
-  z_3\nabla \eta_{rot,1}  \nabla z_1
\\
&= -\pi\sin(\alpha) + o(1),
\end{align*}
where  $o(1)\to 0 $ as $R_{rot}\to0$.
In a similar way,
\begin{align*}
\int_{\Sigma\cap \{z\in(0,\pi)\} } w_1' z_2 &= \pi\cos(\alpha)+o(1),\\
\int_{\Sigma\cap \{z\in(0,\pi)\} } w_2' z_1 &= \pi\sin(\alpha)+o(1),&
\int_{\Sigma\cap \{z\in(0,\pi)\} } w_2' z_2 &= \pi\cos(\alpha)+o(1) .
\end{align*}
We have   also
\begin{align*}
\int_{\Sigma\cap \{z\in(0,\pi)\} } w_{1}' z_{3}=0 ,
\quad
\int_{\Sigma\cap \{z\in(0,\pi)\} } w_{2}' z_{3}=0 .
\end{align*}
Indeed, consider
\begin{align*}
\int_{\Sigma\cap \{z\in(0,\pi)\} } w_1' z_3 &=
\int_{\Sigma\cap \{z\in(0,\pi)\} }  (\Delta+|A|^2)(\eta_{rot,1} z_3   ) z_3
\\
&=
\int_{\Sigma\cap \{z\in(0,\pi)\} } (z_3 \Delta\eta_{rot,1} + 2 \nabla \eta_{rot,1} \nabla z_3 ) z_3
\\
&=
\int_{\Sigma\cap \{z\in(0,\pi)\} }
\nabla\cdot(z_3^2 \nabla \eta_{rot,1})
\\
&=0.
\end{align*}
The integral $\int w_{2}' z_{3}=0 $ is computed similarly.

Finally, we observe that
\begin{align*}
\int_{\Sigma\cap \{z\in(0,\pi)\} } w_1' z_4 =O(1),
\qquad
\int_{\Sigma\cap \{z\in(0,\pi)\} } w_2' z_4 =O(1),
\end{align*}
as $R_{rot}\to\infty$.
Then
\begin{align*}
M_2=
\left[
\begin{matrix}
\int w_1 z_1 & \int w_1 z_2 & \int w_1 z_3 & \int w_1 z_4 \\
\int w_2 z_1 & \int w_2 z_2 & \int w_2 z_3 & \int w_2 z_4 \\
\int w_1' z_1 & \int w_1' z_2 & \int w_1' z_3 & \int w_1' z_4 \\
\int w_2' z_1 & \int w_2' z_2 & \int w_2' z_3 & \int w_2' z_4
\end{matrix}
\right]
=\left[
\begin{matrix}
0 & 0 & -\pi & -\pi \\
0 & 0 & -\pi & \pi \\
-\pi\sin(\alpha) & \pi\cos(\alpha) & 0 & O(1) \\
\pi\sin(\alpha) & \pi\cos(\alpha) & 0 & O(1)
\end{matrix}
\right]+o(1)
\end{align*}
This shows that $M_2$ is invertible if we fix both $R_{rot}+10<R_{tr}$ large,which finishes the proof.
\end{proof}

\section{Linear theory}

Let $\mathcal E = \mathcal M \cap \{ s \geq \frac{\delta_s}{3\ve} \}$.
In this section we construct a right inverse of the operator $\Delta + |A|^2 + \ve \partial_z $ on $\mathcal E$. More precisely we want, given $h$ defined on $\mathcal E$ with some decay, a solution $\phi$ of
\begin{align}
\label{linear ends}
\Delta \phi + |A|^2 \phi + \ve \partial_z \phi = h
\quad \text{on }\mathcal E.
\end{align}
Given $\alpha,\gamma\in (0,1)$, let us define the following norms:
\begin{align}
\label{norm end sol}
\|\phi_2\|_{*,\mathcal E}
&= \ve^2 \sup_{x \in \mathcal E}  e^{ \gamma\delta_s/\ve+\gamma \ve s(x)} \|\phi_2 \|_{C^{2,\alpha}(\overline B_1(x))}
\\
\label{norm end rhs}
\|h_2\|_{**,\mathcal E}
&= \sup_{x \in \mathcal E} e^{  \gamma\delta_s/\ve+\gamma \ve s(x)} \|h_2\|_{C^{2,\alpha}(\overline B_1(x))} ,
\end{align}
where $B_1(x)$ is the geodesic ball with center $x$ and radius 1, and $s$ is the function defined in the construction of $\mathcal M$, Section~\ref{sec:construction of M}.
The factor $e^{\gamma\delta_s/\ve}$ in front of both norms is not immediately relevant; it will be useful later.

We have the following result.
\begin{prop}
\label{prop linear ends}
Let $0<\gamma<1$. Then there is a linear operator that associates to a function $h$ defined on $\mathcal E$ with $\|  h\|_{**,\mathcal E}<\infty$
 a solution $\phi$ to \eqref{linear ends}. Moreover,
$$
\| \phi\|_{*,\mathcal E}\leq C \|  h\|_{**,\mathcal E}.
$$
\end{prop}

For the proof, we scale to size one, that is, we work on $\tilde{\mathcal E} =\ve \mathcal E$. Then \eqref{linear ends} becomes equivalent to
\begin{align}
\label{eq linear ends2}
\Delta \phi + |A|^2 \phi + \partial_z \phi = \tilde h
\quad\text{on }\tilde{\mathcal E},
\end{align}
with $\tilde h(x)= \ve^{-2} h(x/\ve)$.

We  study the linear operator on the unbounded pieces in the following section, then in Section \ref{ssec:Linear theory bounded}, we deal with the bounded piece. We point out that in the radially symmetric case a related linear theory for the Jacobi operator was developed on \cite{dkw}.

\subsection{Linear theory on the ends}
\label{sec Linear theory on the ends}

Let $\mathcal E_u$ be any of the unbounded components of $\mathcal E$ and $\tilde{\mathcal E}_u = \ve \mathcal E_u$.
We introduce coordinates on $\tilde{\mathcal E}_u$ as follows. Consider a curve
$$
s\mapsto (\gamma_1(s),\gamma_3(s)),
$$
parametrized by arc length, with $s \in [0,\infty)$ and $\gamma_1'(s)>0$, that solves the ordinary differential equation \eqref{eq:self-trans-ode},
with initial conditions at $s=0$  chosen to be compatible with the construction of the initial approximation $\mathcal M$ in Section~\ref{section:Fitting the Scherk surface}. These initial conditions are functions of the parameters of the construction, in particular, the parameters $\beta_1,\beta_4, \tau_1,\tau_4$  used in the dislocations are all in $[-\delta_p,\delta_p]$. We note that the $s$ here differs from $s$ in the construction of $\mathcal M$, Section~\ref{sec:construction of M}, by a shift and a scaling.

Then
\begin{align}
\label{X0}
X_0 = \nu_0(s,\theta) = (\gamma_1(s)\cos(\theta),\gamma_1(s)\sin(\theta), \gamma_3(s))
\end{align}
$s \in [0,\infty)$, $\theta\in [0,2\pi]$, parametrizes part of the catenoid $\mathcal W_R$ or the paraboloid $\mathcal P$. Let
$$
\nu_0(s,\theta) = (-\gamma_3'(s)\cos(\theta),-\gamma_3'(s)\sin(\theta), \gamma_1'(s))
$$
be a unit normal vector.
Then a parametrization of the unbounded end $\tilde{\mathcal E}_u$ is given by
\begin{align}
\label{X}
X(s,\theta) = X_0(s,\theta)+ \nu_0(s,\theta) f(s/\ve,{\theta}/{\ve})
\end{align}
where $f$ is essentially a cut-off function times $f_\alpha$, the function that allows one to write the Scherk surface as a graph over a plane, see Lemma~\ref{lem:scherk-properties}. The important properties of $f$ are that $|\partial^k f(\tilde s,\tilde \theta)| \leq C_k e^{ - \delta_s /(10\ve)}$ for some $C_k$ and that it vanishes for $\tilde s\geq 10 \delta_s/\ve$.

Then we have the following expression for the operator $\Delta+|A|^2+\partial_z $ in the coordinates $s$ and $\theta$:
\begin{align}
\label{Jacobi ends}
\Delta \phi +|A|^2\phi+\partial_z\phi=
\partial_{ss} \phi + \frac{1}{\gamma_1(s)^2} \partial_{\theta\theta} \phi + \left( \frac{\gamma_1'(s)}{\gamma_1(s)}+\gamma_3'(s)\right)\partial_s\phi+|A|^2 \phi  + \tilde L \phi
\end{align}
where $\tilde L$ is a second order differential operator in $\phi$ with coefficients that are $o(1)$ as $\ve\to0$ and supported in $s\in[0,10\delta_s]$.
Using \eqref{asymp varphi} and the fact that the principal curvatures of a surface of revolution $z = F(r)$ are given by
$$
\kappa_1 = \frac{F''(r)}{(1+F'(r)^2)^{3/2}},
\quad
\kappa_2 = \frac{F'(r)}{r (1+F'(r)^2)^{1/2}},
$$
we can write
$$
\Delta \phi +|A|^2\phi+\partial_z\phi=\partial_{ss} \phi +a(s) \partial_{\theta\theta} \phi + b(s) \partial_s \phi + |A|^2 \phi + \tilde L\phi,
$$
where we have the following properties of the coefficients:
\begin{align}
\label{asymp coeff L0}
\left\{
\begin{aligned}
a(s) &= \frac{1}{2s}(1+O(s^{-1/2}))\\
b(s) &= 1+O(s^{-1/2})\\
|A|^2 &=  \frac{1}{2s}(1+O(s^{-1/2}))
\end{aligned}
\right.
\end{align}
as $s\to\infty$, and $a(s)>0$ for all $s\geq 0$.

Let
$$
L_0\phi=\partial_{ss} \phi +a(s) \partial_{\theta\theta}\phi + b(s) \partial_s \phi + |A|^2 \phi ,
$$
and consider the equation
\begin{align}
\label{eq linear ends3}
L_0\phi = h, \quad s\in (0,\infty), \ \theta\in[0,2\pi].
\end{align}
To prove Proposition~\ref{prop linear ends}, we will first construct an inverse operator for $L_0$.
\begin{prop}
\label{prop inverse L0}
Let $0<\gamma<1$. Then there is a linear operator $h\mapsto \phi$ that associates to a function $h=h(s,\theta)$ defined for $ (s,\theta)\in (0,\infty)\times[0,2\pi]$, $2\pi$-periodic in $\theta$ with $\|e^{\gamma s} h\|_{L^\infty}<\infty$ a solution $\phi$ of  \eqref{eq linear ends3}, which is  $2\pi$-periodic in $\theta$ and satisfies
$$
\|e^{\gamma s}\phi\|_{L^\infty}\leq C \|e^{\gamma s} h \|_{L^\infty}.
$$
\end{prop}

For the proof of this result, we will write $h$ and $\phi$ in Fourier series
\begin{align}
\label{fourier}
\phi(s,\theta) = \sum_{k\in \Z} \phi_k(s) e^{i k \theta} ,\quad
h(s,\theta) = \sum_{k\in \Z} h_k(s) e^{i k \theta} .
\end{align}
Then, if $\phi$ is smooth with exponential decay, equation \eqref{eq linear ends3} is equivalent to
\begin{align}
\label{eq phi k}
\phi_k'' + b(s) \phi_k' + (|A|^2-a(s) k^2 ) \phi_k = h_k,
\quad \text{for all } s >0 ,
\end{align}
for all $k\in\Z$. We need a couple of lemmas before starting  the proof of Proposition~\ref{prop inverse L0}. Combined, they allow us to deal with the low modes, i.e $|k| \leq k_0$ for some fixed $k_0$, and focus our attention on the higher frequencies.

\begin{lemma}
\label{sol each mode}
Let $0<\gamma<1$. If $|h_k(s)|\leq C e^{-\gamma s}$, then  \eqref{eq phi k} has a unique solution $\phi_k$ with
$$
\| e^{\gamma s}\phi_k\|_{L^\infty} \leq C_k \| e^{\gamma s} h_k\|_{L^\infty} .
$$
\end{lemma}
\begin{proof}

Note that as $s\to\infty$, equation \eqref{eq phi k} with $h_k=0$ is asymptotic to
\begin{align}
\label{homog eq k}
\phi_k''+\phi_k'+ \frac{b_k}s \phi_k=0 ,
\end{align}
where $b_k = \frac{1-k^2}{2}$. Note that $b_0=\frac12$, $b_1=0$, and $b_k<0$ for $k>1$.
We see that the homogeneous equation \eqref{homog eq k} has two independent solutions with the behaviors $s^{-b_k}(1+O(s^{-1}))$ and $s^{b_k} e^{-s} (1+O(s^{-1}) ) $ as $s\to\infty$. Using these solutions, it is possible to construct, for each $k\geq 0$, two elements $z_{1,k}$ and $z_{2,k}$ in the kernel of the operator $\phi'' + b(s) \phi' + (|A|^2-a(s) k^2 ) \phi$ such that
$$
z_{1,k}(s) = s^{-b_k}(1+O(s^{-\sigma})), \quad
z_{2,k}(s) =s^{b_k} e^{-s} (1+O(s^{-\sigma}) ) ,
$$
as $s\to \infty$, where $\sigma\in(0,1)$.
Now we can construct the solution $\phi_k$ using the variation of parameters formula:
$$
\phi_k =  - z_{1,k} \int \frac{h_k z_{2,k}}{W_k}+z_{2,k} \int \frac{h_k z_{1,k}}{W_k},
$$
where $W_k = z_{1,k}z_{2,k}'- z_{1,k}'z_{2,k} = e^{-s} (1+O(s^{-\sigma}))$ and the integrals are chosen to have the desired decay.
An alternative for the construction is to use the super solution $\bar\phi = e^{-\gamma s}$ and the calculation as in Lemma~\ref{lemma a priori big k}.
\end{proof}

Let $k_0\in \N$.
Let $\phi$ be a bounded measurable function on $[0,\infty)\times[0,2\pi]$. We will say that the Fourier coefficients of $\phi$ of order less than $k_0$ vanish if
$$
\int_0^{2\pi} \phi(s,\theta)e^{-ik\theta} \,d\theta=0 ,
\quad \forall s>0,  \ \forall |k|< k_0.
$$
\begin{lemma}
\label{lemma a priori big k}
There is a $k_0$ with the following property: Suppose that $\phi$ is continuous, satisfies  \eqref{eq linear ends3}, with $\phi$ and $h$ $2\pi$-periodic in $\theta$, that
\begin{gather*}
\phi(0,\theta)=0\quad \forall \theta\in[0,2\pi], \\
\|e^{\gamma s}\phi\|_{L^\infty}<\infty, \quad \|e^{\gamma s}h\|_{L^\infty}<\infty,
\end{gather*}
and that the Fourier coefficients of order less than $k_0$ of $\phi$ and $h$ are zero, then there is a constant $C$  independent of $\phi$ and $h$ such that
$$
\|e^{\gamma s}\phi\|_{L^\infty} \leq C\|e^{\gamma s}h\|_{L^\infty}.
$$

\end{lemma}
\begin{proof}
We proceed by contradiction. Assume that the statement fails, so that there exist sequences $\phi_n$, $h_n$ such that $\phi_n$, $h_n$ are $2\pi$-periodic in $\theta$, with vanishing Fourier coefficients of order less than $k_0$ ($k_0$ will be fixed later), $\phi_n$ solves \eqref{eq linear ends3} with right hand side $h_n$, $\phi_n(0,\theta)=0$ for $\theta\in[0,2\pi]$, and
\begin{align}
\label{conv}
\|e^{\gamma s}\phi_n\|_{L^\infty}=1, \quad
\|e^{\gamma s}h_n\|_{L^\infty}\to0, \quad \text{as }n\to\infty.
\end{align}

Consider $\bar\phi(s) = e^{-\gamma s}$. By \eqref{asymp coeff L0} we see that
$$
L_0 \bar\phi = (\gamma^2-\gamma)e^{-\gamma s}(1+O(s)) \leq \frac{\gamma^2-\gamma}{2}e^{-\gamma s}
$$
for $s\geq s_0$ (here we fix $s_0>0$ large).
Using the maximum principle with $\bar\phi + \sigma e^{-\tilde \gamma s}$ with $0<\tilde \gamma<\gamma$ and $\sigma>0$, and then letting $\sigma\to0$, we obtain
$$
\|e^{\gamma s}\phi_n\|_{L^\infty( (s_0,\infty)\times[0,2\pi]) ) }
\leq C\|e^{\gamma s}h_n\|_{L^\infty} + C \|e^{\gamma s}\phi_n\|_{L^\infty( (0,s_0)\times[0,2\pi]) )}.
$$
From this and \eqref{conv}, we deduce that for a subsequence (also denoted $\phi_n$)
$$
\|e^{\gamma s}\phi_n\|_{L^\infty( (0,s_0)\times[0,2\pi]) )}\geq c_0
$$
for some constant $c_0>0$. By standard elliptic estimates, up to a new subsequence, $\phi_n \to \phi$ uniformly on compact subsets of $[0,\infty)\times[0,2\pi]$ and by the previous remark, $\phi\not\equiv0$. But because of \eqref{conv}, $\phi$ satisfies
$$
L_0 \phi = 0\quad \text{in } (0,\infty)\times[0,2\pi],
$$
with $\phi(0,\theta)=0$ and $|\phi(s)|\leq e^{-\gamma s}$.
Let us write $\phi$ as a Fourier series \eqref{fourier}. Due to the hypotheses, $\phi_k=0$ for $|k|<k_0$. Note that $\phi_k$ satisfies \eqref{eq phi k} with right-hand side equal to 0 and $\phi_k(0)=0$.
The function $\bar\phi(s)=e^{-\gamma s}$  satisfies
$$
\bar\phi'' + b(s) \bar\phi' + (|A|^2-a(s) k^2 ) \bar\phi =
\Big(\gamma^2-\gamma b(s) +  |A|^2-a(s) k^2  \Big) e ^{-\gamma s}  .
$$
But from the fact that $a(s)>0$ for all $s\geq 0$ and the estimates on the coefficients \eqref{asymp coeff L0}, we see that there is $k_0$ such that
\begin{align}
\label{super}
\bar\phi'' + b(s) \bar\phi' + (|A|^2-a(s) k^2 ) \bar\phi < - c_k e^{-\gamma s},
\quad \forall |k|\geq k_0, \forall s\geq 0,
\end{align}
where $c_k>0$.
From the maximum principle, we deduce that $\phi_k\equiv0$ for all $|k|\geq k_0$. This is a contradiction.
\end{proof}

\begin{proof}[Proof of Proposition~\ref{prop inverse L0}]
Let $k_0$ be as in Lemma~\ref{lemma a priori big k}.
Using Lemma~\ref{sol each mode}, for each $|k|<k_0$, we find a solution $\phi_k$ of \eqref{eq phi k} in $(0,\infty)$. Then we need to prove the lemma only under the assumption that the Fourier coefficients of order less than $k_0$ of $h$ vanish.

For the moment let us assume in addition that $h$ is $C^2$  and
\begin{align}
\label{2der}
|h_{\theta\theta}(s,\theta)|\leq C e^{-\gamma s} \quad \text{for all }(s,\theta)\in [0,\infty)\times [0,2\pi].
\end{align}
Writing $h$ in Fourier series as in \eqref{fourier}
for $|k|\geq k_0$ we can find a solution $\phi_{k}$ of \eqref{eq phi k}  in $(0,\infty)$ with right hand side $h_{k}$, satisfying $\phi_{k}(0)=0$. This can be done using the supersolution  $\bar\phi(s)=e^{-\gamma s}$ and \eqref{super}. Alternatively, one can use the variation of parameters formula and elements in the kernel as in Lemma~\ref{sol each mode}.
For $m>k_0$, let
$$
\phi_{m}(s,\theta) = \sum_{k_0 \leq |k|\leq m} e^{i k \theta} \phi_{k}(s)
$$
and similarly define $h_m$. By Lemma~\ref{lemma a priori big k},
$$
\|e^{\gamma s} \phi_{m}\|_{L^\infty} \leq C
\|e^{\gamma s} h_{m}\|_{L^\infty} ,
$$
with $C$ independent of $m$.
Note that
\begin{align*}
|h_k(s)|&= \frac{1}{2\pi}\left|\int_0^{2\pi} h(s,\theta) e^{ik\theta}\,d\theta\right|=\frac{1}{2\pi k^2}\left|\int_0^{2\pi} h_{\theta\theta}(s,\theta) e^{ik\theta}\,d\theta\right|\leq \frac{C}{k^2}e^{-\gamma s} .
\end{align*}
Then
\begin{align*}
|h_m(s,\theta)| &= \left|\sum_{k_0\leq |k|\leq m} h_k(s)e^{i k \theta}\right|\leq \sum_{k_0\leq |k|\leq m} |h_k(s)| \leq C e^{-\gamma s}  \sum_{k_0\leq |k|\leq m} k^{-2}.
\end{align*}
Therefore
$$
\|e^{\gamma s} \phi_{m}\|_{L^\infty} \leq C,
$$
with $C$ independent of $m$. By standard elliptic estimates, for a subsequence $m\to\infty$ we find $\phi_m\to\phi$ uniformly on compact subsets of $[0,\infty)\times[0,2\pi]$ and $\phi$ is the desired solution.

Next, we lift the assumption \eqref{2der}. Indeed, assume only $\|e^{\gamma s} h\|_{L^\infty}<\infty$ and that the Fourier coefficients of $h$ vanish for $|k|<k_0$. Let $\rho_n$ be a sequence of mollifiers in $\R^2$ and
$$
h_n = h*\rho_n
$$
(extending $h$ by 0 for $s\leq 0$). We have $h_n\to h$ a.e; the Fourier coefficients of order less than $k_0$ of $h_n$ vanish and
$$
|h_n(s,\theta)|\leq C e^{-\gamma s},
$$
$$
|\partial_{\theta\theta} h_n(s,\theta)|\leq C_n e^{-\gamma s},
$$
Using the previous argument we find a solution $\phi_n$ with
$$
\|e^{\gamma s}\phi_n\|_{L^\infty}\leq C \|e^{\gamma s}h_n\|_{L^\infty}\leq C,
$$
with $C$ independent of $n$ and right-hand side $h_n$. Passing to a subsequence, we find the desired solution.
\end{proof}

\subsection{Linear theory on the bounded piece}
\label{ssec:Linear theory bounded}

Let us consider the bounded component $\mathcal E_b $ of  $\mathcal E = \mathcal M \setminus \{ s \leq \frac{\delta_s}{3\ve} \}$ and let $\tilde{\mathcal E}_b = \ve \mathcal E_b$.

\begin{lemma}
\label{solvability bounded component}
For  $h \in C^\alpha(\overline{\tilde{\mathcal E}_b})$, there is a unique solution $\phi$ of
\begin{align}
\label{eq bounded part}
\left\{
\begin{aligned}
\Delta  \phi + |A|^2 \phi + \partial_z \phi & = h \quad \text{in } \tilde{\mathcal E}_b
\\
\phi & = 0 \quad \text{on } \partial \tilde{\mathcal E}_b .
\end{aligned}
\right.
\end{align}
with
\begin{align}
\label{est}
\|\phi\|_{L^\infty} \leq C \|h\|_{L^\infty},
\end{align}
where the constant $C$ is independent of $h$ and $\phi$.
\end{lemma}
\begin{proof}

Let $L\phi$ be $\Delta  \phi + |A|^2 \phi + \partial_z \phi$, where the geometric quantities and Laplacian are the ones for $\tilde{\mathcal E}_b$. We let $L_0$ denote the corresponding operator for the paraboloid $\mathcal P$. 

We can parametrize $\mathcal P$ with polar coordinates
$$
X_0(r,\theta) = (r \cos(\theta),r\sin(\theta),F(r)) ,
$$
with $r\geq 0$, $\theta\in[0,2\pi]$, where $F=F_0$ is the unique radially symmetric solution of \eqref{eqFrad} with $F(0)=0$.
Then, in the coordinates $r$ and $\theta$,
$$
L_0 \phi = B_1(r) \phi_{rr} + B_2(r) \phi_r + \frac{1}{r^2} \phi_{\theta\theta} + A^2(r) \phi,
$$
where
\begin{align*}
B_1(r) & =  \frac{1}{1+F'(r)^2},
\\
B_2(r) & =
\frac{1}{r(1+F'(r)^2)}-\frac{F'(r)F''(r)}{(1+F'(r)^2)^2},
\\
A^2(r) & = \frac{F''(r)^2}{(1+F'(r)^2)^3} + \frac{F'(r)^2}{r^2(1+F'(r)^2)} .
\end{align*}
As before, we denote by $\nu_0$ the unit normal vector to $\mathcal P$ such that $\langle \nu_0, e_z \rangle >0$. The surface $\tilde{\mathcal E}_b$ can then be parametrized by
$$
X_0(r,\theta) + \nu_0(r,\theta) f({r}/{\ve},{\theta}/{\ve})
$$
for $r\in[0,R_1]$, $\theta\in [0,2\pi]$, with some $R_1>0,$ and where $f$ has the property that $f({r}/{\ve},{\theta}/{\ve})$ is supported where $r\in  [R_1-10\delta_s,R_1]$ and $f$ and its derivatives can be bounded by $e^{-\delta_s/(10\ve)}$.  This implies that
$$
L\phi= L_0 \phi + \tilde L \phi,
$$
where $\tilde L$ is a second order differential operator in $\phi$ with coefficients that are $o(1)$ as $\ve\to0$ and supported in $r \in  [R_1-10\delta_s,R_1]$.

Let $v (r) = \langle\nu_0 ,e_z\rangle$. Then $L_0 v= 0$ and $v(r)>0$ for all $r \geq 0$. We define now
$$
\bar \phi(r) = v(r) - \mu e^{-K r},
$$
where $\mu = \frac12 \inf_{r\in [0,R_1]} v(r)>0$ and $K>0$ is to be chosen. Then, we compute
\begin{align*}
L \bar \phi =
\mu e^{-K s} \left[ - B_1(r) K^2 + K  B_2(r) - A^2(r) \right] - \mu \tilde L (e^{-K s})+
\tilde L v.
\end{align*}
But $\tilde L v = O(\ve)$, $\tilde L (e^{-Kr}) = o(1) K^2 e^{-K r}$ where in the last expression $o(1)$ is uniform in $K$ as $\ve\to0$. Since $B_1(r)$ is positive in $[0,R_1]$ we can choose $K$ large so that
$$
L\bar \phi \leq -c , \quad \forall r\in[0,R_1],
$$
for some $c>0$ and all $\ve>0$ small. Then $\bar \phi$ is a super solution for the operator $L$ and hence the bound \eqref{est} holds. The equation is solved then by super and subsolutions.
\end{proof}

\begin{proof}[Proof of Proposition~\ref{prop linear ends}]
Let $0<\gamma<1$.

Let us consider first one of the unbounded ends $\mathcal E_u$ and $\tilde{\mathcal E}_u = \ve \mathcal E_u$.
Recall that $\tilde{\mathcal E}_u$ is parametrized by \eqref{X}, in particular, $s\geq 0$, $\theta \in [0,2\pi)$ are global coordinates on this surface. We write $s(x)$ for $x \in \tilde{\mathcal E}_u$.
Using $s,\theta$ we may identify $\tilde{\mathcal E}_u $ with an unbounded piece of the paraboloid or catenoid. We write this piece as $E_u$.

Let $\alpha \in (0,1)$ and define the norms for functions defined on $E_u $:
\begin{align*}
\|h\|_{k,\alpha} = \sup_{x\in E_u} e^{\gamma s(x)} \|h\|_{C^{k,\alpha}(\overline B_1(x) )} ,
\end{align*}
where $B_1(x)$ is the geodesic ball of center $x$ and radius 1 in $E_u$.

If $h$ is defined on $E_u $ and $\|h\|_{0,\alpha}<\infty$, using Proposition~\ref{prop inverse L0}, we obtain a solution $\phi= T_0(h)$  of  \eqref{eq linear ends3}, $2\pi$-periodic in $\theta$, and such that
$$
\|e^{\gamma s}T_0(h)\|_{L^\infty}\leq C \|e^{\gamma s} h \|_{L^\infty} \leq C \|h\|_{0,\alpha}.
$$
By considering $\phi$ and $h$ as functions on $E_u$, equation  \eqref{eq linear ends3} becomes
$$
\Delta_{E_u} \phi + |A_{E_u}|^2\phi + \partial_z \phi = h \quad \text{on } E_u.
$$
Standard elliptic estimates, applied on geodesic balls of radius 1 of $E_u$, give that
$$
\|\phi\|_{2,\alpha} \leq C ( \|h\|_{0,\alpha} + \|\phi\|_{C^{2,\alpha}(\partial E_u)}
) .
$$
Note that $\partial E_u$ in the coordinates $s,\theta$ corresponds to $s=0$.
Consider the representations of $\phi$ and $h$ as Fourier series as in \eqref{fourier}. We recall that, by construction, if $\phi = T_0(h)$ then $\phi(0,\cdot)$ has Fourier modes of index $|k|\geq k_0$ equal to zero, where $k_0$ is a fixed integer. Hence
$$
\|\phi\|_{C^{2,\alpha}(\partial E_u)} \leq C \sum_{|k|\leq k_0} |\phi_k(0)|.
$$
The solutions $\phi_k$ are constructed in Lemma~\ref{sol each mode} and, in particular, we have $|\phi_k(0)|\leq C_k \|e^{\gamma s} h_k\|_{L^\infty}$. From this, we deduce
$$
\|\phi\|_{C^{2,\alpha}(\partial E_u)} \leq C \|e^{\gamma s} h\|_{L^\infty},
$$
and hence
\begin{align}
\label{est T0}
\|T_0(h)\|_{2,\alpha} \leq C  \|h\|_{0,\alpha}.
\end{align}

To solve equation \eqref{eq linear ends2}, we rewrite it
\begin{align}
\label{fixed1}
\Delta_{E_u} \phi + |A_{E_u}|^2\phi + \partial_z \phi = h + \tilde L\phi \quad \text{on } E_u,
\end{align}
where $\tilde L$ a second order elliptic equation with coefficients that are $o(1)$ as $\ve \to0$ and with compact support. This translates to
\begin{align}
\label{estim2}
\|\tilde L \phi\|_{0,\alpha} \leq o(1) \|\phi\|_{2,\alpha}.
\end{align}
Using the operator $T_0$, we can find a solution of \eqref{fixed1} by solving the fixed point problem
$$
\phi = T_0( h + \tilde L \phi)
$$
in the Banach space $\{ \phi \in C^{2,\alpha}(\overline E_u): \|\phi\|_{2,\alpha}<\infty\}$.
By \eqref{est T0}, \eqref{estim2} and the contraction mapping principle, this fixed point problem has a unique solution. This yields a solution of \eqref{eq linear ends2}. By scaling we obtain therefore a solution of \eqref{linear ends} in any of the unbounded ends.

The proof for  the bounded component $\mathcal E_b $ of  $\mathcal E = \mathcal M \setminus \{ s \leq \frac{\delta_s}{3\ve} \}$ is similar, if one uses Lemma~\ref{solvability bounded component}, and the fact that the boundary condition is  taken equal to 0.
\end{proof}

\section{Proof of Theorem~\ref{main thm}}
\label{sec proof them}

To prove the theorem, it is sufficient to find a solution $\phi$ of \eqref{b1}, that is,
\begin{align}
\label{eq M1}
\Delta \phi + |A|^2 \phi + \ve \nabla \phi\cdot e_z + E + Q(x,\phi,\nabla \phi,D^2\phi,x)=0 \quad\text{in } \mathcal M ,
\end{align}
where $\mathcal M$ is the surface constructed in Section~\ref{sec:construction of M}, which depends on $\beta_1,\beta_2,\tau_1,\tau_4 \in [-\delta_p,\delta_p]$, and $E = H - \ve \nu\cdot e_z$. Later, we will verify that  $\{ x+ \phi(x) \nu(x): x \in \mathcal M \}$ is an embedded complete surface.

Thanks to Proposition~\ref{prop error},  $ E = E_0 + E_d
$ with $
\|E_0\|_{**} \leq C \ve
$ and
$$
E_d = \tau_1 w_1 + \tau_4 w_2 + \beta_1 w_1' + \beta_4 w_2' + O\left(\sum \beta_i^2 + \tau_i^2\right),
$$
where $O\left(\sum \beta_i^2 + \tau_i^2\right)$ are smooth functions with compact support, with $\| \ \|_{**}$ bounded by $\sum_{i=1,4} \beta_i^2 + \tau_i^2$.
Thus \eqref{eq M1} takes the form
\begin{align}
\label{eqM2}
& \Delta \phi + |A|^2 \phi + \ve \nabla \phi\cdot e_z + \tilde E + Q(\phi,\nabla \phi,D^2\phi)+ \tau_1 w_1
\\
\nonumber
& \qquad + \tau_4 w_2 + \beta_1 w_1' + \beta_4 w_2' =0 ,
\quad \text{ in }\mathcal M ,
\end{align}
where $\tilde E = E_0 +  O\left(\sum_{i=1,4} \beta_i^2 + \tau_i^2\right)$. Hence,
$$
\|\tilde E\|_{**} \leq C \left( \ve + \sum_{i=1,4}\left(\beta_i^2+\tau_i^2\right)\right) .
$$
Note that $\mathcal M = \mathcal M[\beta_1,\beta_2,\tau_1,\tau_2]$ and the unkonwns are $\phi$ and $\beta_1,\beta_2,\tau_1,\tau_2$.

We look for a solution $\phi$ of \eqref{eqM2} of the form
$$
\phi = \eta_1 \phi_1 + \eta_2 \phi_2,
$$
where $\phi_1,\phi_2$ are new unknown functions, which solve an appropriate system, and $\eta_1$, $\eta_2$ are smooth cut-off functions such that
\begin{align*}
\eta_1(s) &=
\begin{cases}
1 & \text{if } s \leq \frac{2\delta_s}{\ve}-1
\\
0 & \text{if } s \geq \frac{2\delta_s}{\ve}
\end{cases}
\\
\eta_2(s) &=
\begin{cases}
0 & \text{if } s \leq \frac{\delta_s}{3\ve}
\\
1 & \text{if } s \geq \frac{\delta_s}{2\ve}
\end{cases} ,
\end{align*}
where $s=s(x)$ measures geodesic distance from the core of $\mathcal M$.

We introduce next the following system for $\phi_1$, $\phi_2$:
%
\begin{multline}
\label{sys-eq1}
 \Delta \phi_1 + |A|^2 \phi_1
+ \ve \partial_z \phi_1 =  - \ve \phi_2 \partial_z \eta_2
 - 2 \nabla \phi_2 \nabla \eta_2 -  \phi_2 \Delta \eta_2\\
- \tilde \eta_1 (\tilde E+Q(x,\phi,\nabla \phi,D^2\phi))
- \tau_1 w_1 + \tau_4 w_2 + \beta_1 w_1' + \beta_4 w_2'
\quad \text{in }\mathcal C
\end{multline}
\begin{multline}
\label{sys-eq2}
\Delta\phi_2 + |A|^2 \phi_2 + \ve \partial_z \phi_2 =
- \ve\phi_1 \partial_z \eta_1
-2 \nabla \phi_1 \nabla \eta_1 - \phi_1 \Delta \eta_1  \\
- \tilde \eta_2 ((\tilde E+Q(x,\phi,\nabla \phi,D^2\phi))) \quad
\text{in }\mathcal E,
\end{multline}
where  $\mathcal E = \mathcal M \cap \{ s \geq \frac{\delta_s}{3\ve} \}$
is the union of the ends, $\mathcal C = \mathcal M \cap \{ s \leq \frac{2\delta_s}{\ve} \}$ is close to a Scherk surface, and $\tilde \eta_1$, $\tilde \eta_2$ are smooth cut-off functions such that: 
\begin{align*}
\tilde\eta_1(s) &=
\begin{cases}
1 & \text{if } s \leq \frac{\delta_s}{\ve}-1
\\
0 & \text{if } s \geq \frac{\delta_s}{\ve}
\end{cases}
\\
\tilde\eta_2(s) &= 1-\tilde\eta_1(s) .
\end{align*}
In the term $Q(\phi,\nabla \phi,D^2\phi)$ of \eqref{sys-eq1}, \eqref{sys-eq2}, $\phi$  means $\phi= \phi_1\eta_1 + \phi_2\eta_2$.
If $\phi_1,\phi_2$ solve  \eqref{sys-eq1}, \eqref{sys-eq2}, then multiplying \eqref{sys-eq1} by $\eta_1$ and \eqref{sys-eq2} by $\eta_2$ we see that $\phi= \phi_1\eta_1 + \phi_2\eta_2$ is a solution of \eqref{eqM2}.


Using the change of variables introduced in the construction of $\mathcal M$ in Section~\ref{sec:construction of M}, we see that solving \eqref{sys-eq1} is equivalent to finding a solution to
\begin{multline}
\label{sys-eq1-b}
 \Delta_\Sigma \phi_1 + |A_\Sigma|^2 \phi_1
= L'\phi_1
 - \ve \partial_z \phi_1 - \ve\phi_1 \partial_z \eta_1  - \ve \phi_2 \partial_z \eta_2
- 2 \nabla \phi_2 \nabla \eta_2 -  \phi_2 \Delta \eta_2
\\
- \tilde \eta_1 (\tilde E+Q(x,\phi,\nabla \phi,D^2\phi))
- \tau_1 w_1 + \tau_4 w_2 + \beta_1 w_1' + \beta_4 w_2'
\quad \text{in } \Sigma,
\end{multline}
where now all the functions are considered on the Scherk surface $\Sigma$; $\Delta_\Sigma$ and $A_\Sigma$ refer to the Laplace operator and second fundamental form of $\Sigma$. By Proposition~\ref{prop:linearopclose}, $L'$ is a second order operator with coefficients supported on $s\leq 2\delta_s/\ve$ and whose $C^1$ norms on the region $s\leq 5 \delta_s/\ve$ are bounded by $\delta_s+\delta_p+\ve$. In principle, by the change of variables, we need to solve \eqref{sys-eq1-b} on a subset of $\Sigma$, but finding a solution in all $\Sigma$ is sufficient. This solution is multiplied later by the cut-off $\eta_1$.

Similarly, we consider $\mathcal E_0 = \mathcal M_0 \cap \{s \geq \frac{\delta_s}{3\ve}\}$, where $\mathcal M_0$ is the initial approximation corresponding to $\beta_1=\beta_2=\tau_1=\tau_2=0$. For $|\beta_i|+|\tau_i|\leq \delta_p$ and $\delta_p>0$ fixed small, $\mathcal E$ is mapped onto $\mathcal E_0$ and this mapping allows us to write \eqref{sys-eq2} as
\begin{multline}
\label{sys-eq2-b}
\Delta_{\mathcal E_0}\phi_2 + |A_{\mathcal E_0}|^2 \phi_2 + \ve \partial_z \phi_2 = L''\phi_2 -2 \nabla \phi_1 \nabla \eta_1 - \phi_1 \Delta \eta_1
\\
- \tilde \eta_2 ((\tilde E+Q(x,\phi,\nabla \phi,D^2\phi))) \quad
\text{in }\mathcal E_0,
\end{multline}
where now all functions are considered on the Scherk surface $\mathcal E_0$; $\Delta_{\mathcal E_0}$ and $A_{\mathcal E_0}$ refer to the Laplace operator and second fundamental form of $\mathcal E_0$.
The operator $L''$ has coefficients whose $C^1$ norms are $o(\delta_p)$ as $\delta_p\to 0$. Indeed, one can see this using the form of the operator on the ends in \eqref{Jacobi ends} and the continuous dependence of ODE on initial conditions, because the parameters $\beta_i$, $\tau_i$ determine the initial condition for the differential equation \eqref{eq:self-trans-ode}.

We solve \eqref{sys-eq1-b} on the Scherk surface using  Proposition~\ref{prop linear scherk 4} with norms
\begin{align}
\label{norm sol Sigma}
\|\phi_1\|_{*,\Sigma}
&= \sup_{x\in\Sigma} e^{\gamma s(x)} \|\phi_1\|_{C^{2,\alpha}(\overline B_1(x))},
\\
\nonumber
\|h_1 \|_{**,\Sigma}
&= \sup_{x\in\Sigma} e^{\gamma s(x)} \|h_1\|_{C^{\alpha}(\overline B_1(x))} .
\end{align}

We consider the system \eqref{sys-eq2}, \eqref{sys-eq1-b} as a fixed point problem for  $(\phi_1,\phi_2,\beta_1,\beta_2,\tau_1,\tau_2)$ belonging to the subset $\mathcal B$ of $C^{2,\alpha}(\Sigma) \times C^{2,\alpha}(\mathcal E_0) \times \R^4$ defined by
	\[
	\mathcal B = \{(\phi_1,\phi_2,\beta_1,\beta_2,\tau_1,\tau_2) \mid
 \max_{i=1,2} (\|\phi_i\|_{*,\Sigma}, |\beta_i|, |\tau_i|) \leq M \ve \},
	\]
where $M >0 $ is a constant to be chosen later and the norms are defined in  \eqref{norm end sol} and \eqref{norm sol Sigma}.

Consider $(\phi_1,\phi_2,\beta_1,\beta_2,\tau_1,\tau_2)\in \mathcal B$.
We define $F(\phi_1,\phi_2,\beta_1,\beta_2,\tau_1,\tau_2)$ as follows.
Using Proposition~\ref{prop linear scherk 4}, we let $\bar \phi_1$, $\bar \beta_1$, $\bar \beta_2$, $\bar \tau_1$, $\bar \tau_2$ be the solution of
\begin{multline}
\label{sys-eq1-b2}
 \Delta_\Sigma \bar \phi_1 + |A_\Sigma|^2 \bar\phi_1
= L'\phi_1
- \ve \partial_z \phi_1 - \ve\phi_1 \partial_z \eta_1  - \ve \phi_2 \partial_z \eta_2
- 2 \nabla \phi_2 \nabla \eta_2 -  \phi_2 \Delta \eta_2
\\
- \tilde \eta_1 (\tilde E+Q(x,\phi,\nabla \phi,D^2\phi))
- \bar \tau_1 w_1 + \bar \tau_4 w_2 + \bar \beta_1 w_1' + \bar \beta_4 w_2'
\quad \text{in } \Sigma .
\end{multline}
In the right-hand side of this equation, all terms are well defined in spite of the fact that $\phi_1$, $\phi_2$ (after changing variables) are defined only on some subsets, because of the cut-off functions. As we will verify later, the right-hand side has finite $\|\ \|_{**,\Sigma}$ norm, thus Proposition~\ref{prop linear scherk 4}  applies.
Next, using Proposition~\ref{prop linear ends}, we find a solution of
\begin{multline}
\label{sys-eq2-b2}
\Delta_{\mathcal E_0}\bar\phi_2 + |A_{\mathcal E_0}|^2 \bar\phi_2 + \ve \partial_z \phi_2 = L''\phi_2 -2 \nabla \phi_1 \nabla \eta_1 - \phi_1 \Delta \eta_1
\\
- \tilde \eta_2 ((\tilde E+Q(x,\phi,\nabla \phi,D^2\phi))) \quad
\text{in }\mathcal E_0,
\end{multline}
where the right-hand side will be proven to have finite $\| \ \|_{**,\mathcal E_0}$ norm.
We define
$$
(\bar\phi_1,\bar\phi_2,\bar\beta_1,\bar\beta_2,\bar\tau_1,\bar\tau_2)= F (\phi_1,\phi_2,\beta_1,\beta_2,\tau_1,\tau_2)
.
$$


Let us verify that $F$ maps $\mathcal B$ into $\mathcal B$ and is a contraction. Consider $(\phi_1,\phi_2,\beta_1,\beta_2,\tau_1,\tau_2) \in \mathcal B$ and let $(\bar\phi_1,\bar\phi_2,\bar\beta_1,\bar\beta_2,\bar\tau_1,\bar\tau_2)= F (\phi_1,\phi_2,\beta_1,\beta_2,\tau_1,\tau_2) $.
Using \eqref{est linear scherk2} and standard elliptic estimates, we deduce
\begin{multline}
\|\bar\phi_1\|_{*,\Sigma} + |\bar\beta_i|+|\bar\tau_i|
\\\leq C \| L'\phi_1
- \ve \partial_z \phi_1 - \ve\phi_1 \partial_z \eta_1  - \ve \phi_2 \partial_z \eta_2
- 2 \nabla \phi_2 \nabla \eta_2 -  \phi_2 \Delta \eta_2 \|_{**,\Sigma}
\label{est bar phi1}
\\+\| \tilde \eta_1 (\tilde E+Q(x,\phi,\nabla \phi,D^2\phi))\|_{**,\Sigma} .
\end{multline}
We first remark that by, Proposition \ref{prop error}, we have
$$
\|\tilde \eta_1 \tilde E\|_{**,\Sigma} \leq C_E \ve .
$$
If $x$ lies in the support of $\eta_2$, that is, $\frac{\delta_s}{3\ve}\leq s(x)\leq \frac{\delta_s}{2\ve}$, then
$$
\|\Delta \eta_2 \|_{C^{\alpha}(\overline{B}_1(x))}\leq C \ve^2 , \quad
\|\phi_2\|_{C^{\alpha}(\overline{B}_1(x))} \leq \ve^{-2} e^{-\gamma\delta_s/\ve - \ve \gamma s(x)}\|\phi_2\|_{*,\mathcal E_0}.
$$
Then
$$
\| \phi_2 \Delta\eta_2 \|_{**,\Sigma} \leq C e^{-\frac{\delta_s}{2\ve}} \|\phi_2\|_{*,\mathcal E_0}.
$$
Similarly,
$$
\|\nabla \phi_2 \nabla \eta_2\|_{**,\Sigma} \leq C \ve^{-1} e^{-\frac{\delta_s}{2\ve}} \|\phi_2\|_{*,\mathcal E_0},
\quad
\text{and}
\quad
\|\phi_2 \partial_z \eta_2\|_{**,\Sigma} \leq C \ve^{-1} e^{-\frac{\delta_s}{2\ve}} \|\phi_2\|_{*,\mathcal E_0}.
$$
Because the $C^1$ norm of the coefficients of $L'$ in $s\leq 5\delta_s/\ve$ are bounded by $C(\delta_p+\delta_s+\ve)$, we have
$$
\|L'\phi_1\|_{**,\Sigma}\leq C(\delta_p+\delta_s+\ve)\|\phi_1\|_{*,\Sigma}.
$$
Also,
$$
\|\ve \partial_z \phi_1\|_{**,\Sigma}\leq C \ve \|\phi_1\|_{*,\Sigma}.
$$
For the $Q$ term, analogously to Proposition~\ref{prop quadratic}, we have
$$
\|\tilde \eta_1 Q(x,\phi,\nabla \phi,D^2\phi) \|_{**,\Sigma} \leq C \|\phi_1\|_{*,\Sigma}^2 + C \ve^{-4} e^{-\frac{\gamma\delta_s}{\ve}} \|\phi_2\|_{*,\mathcal E_0}^2.
$$
Therefore, using \eqref{est bar phi1} and the previous inequalities, we obtain
%
\begin{multline*}
\nonumber
\|\bar\phi_1\|_{*,\Sigma}
 \leq
C C_E\ve + C \ve^{-1} e^{-\frac{\delta_s}{2\ve}} \|\phi_2\|_{*,\mathcal E_0} + C ( \delta_p +\delta_s +\ve) \|\phi_1\|_{*,\Sigma}
\\
 +C \|\phi_1\|_{*,\Sigma}^2 + C \ve^{-4} e^{-\frac{\gamma\delta_s}{\ve}} \|\phi_2\|_{*,\mathcal E_0}^2,
\end{multline*}
which gives
\begin{multline}
\|\bar\phi_1\|_{*,\Sigma}
\label{B to B 1}
\\
\leq C C_E\ve  +  C \ve^{-1} e^{-\frac{\delta_s}{2\ve}} M\ve  + C ( \delta_p +\delta_s +\ve) M \ve
 +C M^2 \ve^2 + C \ve^{-4} e^{-\frac{\gamma\delta_s}{\ve}} M^2\ve.
 \end{multline}

To estimate $\bar\phi_2$, we use Proposition~\ref{prop linear ends} to obtain
\begin{align}
\label{est phi2}
\|\bar\phi_2\|_{*,\mathcal E_0} \leq C \|L''\phi_2 -2 \nabla \phi_1 \nabla \eta_1 - \phi_1 \Delta \eta_1 - \tilde \eta_2 ((\tilde E+Q(\phi,\nabla \phi,D^2\phi))) \|_{**,\mathcal E_0}.
\end{align}
We first remark that, by Proposition \ref{prop error}, we have
$$
\|\tilde \eta_2 \tilde E\|_{**,\mathcal E_0} \leq C_E \ve .
$$
We note that $\Delta \eta_1$ is supported on $\frac{2\delta_s}\ve-1\leq s \leq \frac{2\delta_s}\ve$. Using the smoothness of $\eta_1$, we get
$$
\| \phi_1 \Delta\eta_1\|_{**,\mathcal E_0}+\| \nabla \phi_1 \nabla \eta_1\|_{**,\mathcal E_0}\leq C e^{ -\frac{\gamma\delta_s}{\ve}}\|\phi_1\|_{*,\Sigma}.
$$
Again, as in Proposition~\ref{prop quadratic}, we have
$$
\|\tilde \eta_2 Q(x,\phi,\nabla \phi,D^2\phi) \|_{**,\mathcal E_0} \leq C \|\phi_1\|_{*,\Sigma}^2 + C \ve^{-4} e^{-\frac{\gamma\delta_s}{\ve}} \|\phi_2\|_{*,\mathcal E_0}^2.
$$
Combining \eqref{est phi2} with the previous inequalities, we arrive at
\begin{align}
\label{B to B 2}
\|\bar\phi_2\|_{*,\Sigma}
&\leq C C_E\ve  + o(\delta_p) M \ve +  C \ e^{-\frac{\delta_s}{2\ve}} M\ve
 +C M^2 \ve^2 + C \ve^{-4} e^{-\frac{\gamma\delta_s}{\ve}} M^2\ve.\end{align}

The right-hand sides of \eqref{B to B 1} and \eqref{B to B 2} are less that $M\ve$ provided $M$ is fixed large (for example $2 C C_E)$, then we fix $\delta_p$, $\delta_s$ small and  work with small  $\ve>0$. A similar estimate holds for $\bar\beta_i$, $\bar\tau_i$ and this shows that $F$ maps $\mathcal B$ to itself.

Let us verify that $F$ is a contraction in $\mathcal B$. For this, we first
claim that
$$
|F(\phi_1,\phi_2,\beta_1,\beta_2,\tau_1,\tau_2) - F(\psi_1,\psi_2,\beta_1,\beta_2,\tau_1,\tau_2) |\leq o(1) (
\|\phi_1-\psi_1\|_{*, \Sigma} +
\|\phi_2-\psi_2\|_{*,\mathcal E_0} )
$$
for $(\phi_1,\phi_2,\beta_1,\beta_2,\tau_1,\tau_2) , (\psi_1,\psi_2,\beta_1,\beta_2,\tau_1,\tau_2)  \in \mathcal B$, where $o(1)$ is small if we choose the parameters $\delta_p,\delta_s>0$ small and then let $\ve$ be small. The estimate relies on the same computations as before for the terms that are linear in $\phi_1$, $\phi_2$ in the right hand side of equations \eqref{sys-eq1-b2} and \eqref{sys-eq2-b2}. For the nonlinear terms, it is enough to have the following inequalities, whose proof is similar to Proposition~\ref{prop quadratic}: For $\phi = \eta_1\phi_1 + \eta_2\phi_2$, $\psi = \eta_1\psi_1 + \eta_2\psi_2$,
\begin{align*}
& \|\tilde \eta_1 Q(x,\phi,\nabla \phi,D^2\phi) -\tilde \eta_1 Q(x,\psi,\nabla \psi,D^2\psi) \|_{**,\Sigma}
\\
&  \qquad \leq C ( \|\phi_1\|_{*,\Sigma} + \|\phi_2\|_{*,\mathcal E_0} +  \|\psi_1\|_{*,\Sigma} + \|\psi_2\|_{*,\mathcal E_0}) (  \|\phi_1-\psi_1\|_{*,\Sigma} + \|\phi_2 - \psi_2\|_{*,\mathcal E_0}  )
\end{align*}
and there is a similar estimate for $\tilde \eta_2 Q$.
The Lipschitz dependence with small constant of $F$ on $\beta_i$, $\tau_i$ is proved using that, in each term in the right-hand side of \eqref{sys-eq1-b2} and \eqref{sys-eq2-b2}, either the dependence  on the parameters is Lipschitz with small constant or is quadratic (this is the case of  $\tilde E$).

By the contraction mapping principle, for $\ve>0$ small, $F$ has a unique fixed point in $\mathcal B$. This gives the desired solution.

\bigskip
\noindent
{\bf Acknowledgements.} { We would like to thank Sigurd Angenent, Nikolaos Kapouleas, Frank Pacard and Juncheng Wei   for useful conversations.  The first and second authors have been supported by grants Fondecyt 1130360 and 1150066, Fondo Basal CMM and by N\'ucleo Mienio CAPDE.  }

\end{document}